\documentclass{amcjoucc}

\usepackage{mathptmx}

\makeatletter\g@addto@macro\UrlSpecials{\do\/{\closerslashes}}\makeatother
     \def\closerslashes{\futurelet\Urlssnext\finishcloserslashes}
     \def\finishcloserslashes{%
       \mathchar8239  
       \nobreak\ifx\Urlssnext/\mskip-3mu \fi}
\makeatletter\g@addto@macro\UrlSpecials{\do\:{\nobreak\mathchar"603A\nobreak}}\makeatother
\newcommand{\AllowBreakBeforeURL}[1]{\hfil\penalty500\hfilneg\ \ \  \url{#1}}

\makeatletter

\def\enumerate{%
  \ifnum \@enumdepth >\thr@@\@toodeep\else
    \advance\@enumdepth\@ne
    \edef\@enumctr{enum\romannumeral\the\@enumdepth}%
      \expandafter
      \list
        \csname label\@enumctr\endcsname
        {\usecounter\@enumctr\def\makelabel##1{\hss\llap{##1}}%
	\setlength{\topsep}{2pt}
	 \ifnum \@enumdepth > 0 \setlength{\topsep}{-2pt} \fi 
	 \ifnum \@itemdepth > 0 \setlength{\topsep}{-2pt} \fi 
	\setlength{\partopsep}{0pt}
      	\advance\leftmargin12pt
	\@beginparpenalty=9999
          \itemsep=0pt}%
  \fi}

\def\itemize{%
  \ifnum \@itemdepth >\thr@@\@toodeep\else
    \advance\@itemdepth\@ne
    \edef\@itemitem{labelitem\romannumeral\the\@itemdepth}%
    \expandafter
    \list
      \csname\@itemitem\endcsname
      {\def\makelabel##1{\hss\llap{##1}%
	}%
	\setlength{\topsep}{2pt}
	 \ifnum \@itemdepth > 1 \setlength{\topsep}{-3pt} \fi 
	\setlength{\partopsep}{0pt}
      	\advance\leftmargin12pt
	\@beginparpenalty=9999
        \itemsep=0pt}%
  \fi}

\makeatother

\makeatletter
\renewenvironment{abstract}
{
\hrule height 0.25pt
\vskip 5pt
\noindent \textbf{Abstract}
\vskip 5pt
}
{
\vskip 5pt
\noindent \textit{\small Keywords:~\@keywords}

\vskip 3pt
\noindent \textit{\small Math.\ Subj.\ Class.: \@msc}
\vskip 5pt
\hrule height 0.25pt
}

\newcommand{\refereeheaders}{%
\renewcommand{\@evenhead}{\hfil \small \emph{Notes to aid the referee} \hfil\textit{\thepage}}
}

\usepackage{color,mparhack}
\newcommand\refnote[2][0]{\marginpar{\vskip-7pt\vskip#1pt\color{blue}[\raise0.5pt\hbox{\smaller note}\,\cref{#2}]}}
\reversemarginpar

\newtheoremstyle{plain}{\medskipamount}{\medskipamount}{\normalfont\it}{0pt}{\normalfont\bf}{.}{0.5em plus 0.1em}{}
\theoremstyle{plain}

\renewenvironment{proof}[1][\proofname]{\par
  \pushQED{\qed}%
  \normalfont \vskip-\lastskip\vskip \medskipamount\relax
       \noindent {\bfseries#1\@addpunct{.}\ \ \ignorespaces}%
}{%
  \popQED\endtrivlist\@endpefalse \goodbreak \vskip\medskipamount\relax
}

\renewcommand{\titledata}[2]{%
\global\def\p@pertitle{#1}%
\vskip 5mm
\begin{center}\LARGE\bfseries
\mathversion{bold}
#1\if!#2!\else
\global\advance\@footcount by 1%
\symbolfootnote[\@footcount]{#2}\fi\end{center}%
\vskip 1mm
}
\SetSymbolFont{operators}{bold}{OT1}{cmr}{b}{n}
\SetSymbolFont{letters}{bold}{OML}{cmm}{b}{it}

\renewcommand{\section}{%
\@startsection{section}{1}{0pt}{-\baselineskip}{0.5\baselineskip}{\large\bfseries\mathversion{bold}}%
}

\makeatother

\usepackage{amssymb,relsize}

\usepackage{graphicx}

\usepackage{hyperref}
\usepackage[capitalize]{cleveref}

\newcommand{\pref}[1]{\textup(\ref{#1}\textup)}
\newcommand{\csee}[1]{\textup(see \cref{#1}\textup)}

\newcommand{\fullccf}[2]{\textup(cf.\ \fullcref{#1}{#2}\textup)}
\newcommand{\fullcref}[2]{\cref{#1}\pref{#1-#2}}

\newcommand{\ZZ}{\mathbb{Z}}
\renewcommand{\natural}{\mathbb{N}}
\newcommand{\iso}{\cong}
\newcommand{\normal}{\triangleleft}

\newcommand{\sm}{\smallsetminus}

\newcommand{\quot}[1]{\overline{#1}}
\newcommand{\ul}[1]{\underline{#1}}
\newcommand{\voltage}{\mathop{\Pi}}
\newcommand{\connsum}{\mathbin{\#}}

\renewcommand{\pmod}[1]{\ (\mathop{\rm mod}\ #1)}

\newcommand{\bya}[1]{&\stackrel{\textstyle a}{\longrightarrow}&#1}
\newcommand{\byai}[1]{&\stackrel{\textstyle a^{-1}}{\longrightarrow}&#1}
\newcommand{\byb}[1]{&\stackrel{\textstyle b}{\longrightarrow}&#1}

\newcommand{\byc}[1]{&\stackrel{\textstyle c}{\longrightarrow}&#1}

\DeclareMathOperator{\Cay}{Cay}

\newtheorem{CASE}{Case}
\newtheorem{thm}[CASE]{Theorem}
\newtheorem{prop}[CASE]{Proposition}
\newtheorem{lem}[CASE]{Lemma}
\newtheorem{cor}[CASE]{Corollary}
\newtheorem{ResearchProblem}[CASE]{Research Problem}

\newtheorem{FGL}[CASE]{Lemma}

\crefformat{CASE}{Case~#2#1#3}
\crefformat{prop}{Proposition~#2#1#3}
\crefformat{thm}{Theorem~#2#1#3}
\crefformat{lem}{Lemma~#2#1#3}
\crefformat{cor}{Corollary~#2#1#3}
\crefformat{mainthm}{Theorem~#2#1#3}
\crefformat{FGL}{the Factor Group Lemma~#2\textup(#1\textup)#3}
\crefformat{MarusicMethod}{Maru\v si\v c's Method~#2(#1)#3}
\crefformat{MarusicMethodCor}{Maru\v si\v c's Method~#2(#1)#3}

\crefmultiformat{CASE}{Cases~#2#1#3}{ and~#2#1#3}{, #2#1#3}{, and~#2#1#3}

\crefname{CASE}{case}{cases}
\Crefname{CASE}{Case}{Cases}
\crefname{subcase}{subcase}{subcases}
\Crefname{subcase}{Subcase}{Subcases}
\crefname{thm}{theorem}{theorems}
\crefname{prop}{proposition}{propositions}
\crefname{lem}{lemma}{lemmas}
\Crefname{lem}{Lemma}{Lemmas}
\crefname{cor}{corollary}{corollaries}
\crefname{mainthm}{theorem}{theorems}
\crefname{FGL}{lemma}{lemmas}
\crefname{MarusicMethod}{method}{methods}
\crefname{MarusicMethodCor}{method}{methods}

\theoremstyle{definition}

\newtheorem{assump}[CASE]{Assumption}
\newtheorem{assumps}[CASE]{Assumptions}
\newtheorem{defn}[CASE]{Definition}
\newtheorem{rem}[CASE]{Remark}
\newtheorem{note}[CASE]{Note}
\newtheorem{notation}[CASE]{Notation}

\newtheorem*{outline}{Outline}

\theoremstyle{remark}

\crefformat{rem}{Remark~#2#1#3}
\crefformat{assump}{Assumption~#2#1#3}
\crefformat{notation}{Notation~#2#1#3}

\crefname{rem}{remark}{remarks}
\crefname{assump}{assumption}{assumptions}
\crefname{notation}{notation}{notations}

\crefformat{figure}{Figure~#2#1#3}
\crefname{figure}{figure}{figures}
\crefformat{section}{Section~#2#1#3}

\numberwithin{CASE}{section}
\numberwithin{equation}{CASE}

\newcounter{subcase}
 \newenvironment{subcase}[1][\unskip]{\refstepcounter{subcase}\bf
 \unskip\medskip \noindent \hskip\parindent Subcase \thesubcase\ #1. \it}{\unskip\upshape} 
\numberwithin{subcase}{CASE}
\renewcommand{\thesubcase}{\theCASE.\arabic{subcase}}
\crefname{subcase}{subcase}{subcases}
\Crefname{subcase}{Subcase}{Subcases}
\crefformat{subcase}{Subcase~#2#1#3}

\newcounter{subsubcase}
 \newenvironment{subsubcase}[1][\unskip]{\refstepcounter{subsubcase}\bf
 \unskip\medskip \noindent \hskip2\parindent Subsubcase \thesubsubcase\ #1. \it}{\unskip\upshape}
\numberwithin{subsubcase}{subcase}
\renewcommand{\thesubsubcase}{\thesubcase.\arabic{subsubcase}}
\crefname{subsubcase}{subsubcase}{subsubcases}
\Crefname{subsubcase}{Subsubcase}{Subsubcases}
\crefformat{subsubcase}{Subsubcase~#2#1#3}

\theoremstyle{definition}
\newtheorem{subsubassumps}[subsubcase]{\hskip2\parindent Assumptions}
\crefformat{subsubassumps}{Assumptions~#2#1#3}

\newenvironment{case}[1][\unskip]{\refstepcounter{case}%
 \em
 \medskip \noindent Case \thecase\ #1.\ }{\unskip\upshape}
 \newcommand{\thecase}{\arabic{case}}
 \newcounter{case}
\crefname{case}{case}{cases}

\crefformat{case}{Case~#2#1#3}

\newcommand{\MR}[1]{\href{http://www.ams.org/mathscinet-getitem?mr=#1}{\ MR\,#1}}

\theoremstyle{definition}
\newtheorem{aid}{}
\numberwithin{aid}{section}
\numberwithin{equation}{aid}
\newcommand{\oldendaid}{}
\let\oldendaid=\endaid
\renewcommand{\endaid}{\oldendaid\bigskip\vfill\hrule width\textwidth \bigskip\vfill\filbreak}

\numberwithin{equation}{CASE}

\newcounter{ProofCase}
 \newenvironment{ProofCase}[1][\unskip]{\refstepcounter{ProofCase}\it
 \unskip\medskip \indent \hskip\parindent Case \theProofCase\ #1. \it}{\unskip\upshape} 
 \renewcommand{\theProofCase}{\arabic{ProofCase}}
\crefformat{ProofCase}{Case~#2#1#3}
\crefname{ProofCase}{case}{cases}
\Crefname{ProofCase}{Case}{Cases}

\newcounter{ProofSubcase}
 \newenvironment{ProofSubcase}[1][\unskip]{\refstepcounter{ProofSubcase}\it
 \unskip\medskip \indent \hskip\parindent Subcase \theProofSubcase\ #1. \it}{\unskip\upshape} 
\numberwithin{ProofSubcase}{ProofCase}
\crefformat{ProofSubcase}{Subcase~#2#1#3}
\crefname{ProofSubcase}{subcase}{subcases}
\Crefname{ProofSubcase}{Subcase}{Subcases}

\renewcommand{\contentsname}{\normalfont\normalsize \hskip\parindent 
	The proof of \cref{G'=2p} is a lengthy case-by-case analysis, based on the choice of certain elements $a$ and~$b$ of the Cayley graph's connection set \csee{abDefn}. Here is an outline of the paper:} 
\makeatletter
\renewcommand{\l@section}{\@dottedtocline{1}{3.5em}{1em}}
\renewcommand{\l@subsection}{\@dottedtocline{2}{6em}{1.5em}}
\makeatother

\makeatletter
\renewenvironment{frontmatter}
{\thispagestyle{amctitle}%
\setcounter{page}{\@startpage}%
\vskip 10pt%
\centering}
{\vskip 20pt%
\blfootnote{\raggedright \ifnum\@authorcount=1\textit{E-mail address:}\else\textit{E-mail addresses:}\fi ~\@emails}%
}
\renewcommand{\ps@amctitle}{%
  \renewcommand\@oddhead{}
  \let\@evenhead\@oddhead
  \renewcommand\@evenfoot{\hfil}
  \let\@oddfoot\@evenfoot
}
\def\@oddrunninghead{Cayley graphs with commutator subgroup of order~$2p$ are hamiltonian}
\def\@evenrunninghead{Dave Witte Morris}
\makeatother

\begin{document}

\begin{frontmatter}

\titledata{Cayley graphs on groups with commutator 
\\ subgroup of order~$2p$ are hamiltonian}{}
{}

\authordata{Dave Witte Morris}
{Department of Mathematics and Computer Science,
University of Lethbridge, \\
Lethbridge, Alberta, T1K~3M4, Canada}
{Dave.Morris@uleth.ca, http://people.uleth.ca/$\sim$dave.morris/}
{}


\begin{abstract}
We show that if $G$ is a finite group whose commutator subgroup $[G,G]$ has order~$2p$, 
where $p$ is an odd prime, then every connected Cayley graph on~$G$ has a hamiltonian cycle.
\end{abstract}

\end{frontmatter}

\section{Introduction}

Let $G$ be a finite group.
It is easy to show that if $G$ is abelian (and $|G| > 2$), then every connected Cayley graph on~$G$ has a hamiltonian cycle. 
(See \cref{CayleyDefn} for the definition of the term \emph{Cayley graph}.)
To generalize this observation, one can try to prove the same conclusion for groups that are close to being abelian. Since a group is abelian precisely when its commutator subgroup is trivial, it is therefore natural to try to find a hamiltonian cycle when the commutator subgroup of~$G$ is close to being trivial. The following theorem, which was proved in a series of papers, is a well-known result along these lines.

\begin{thm}[Maru\v si\v c \cite{Marusic-HamCircCay}, Durnberger \cite{Durnberger-semiprod,Durnberger-prime}, 
1983--1985] \label{G'=p}
If the commutator subgroup $[G,G]$ of~$G$ has prime order, then every connected Cayley graph on~$G$ has a hamiltonian cycle.
\end{thm}

D.\,Maru\v si\v c (personal communication) suggested more than thirty years ago that it should be possible to replace the prime with a product $pq$ of two distinct primes:

\begin{ResearchProblem}[D.\,Maru\v si\v c, personal communication, 1985] \label{G'=pq}
Show that if the commutator subgroup of~$G$ has order~$pq$, where $p$ and~$q$ are two distinct primes, then every connected Cayley graph on~$G$ has a hamiltonian cycle.
\end{ResearchProblem}

This has recently been accomplished when $G$ is either nilpotent \cite{GhaderpourMorris-Nilpotent} or of odd order \cite{Morris-OddPQCommHam}. As another step toward the solution of this problem, we establish the special case where $q = 2$:

\begin{thm} \label{G'=2p}
If the commutator subgroup of~$G$ has order~$2p$, where $p$ is an odd prime, then every connected Cayley graph on~$G$ has a hamiltonian cycle.
\end{thm}

See the bibliography of \cite{M2Slovenian-LowOrder} for references to other results on hamiltonian cycles in Cayley graphs.

\setcounter{tocdepth}{1} 
\tableofcontents

\section{Some known results}

We recall a few results that provide hamiltonian cycles in various Cayley graphs.  

\begin{defn}[cf.\ {\cite[p.~34]{GodsilRoyle}}] \label{CayleyDefn}
For any subset~$S$ of a finite group~$G$, $\Cay(G;S)$ is the graph whose vertex set is~$G$, with an edge joining~$g$ to~$gs$, for each $g \in G$ and $s \in S$. This is called the \emph{Cayley graph} of  the connection set~$S$ on the group~$G$.
\end{defn}

\begin{rem}
Unlike most authors (including \cite{GodsilRoyle}), we do not require the connection set~$S$ to be symmetric in the definition of a Cayley graph; that is, we do not assume $S$ is closed under inverses. This does not change the set of graphs that are considered to be Cayley graphs, because, in our notation, $\Cay(G;S) = \Cay( G; S \cup S^{-1})$, where $S^{-1} = \{\, s^{-1} \mid s \in S \,\}$.
\end{rem}

\begin{thm}[\cite{M2Slovenian-LowOrder,CurranMorrisMorris-16p,GhaderpourMorris-27p,GhaderpourMorris-30p}] \label{|G|small}
Every connected Cayley graph on~$G$ has a hamiltonian cycle if\/ $|G| = k p$ for some prime~$p$ and some $k \in \natural$ with $1 \le k < 32$ and $k \neq 24$.
\end{thm}

\begin{notation} \ 
	\begin{itemize}
	
	\item The symbol~$G$ always represents a finite group.
	
	\item For $g \in G$ and $s_1,\ldots,s_n \in S \cup S^{-1}$, we use $[g](s_1,\ldots,s_n)$ to denote the walk in $\Cay(G;S)$ that visits (in order), the vertices
	$$ g, \, gs_1, \ gs_1 s_2, \, g s_1 s_2 s_3, \ \ldots, \, gs_1s_2\cdots s_n .$$
We may write $(s_1,\ldots,s_n)$ for $[e](s_1,\ldots,s_n)$. 

	\item We use $(s_1,\ldots,s_n)^k$ to denote the concatenation of~$k$ copies of the sequence $(s_i)_{i=1}^n$. 
	
	\item Appending $\#$ to a sequence deletes the last term; that is, $(s_i)_{i=1}^n\# = (s_i)_{i=1}^{n-1}$. 
	
	\item  If $W = [g](s_1,\ldots,s_n)$ is a walk in $\Cay(G;S)$, and $h \in G$, we use $hW$ to denote the translate $[hg](s_1,\ldots,s_n)$.
	
	\item When $C$ is an oriented cycle, we use $-C$ to denote the same cycle as~$C$, but with the opposite orientation.
	
	\item 
	For $g,h \in G$:
		$$ \text{$[g,h] = g^{-1} h^{-1} g h$,
		\quad
		$g^h = h^{-1} g h$,
		\quad and \quad
		${}^h \!g = h g h^{-1}$ ($= g^{h^{-1}}$)}
		. $$

	\item  We use $G'$ to denote the commutator subgroup $[G,G]$ of~$G$.

	\item For convenience, we let $\quot{G} = G/G'$.

	\item For $g \in G$, we let $\quot g = g G'$ be the image of~$g$ in~$\quot{G}$.

	\item  We use $Z(G)$ to denote the center of~$G$.
	
	\end{itemize}
\end{notation}

\begin{defn}[cf.\ {\cite[\S2.1.3, p.~61]{GrossTucker}}] \label{VoltageDefn}
Suppose
 	\begin{itemize}
	\item $N$ is an abelian, normal subgroup of~$G$,
	and
	\item $C = [N v] (s_i)_{i=1}^n$ is an (oriented) cycle in $\Cay(G/N;  S)$.
	\end{itemize}
The \emph{voltage} of~$C$ is ${}^v \! {\left( \prod_{i=1}^n s_i \right)}$. This is an element of~$N$, and it may be denoted $\voltage C$.
\end{defn}

We have the following straightforward observations:

\begin{lem} \label{BasicVoltage}
Assume the notation of \cref{VoltageDefn}. Then: \refnote{BasicVoltageRef}
	\begin{enumerate}
	\item \label{BasicVoltage-welldef}
	$\voltage C$ is determined by the oriented cycle~$C$: it is independent of the choice of the vertex~$N v$ of~$C$, and of the choice of the representative $v$ of~$N v$.
	\item \label{BasicVoltage-translate}
	$\voltage gC = {}^g \! {\left( \voltage C \right)}$ for all $g \in G$.
	\item \label{BasicVoltage-reverse}
	$\voltage(-C) = (\voltage C)^{-1}$.
	\end{enumerate}
\end{lem}

\begin{defn}
A subset $S$ of~$G$ is an \emph{irredundant} generating set of~$G$ if $S$ generates~$G$, but no proper subset of~$S$ generates~$G$.
\end{defn}

\begin{FGL}[``Factor Group Lemma'' {\cite[\S2.2]{WitteGallian-survey}}] \label{FGL}
Suppose
 \begin{itemize}
 \item $N$ is a cyclic, normal subgroup of~$G$,
 \item $(s_i)_{i=1}^m$ is a hamiltonian cycle in $\Cay(G/N;S)$,
 and
 \item the voltage $\voltage (s_i)_{i=1}^m$ generates~$N$.
 \end{itemize}
 Then $(s_1,s_2,\ldots,s_m)^{|N|}$ is a hamiltonian cycle in $\Cay(G;S)$.
 \end{FGL}

\begin{cor}[{}{\cite[Cor.~2.11]{M2Slovenian-LowOrder}}] \label{DoubleEdge}
  Suppose
 \begin{itemize}
 \item $N$ is a normal subgroup of~$G$, such that $|N|$ is prime,
 \item the image of $S$ in $G/N$ is an irredundant generating set of~$G/N$,
 \item there is a hamiltonian cycle in $\Cay(G/N;S)$,
 and
 \item $s \equiv t \pmod{N}$ for some $s,t \in S \cup S^{-1}$ with $s \neq t$.
 \end{itemize}
 Then there is a hamiltonian cycle in $\Cay(G;S)$.
 \end{cor}
 
 \begin{lem}[{}{\cite[Lem.~1 on p.~24]{ChenQuimpo}}] \label{ChenQuimpoEvenGrid}
 Let $P_k \mathbin\Box P_\ell$ be the Cartesian product of a path of length~$k$ with a path of length~$\ell$. If $k \ell$~is even, and $k,\ell \ge 2$, then $P_k \mathbin\Box P_\ell$ has a hamiltonian path from any corner vertex~$v$ to any vertex that is at odd distance from~$v$.
 \end{lem}
 
 \begin{cor} \label{ChenQuimpoOddEndpt}
 Suppose $N$~is a subgroup of an abelian group~$H$, and $\{x,y\} \cup S_0$ is a subset of~$H$ that generates~$H/N$. Let $k = |\langle x, N \rangle : N|$ and $ |\langle x, y, N \rangle : \langle x, N \rangle|$. If $k\ell$~is even, $k,\ell \ge 2$, $0 \le p < k$, $0 \le q < \ell$, and $p + q$ is odd, then $\Cay( H/N ; \{x,y\} \cup S_0)$ has a hamiltonian path $(s_i)_{i=1}^r$, such that $s_1s_2 \cdots s_r = x^p y^q$.
 \end{cor}

 \begin{proof}
If we identify the vertices of $P_k \mathbin\Box P_\ell$ with $\{\, (i,j) \mid 0 \le i < k, \ 0 \le j < \ell \,\}$ in the natural way, then the map $(i,j) \mapsto x^i y^j$ is an isomorphism from $P_k \mathbin\Box P_\ell$ to a subgraph~$X$ of $\Cay \bigl( \langle x,y \rangle;  x,y \bigr)$. So \cref{ChenQuimpoEvenGrid} provides a hamiltonian path $(t_i)_{i=1}^{k\ell-1}$ in~$X$ from~$e$ to $x^p y^q$.
So $t_1 t_2 \cdots t_{k\ell-1} = x^p y^q$.

Let $L = (u_j)_{j=1}^n$ be a hamiltonian path in $\Cay \bigl( H/ \langle x,y,N \rangle \bigr)$, and let 
	$$ (s_i)_{i=1}^r = (L, t_{2i-1}, L^{-1}, t_{2i})_{i=1}^{k\ell/2}\# .$$
From the definition of $k$ and~$\ell$, we see that the natural map from~$X$ to $\Cay \bigl( \langle x,y, N \rangle / N; x,y \bigr)$ is an isomorphism onto a spanning subgraph. Therefore, $(s_i)_{i=1}^r$ is a hamiltonian path in $\Cay( H/N ;S)$. Since $H$ is abelian, it is easy to see that $s_1s_2 \cdots s_r = x^p y^q$.\refnote{ChenQuimpoOddEndptvoltage}
 \end{proof}

Given a hamiltonian cycle~$C_0$ in $\Cay(\quot G ; S)$, the following result often provides a second hamiltonian cycle~$C_1$, such that the voltage of at least one of these two cycles generates~$G'$. (Then \cref{FGL} provides a hamiltonian cycle in $\Cay(G;S)$.)

\begin{lem}[cf.\ Maru\v si\v c \cite{Marusic-HamCircCay} and Durnberger \cite{Durnberger-semiprod}, or see {\cite[Lem.~3.1]{Morris-OddPQCommHam}}] \label{StandardAlteration}
Assume:
	\begin{itemize}
	\item $N$ is an abelian normal subgroup of~$G$, such that $G/N$ is abelian,
	\item $C_0$ is an oriented hamiltonian cycle in $\Cay(G/N;S)$,
	\item $s,t,u \in S^{\pm1}$ and $h \in G$, 
	\item $C_0$ contains:
			\begin{itemize}
		\item the oriented path $[\quot {h s^{-1} u^{-1}} ](s, t, s^{-1})$, 
		and
		\item either the oriented edge $[\quot h](t)$ or the oriented edge $[\quot{ht}](t^{-1})$. 
		\end{itemize}
	\end{itemize}
Then there is a hamiltonian cycle $C_1$ in $\Cay(G/N;S)$, such that \refnote[20]{StandardAlterationPf}
	$$ \left( \bigl( \voltage C_0 \bigr)^{-1} \bigl( \voltage C_1 \bigr) \right)^h 
	= \begin{cases}
	[ u, t^{-1}] \, [s, t^{-1}]^u
		&\text{if $C_0$ contains $[\quot h](t)$}, \\
	[t^{-1}, u] \, [s, t^{-1}]^u
		&\text{if $C_0$ contains $[\quot{ht}](t^{-1})$}
	. \end{cases}$$
Furthermore, $C_0$ and~$C_1$ have exactly the same oriented edges, except for some of the edges in the subgraph induced by $\{\quot h, \quot{hu^{-1}}, \quot{hs^{-1}u^{-1}}, \quot{ht},  \quot{htu^{-1}},  \quot{hts^{-1}u^{-1}} \}$.
\end{lem}

\begin{lem}[{}{\cite[Lem.~2.8]{Durnberger-semiprod}}] \label{Durnberger-commuting}
Assume 
	\begin{itemize}
	\item $S$ is an irredundant generating set of~$G$,
	\item $s,t \in S$, with $s \neq t$,
	\item $s$ commutes with~$t$,
	\item $\langle S \sm \{s\} \rangle \normal G$,
	and
	\item there is a hamiltonian cycle in $\Cay \bigl( \langle S \sm \{s\} \rangle;  S \sm \{s\}  \bigr)$.
	\end{itemize}
Then there is a hamiltonian cycle in $\Cay(G;S)$.
\end{lem}

We do not need the general theory of nilpotent groups, but we will make use of the following two facts.  (The first is essentially the definition of a nilpotent group, which can be found in any graduate-level textbook on group theory.)

\begin{lem}[{\cite[(iii) on p.~175 and Prop.~VI.1.h on page~176]{Schenkman}}] \ \label{G'inZ}
	\begin{enumerate}
	\item Every abelian group is nilpotent.
	\item If $G/Z(G)$ is nilpotent, then $G$ is nilpotent.
	\end{enumerate}
Therefore, if $G' \subseteq Z(G)$ \textup(in other words, if\/ $G/Z(G)$ is abelian\textup), then $G$ is nilpotent.
\end{lem}

\begin{thm}[{\cite{GhaderpourMorris-Nilpotent}}] \label{GhaderpourMorrisNilpotent}
If $G$ is a nontrivial, nilpotent, finite group, and the commutator subgroup of~$G$ is cyclic, then every connected Cayley graph on~$G$ has a hamiltonian cycle.
\end{thm}

The following observation is well known (and easy to prove). 

\begin{lem}[{}{\cite[Lem.~2.27]{M2Slovenian-LowOrder}}] \label{NormalEasy}
 Let $S$ generate a finite group~$G$ and let $s \in S$, such that $\langle s \rangle
\normal G$. If
 \begin{itemize}
 \item $\Cay \bigl( G/\langle s \rangle ; S \bigr)$ has a
hamiltonian cycle,
 and
 \item either
 \begin{enumerate}
 \item \label{NormalEasy-Z} 
 $s \in Z(G)$,
 or
 \item \label{NormalEasy-notZ} 
 $Z(G) \cap \langle s \rangle = \{e\}$,
 or
 \item \label{NormalEasy-p}
 $|s|$ is prime,
 \end{enumerate}
 \end{itemize}
 then $\Cay(G;S)$ has a hamiltonian cycle.
 \end{lem}

\begin{cor} \label{ScapG'}
Suppose
	\begin{itemize}
	\item $G'$ is cyclic of order~$pq$, where $p$ and~$q$ are distinct primes,
	\item $S$ is an irredundant generating set of~$G$,
	and
	\item some nontrivial element~$s$ of~$S$ is in~$G'$.
	\end{itemize}
Then $\Cay(G;S)$ has a hamiltonian cycle.
\end{cor}

\begin{proof}
We may assume $G' = \ZZ_p \times \ZZ_q$. Since every subgroup of a cyclic, normal subgroup is also normal, we know that $\langle s \rangle \normal G$. Also, there are hamiltonian cycles in $\Cay(G/\ZZ_p; S)$, $\Cay(G/\ZZ_q; S)$, and $\Cay(G/G'; S)$ (by \cref{G'=p} and the elementary fact that Cayley graphs on abelian groups have hamiltonian cycles). Hence, we may assume $\langle s \rangle = G'$ and $G' \cap Z(G) = \ZZ_q$ (perhaps after interchanging $p$ and~$q$), for otherwise \cref{NormalEasy} applies.

Let $\widehat{G} = G/\ZZ_p$. We may assume $|\widehat{G}| \neq 27$, for otherwise $|G| = 27p$ so \cref{|G|small} applies. Then, since $\widehat{G}$ is nilpotent \csee{G'inZ} and its commutator subgroup is~$\ZZ_q$, the proof in \cite[\S4]{KeatingWitte} implies there is a hamiltonian cycle $(t_i)_{i=1}^n$ in $\Cay \bigl( \widehat{G} / \widehat{G}' ; S')$ whose voltage%
\refnote{KWUsesFGL}
generates~$\widehat{G}'$.
 Then, since $\ZZ_p \cap Z(G) = \{e\}$, the proof of \fullcref{NormalEasy}{notZ} in \cite[Lem.~2.27(2)]{M2Slovenian-LowOrder} tells us that $( t_i, s^{p-1} )_{i=1}^n$ is a hamiltonian cycle in $\Cay \bigl( G/\ZZ_q ; S \bigr)$. 

Note that, since $\widehat{G}$ is a nilpotent group whose commutator subgroup is in the center and has prime order~$q$, the order of $|\widehat{G}/\widehat{G}'|$ must be a multiple of~$q$; that is, $n$~is a multiple of~$q$ (cf.~\cref{Cents->Homo} below).\refnote{GbarDivbyq} 
Calculating modulo~$\ZZ_p$, we have
	\begin{align*}
	\voltage( t_i, s^{p-1} )_{i=1}^n 
	&\equiv s^{(p-1)n} \, \voltage( t_i )_{i=1}^n
		&& \text{($\widehat{s} \in \widehat G' = \widehat{\ZZ_q} \subseteq Z(\widehat{G})$)}
	\\&\equiv \voltage( t_i )_{i=1}^n
		&& \text{($n$ is a multiple of~$q$)}
	\\&\not\equiv e
		&& \text{($\voltage( t_i )_{i=1}^n$ generates~~$\widehat{G}'$)}
	 . \end{align*}
Therefore $\voltage( t_i, s^{p-1} )_{i=1}^n$ generates $\ZZ_q$. So \cref{FGL} tells us that $\bigl( ( t_i, s^{p-1} )_{i=1}^n \bigr){}^q$ is a hamiltonian cycle in $\Cay(G;S)$.
\end{proof}

\section{Assumptions, group theory, and connected sums}

\begin{assumps}
The remainder of this paper provides a proof of \cref{G'=2p}, so 
	\begin{itemize}
	\item $p$ is an odd prime, 
	\item $G$ is a finite group whose commutator subgroup has order~$2p$, 
	and 
	\item $S$~is an irredundant generating set of~$G$.\refnote{MinEnough}
	\end{itemize}
We wish to show that the Cayley graph $\Cay(G;S)$ has a hamiltonian cycle.
\end{assumps}

\subsection{Basic group theory}

\begin{assump}
Because of \cref{ScapG'}, we may assume $S \cap G' = \emptyset$.
\end{assump}

\begin{notation} \label{abDefn}
The assumption that the commutator subgroup has order $2p$ implies that $G'$ is cyclic (cf.\ \cite[\S2E, proof of Cor.~1.4]{Morris-OddPQCommHam}), so we may write 
	$$G' = \ZZ_2 \times \ZZ_p .$$
From \cref{GhaderpourMorrisNilpotent}, we may assume $G$ is not nilpotent, so $G' \nsubseteq Z(G)$ \csee{G'inZ}.
This implies $\ZZ_p \cap Z(G) = \{e\}$. \refnote{ZpZ}
Hence there exists $a \in S$, such that 
 	\begin{align} \label{aDefn}
	 \text{$a$~does not centralize~$\ZZ_p$} 
	 . \end{align}
Then there exists $b \in S$, such that  \refnote[14]{CommutatorGenZp}
	\begin{align} \label{bDefn}
	\text{$\ZZ_p \subseteq \langle [a,b] \rangle$} 
	. \end{align}
The assumptions \pref{aDefn} and \pref{bDefn} are the basis of most of the arguments in the later sections of the paper.
\end{notation}

For ease of reference, we now collect a few well-known facts from group theory (specialized to our setting).

\begin{lem} \label{Z2inFrattini}
If $S_0 \subseteq G$, such that $\langle S_0, \ZZ_2 \rangle = G$, then $\langle S_0 \rangle = G$.
\end{lem}

\begin{proof}
Since $\ZZ_2 \subseteq Z(G)$, we have 
	$$\langle S_0 \rangle'  = \bigl\langle S_0, Z(G) \bigr\rangle'  \supseteq \langle S_0, \ZZ_2 \rangle' = G' .$$
Therefore 
	\begin{align*}
	\langle S_0 \rangle 
	&= \bigl\langle S_0 , \langle S_0 \rangle' \bigr\rangle
	= \langle S_0 , G' \rangle 
	\supseteq \langle S_0 , \ZZ_2 \rangle 
	= G
	. \qedhere \end{align*}
\end{proof}

\begin{cor} \label{Slessamin}
Suppose $S_0$ is a proper subset of~$S$, such that $\ZZ_p \subseteq \langle S_0 \rangle$. \textup(In particular, this will be the case if\/ $\{a,b\} \subseteq S_0$.\textup) Then $\langle \quot{S_0} \rangle \neq \quot G$.
\end{cor}

\begin{proof}
Suppose $\langle \quot{S_0} \rangle = \quot G$. This means $\langle S_0, G' \rangle = G$. Since $G' = \ZZ_2 \times \ZZ_p$ and $\ZZ_p \subseteq \langle S_0 \rangle $, this implies $\langle  S_0, \ZZ_2 \rangle = G$. So \cref{Z2inFrattini} tells us that $\langle  S_0 \rangle = G$. This contradicts the fact that the generating set~$S$ is irredundant.
\end{proof}

\begin{lem} \label{Cents->Homo}
Let $H$ be a group.\refnote{Cents->HomoPf}
 If $x,y,z \in H$, and $y$~centralizes~$H'$, then $[xy,z] = [x,z] \, [y,z]$.
Therefore $[y^k,z] = [y,z]^k$ for all $k \in \ZZ$.
\end{lem}

\begin{cor} \label{Divbyp} 
If $x,y \in G$, such that $y$ centralizes~$G'$, and $\ZZ_p \subseteq \langle [ x,  y] \rangle$, then $|y|$ is divisible by~$p$.\refnote[-7]{DivbypAid}
\end{cor}

\begin{cor} \label{Cent->Divides}
Let $S_0 \subseteq G$, such that $\ZZ_2 \nsubseteq \langle S_0 \rangle'$. 
If $g \in G$, such that $\ZZ_2 \subseteq \langle g, S_0\rangle'$, then $|\langle \quot{g} , \quot{S_0}\rangle:  \langle \quot{S_0}\rangle|$ is even.\refnote[-7]{Cent->DividesPf}
\end{cor}

In particular, if $\ZZ_2 \subseteq \langle [g,h] \rangle$, then, by taking $S_0 = \{h\}$, we see that $| \langle \quot g, \quot h \rangle : \langle \quot h \rangle|$ is even, so $|\quot g|$ is even (and, similarly, $|\quot h|$ must also be even).

\begin{cor} \label{DivBy4}
$|\quot G|$ is divisible by~$4$.
\end{cor}

\subsection{Connected sums}

 \begin{defn}[{}{\cite[Defn.~5.1]{GhaderpourMorris-Nilpotent}}]  \label{ConnectedSumDef}
Assume $C_1$ and~$C_2$ are two vertex-disjoint oriented cycles in $\Cay(\quot{G}; S)$, and let
	 $g \in G$,
	and
	$s,t \in S \cup S^{-1}$.
If 
		\begin{itemize}
		\item $C_1$ contains the oriented edge $[\quot{g}](t)$,
		and 
		\item $C_2$ contains the oriented edge $[\quot{gst}](t^{-1})$,
		\end{itemize}
	then we use $C_1 \connsum_t^s C_2$ to denote the oriented cycle obtained from $C_1 \cup C_2$ by
		\begin{itemize}
		\item removing the oriented edges $[\quot{g}](t)$ and $[\quot{gst}](t^{-1})$, and \refnote[8]{ConnSumAid}
		\item inserting the oriented edges $[\quot{g}](s)$ and $[\quot{gst}](s^{-1})$.
		\end{itemize}
	This is called the \emph{connected sum} of~$C_1$ and~$C_2$.
\end{defn}

If $[g](t)$ is any oriented edge of an oriented cycle~$C$, and $s \in S$, such that $sC$ is vertex disjoint from~$C$,  then we can form the connected sum $C \connsum_t^s -sC$. This construction can be iterated:

\begin{defn}
Suppose
	\begin{itemize}
	\item $[g_1](t_1), \ldots,[g_k](t_k)$ are oriented edges of an oriented cycle~$C$ in $\Cay(\quot G;S)$, such that $g_i \neq g_{i+1}$ for all~$i$,
	and
	\item $s_1,s_2,\ldots,s_k \in S \cup S^{-1}$, such that the cycles $C$, $s_1 C$, $s_2 s_1 C$,~\dots, $s_k s_{k-1} \cdots s_1 C$ are pairwise vertex-disjoint.
	\end{itemize}
Then we can form the connected sum 
	$$ C \ \connsum_{t_1}^{s_1} \ -s_1C \ \connsum_{t_2}^{s_2} \ s_2 s_1 C \ \connsum_{t_3}^{s_3} \  \cdots \ \connsum_{t_k}^{s_k} \  \pm s_k s_{k-1} \cdots s_1 C .$$
We call this a \emph{connected sum of signed translates of~$C$}.
\end{defn}

 \begin{lem}[{}{cf.\ \cite[Lem.~5.2]{GhaderpourMorris-Nilpotent}}] \label{VoltageOfConnSum}
If $C_1$, $C_2$, $g$, $s$, and~$t$ are as in \cref{ConnectedSumDef}, then
	$$ \voltage (C_1 \connsum_t^s C_2) = \voltage C_1 \cdot {}^g \! [s^{-1},t^{-1}] \cdot \voltage C_2 .$$
\end{lem}

\begin{proof}
We may assume $g = t^{-1}$ (or, in other words, $gt = e$), after translating the cycles by~$(gt)^{-1}$ \fullccf{BasicVoltage}{translate}. 
Write $C_1 = (s_i)_{i=1}^m$ and $C_2 = [st^{-1}](t_j)_{j=1}^n$, so
	$$ (C_1 \connsum_t^s C_2) = \bigl( (s_i)_{i=1}^{m-1}, s, (t_j)_{j=1}^{n-1}, s^{-1} \bigr) .$$
By assumption, $C_1$ contains the edge $\quot {t^{-1}} \to \quot{e}$ and $C_2$ contains the edge $\quot s \to \quot{st^{-1}}$, so $s_m = t$ and $t_n = t^{-1}$. 
Therefore
	\begin{align*}
	\voltage (C_1 \connsum_t^s C_2) 
	&= \prod_{i=1}^{m-1} (s_i) \cdot s \cdot \prod_{j=1}^{n-1} (t_j) \cdot s^{-1} 
	\\&= \prod_{i=1}^m (s_i) \cdot t^{-1} s \cdot \prod_{j=1}^n (t_j) \cdot t s^{-1} 
	\\& = \voltage C_1 \cdot t^{-1} s  \cdot \left( \voltage C_2 \right)^{st^{-1}} \cdot t s^{-1}
	\\& = \voltage C_1 \cdot t^{-1} s t s^{-1} \cdot \voltage C_2
	\\&= \voltage C_1 \cdot {}^{t^{-1}} \! [s^{-1},t^{-1}] \cdot \voltage C_2
	\\&= \voltage C_1 \cdot {}^{g} \! [s^{-1},t^{-1}] \cdot \voltage C_2
	. \qedhere \end{align*}
\end{proof}

\begin{cor} \label{VoltageDiff}
Assume that $C_1$, $C_2$, $g$, $s$, and~$t$ are as in \cref{ConnectedSumDef}. If $C_0$ is another oriented cycle that is vertex-disjoint from~$C_2$ and contains the oriented edge $\quot{g}(t)$, then
	$$ \bigl( \voltage (C_0 \connsum_t^s C_2) \bigr) \bigl( \voltage (C_1 \connsum_t^s C_2) \bigr)^{-1} = (\voltage C_0) (\voltage C_1)^{-1} .$$
\end{cor}

 \begin{cor}[{}{\cite[Lem.~5.2]{GhaderpourMorris-Nilpotent}}] \label{VoltageOfConnSumModZp}
If $C_1$, $C_2$, $g$, $s$, and~$t$ are as in \cref{ConnectedSumDef}, then
	$$ \voltage (C_1 \connsum_t^s C_2) \equiv \voltage C_1 \cdot \voltage C_2 \cdot [s,t] \pmod{\ZZ_p}  .$$
\end{cor}

The following result describes a fairly common situation in which the connected sum provides hamiltonian cycles in $\Cay(G;S)$:

\begin{lem} \label{UsualConnSum}
Let $S_0$ be a nonempty subset of~$S$, $g \in G$, $c \in S \sm S_0$, and $s,t \in S \sm \{c\}$.
Assume $C_0$ and~$C_1$ are oriented hamiltonian cycles in $\Cay \bigl( \langle \quot {S_0} \rangle ; S_0 \bigr)$, such that 
	\begin{itemize}
	\item $(\voltage C_0)^{-1} (\voltage C_1)$ is a nontrivial element of\/~$\ZZ_p$, 
	\item $C_0$ and $C_1$ both contain the oriented edge $[\quot g](s)$,
	\item for every $x \in S_0$, $C_0$ contains at least two edges that are labelled either $x$ or~$x^{-1}$, 
	\item $\ZZ_2 \subseteq \langle [c,t] \rangle$,
		and
	\item either $|\quot{G} : \langle \quot{S_0} \rangle| > 2$ or $s = t$.
	\end{itemize}
If either
	\begin{enumerate}
	\item \label{UsualConnSum-notZ2}
	there exists $u \in S \sm \{c\}$, such that $\ZZ_2 \nsubseteq \langle [u,c] \rangle$,
	or
	\item \label{UsualConnSum-even}
	$|\quot{G} : \langle \quot{S_0}, \quot t \rangle|$ is even,
	\end{enumerate}
then there is a hamiltonian cycle~$C$ in $\Cay(\quot G; S)$, such that $\langle \voltage C \rangle = G'$, so \cref{FGL} yields a hamiltonian cycle in $\Cay(G;S)$.
\end{lem}

\begin{proof}
Let $r = |\quot{G} : \langle \quot{S_0} \rangle|$. We have $\ZZ_p \subseteq \langle (\voltage C_0)^{-1} (\voltage C_1) \rangle \subseteq \langle S_0 \rangle$, so \cref{Slessamin} implies $r \neq 1$.

Suppose $r = 2$. By assumption, this implies $s = t$, which means that $C_0$ and~$C_1$ both contain the oriented edge $[\quot g](t)$. Then the translate $c \,C_0$ contains the oriented edge $[\quot {gc}](t)$. The connected sums $C = C_0 \connsum_t^{c} -c \,C_0$ and $C' = C_1 \connsum_t^{c} -c \,C_0$ are hamiltonian cycles in $\Cay(\quot{G}; S)$. From \cref{VoltageOfConnSumModZp}, we have 
	$$ \voltage C \equiv  \voltage C_0 \cdot \voltage C_0 \cdot  [c,t] \equiv [c,t] \not\equiv 0 \pmod{\ZZ_p}, $$
so $\voltage C$ projects nontrivially to~$\ZZ_2$. \Cref{VoltageDiff} says $(\voltage C)(\voltage C')^{-1} = (\voltage C_0)(\voltage C_1)^{-1}$, which generates~$\ZZ_p$ (because it is conjugate to the inverse of $(\voltage C_0)^{-1}(\voltage C_1)$, which is assumed to be a nontrivial element  of~$\ZZ_p$). Therefore, we see that either $\voltage C$ or $\voltage C'$ generates $G'$, as desired. So we may assume henceforth that $r > 2$.

We now show that we may assume $t \in S_0$. To this end, suppose it is not the case that $t \in S_0$. Let $n = |\langle \quot{S_0}, \quot t \rangle : \langle \quot{S_0} \rangle |$. Then, by choosing a sequence $\{[g_i](s_i)\}_{i=1}^{n-1}$ of oriented edges of~$C_0$, we can form a connected sum~$C_0'$ of signed translates of~$C_0$:
	$$ C_0' = C_0 \connsum^t_{s_1} -tC_0 \connsum^t_{s_2}  \cdots \connsum^t_{s_{n-1}} \pm t^{n-1} C_0 .$$
This is a hamiltonian cycle in $\Cay \bigl( \langle \quot{S_0}, \quot t \rangle; S_0 \cup \{t\} \bigr)$. 
We may assume $s_1 = s$. Then another hamiltonian cycle $C_1'$ can be constructed by replacing the leftmost occurrence of~$C_0$ with $C_1$, and \cref{VoltageOfConnSum} tells us that $(\voltage C_0') (\voltage C_1')^{-1} = (\voltage C_0) (\voltage C_1)^{-1}$, which is a nontrivial element of~$\ZZ_p$ (and $(\voltage C_0)^{-1} (\voltage C_1)$ is conjugate to the inverse of this). 
From the definition of connected sum, it is obvious that $C_0'$ contains at least two edges labelled $t^{\pm1}$. So the hamiltonian cycles $C_0'$ and $C_1'$ satisfy the hypotheses of the \lcnamecref{UsualConnSum} with $S_0 \cup \{t\}$ in the role of~$S_0$ and with $t$ in the role of~$s$.\refnote{UsualConnSumWitht}

\setcounter{ProofCase}{0}

\begin{ProofCase} \label{UsualConnSumnotZ2Case}
Assume there exists $u \in S \sm \{c\}$, such that $\ZZ_2 \nsubseteq \langle [u,c] \rangle$.
\end{ProofCase}

\begin{ProofSubcase} \label{UsualConnSumnotZ2Case-inS0}
Assume $u \in S_0$.
\end{ProofSubcase}
Fix a hamiltonian path $(s_i)_{i=1}^n$ in $\Cay(\quot{G}/\langle \quot{S_0} \rangle; S \sm S_0)$ with $s_1 = c$, and let $\pi_i = \prod_{j=1}^i s_j$. Any connected sum $C_0 \connsum_{t_1}^{s_1} (-\pi_1C_0) \connsum_{t_2}^{s_2} \cdots \connsum_{t_n}^{s_n} (\pm\pi_n C_0)$ is a hamiltonian cycle~$C$ in $\Cay(\quot{G}; S)$. 

Since $[t,c]$ and $[u,c]$ do not have the same projection to~$\ZZ_2$, the voltages of $C_0 \connsum_t^c -\pi_1C_0$ and $C_0 \connsum_u^c -\pi_1C_0$ do not have the same projection to~$\ZZ_2$. Therefore, by choosing $t_1$ to be the appropriate element of $\{t,u\}$, we may assume the projection of $\voltage C$ to~$\ZZ_2$ is nontrivial \csee{VoltageOfConnSumModZp}. Note also that if $|\quot{G} : \langle \quot{S_0} \rangle| = 2$, then we must have $t_1 = t$.\refnote{t1=t}

We may assume that $t_n = s$, and that the connected sum $(-1)^{n-1} \pi_{n-1}C_0 \connsum_s^{s_n} (-1)^n\pi_n C_0$ is relative to the oriented edge $[\quot {\pi_n g}](s)$ of $\pi_n C_0$ that is also in~$\pi_n C_1$.\refnote{Cans}
Therefore, another hamiltonian cycle $C'$ can be constructed by replacing $\pi_n C_0$ with $\pi_n C_1$ in the connected sum. Then \cref{VoltageOfConnSum} (together with \fullcref{BasicVoltage}{translate}) implies that $(\voltage C)^{-1} (\voltage C')$ is conjugate to $(\voltage C_0)^{-1} (\voltage C_1)$, which is a generator of~$\ZZ_p$. Therefore, either $\voltage C$ or $\voltage C'$ generates $G'$, as desired. 

\begin{ProofSubcase}
Assume $u \notin S_0$.
\end{ProofSubcase}
Let $S_u =  \{u\} \cup S_0$, let $n = |\langle \quot {S_u} \rangle : \langle \quot {S_0} \rangle| - 1$, let $(s_i)_{i=1}^m$ be a hamiltonian path in $\Cay \bigl( \quot{G}/\langle \quot{S_u} \rangle; S \sm S_u \bigr)$ with $s_1 = c$, and let $\pi_i = \prod_{j=1}^i s_j$. (Since $S \sm S_0$ is an irredundant generating set for $\quot G/\langle \quot {S_0} \rangle$, we have $m,n \ge 1$.) Any connected sum 
	$$C_u = C_0 \connsum_{t_1}^u -uC_0 \connsum_{t_2}^u \cdots \connsum_{t_{n}}^u \pm u^{n} C_0$$
 is a hamiltonian cycle in $\Cay \bigl( \langle \quot {S_u} \rangle ; S_u \bigr)$, so any connected sum 
	$$ C = C_u \connsum_{t_1'}^{s_1} -\pi_1C_u \connsum_{t_2'}^{s_2} \cdots \connsum_{t_m'}^{s_m} \pm\pi_m C_u$$
is a hamiltonian cycle in $\Cay(\quot{G}; S)$. 

Since $t \in S_0$, we know that $C_0$ contains more than one edge labeled~$t^{\pm1}$, so $-u C_0$ has an edge labeled~$t^{\pm1}$ that was not removed in the construction of the connected sum $C_0 \connsum_{t_1}^{u} -\pi_1C_0$. Furthermore, the definition of the connected sum implies that $C_0 \connsum_{t_1}^{u} -\pi_1C_0$ also contains an edge labeled~$u$. Therefore, we may form connected sums
	$$ \text{$C_u \connsum_{t^{\pm1}}^{c} -\pi_1C_u$ 
	\ and \  
	$C_u \connsum_u^{c} -\pi_1C_u$} $$
without removing any of the edges of~$C_u$.
Since $[c,t]$ and $[c,u]$ do not have the same projection to~$\ZZ_2$, the voltages of these two connected sums do not have the same projection to~$\ZZ_2$ \csee{VoltageOfConnSumModZp}. Therefore, by choosing $t_1'$ to be the appropriate element of $\{t^{\pm1}, u\}$, we may assume the projection of $\voltage C$ to~$\ZZ_2$ is nontrivial.

We have
	$$C = C_u \connsum_{t_1'}^{s_1} -\pi_1C_u \connsum_{t_2'}^{s_2} \cdots \connsum_{t_{m-1}'}^{s_{m-1}} \pm\pi_{m-1} C_u \connsum_{t_m'}^{s_m}\bigl(  \pm
	\pi_m C_0 \connsum_{t_1}^u \pm \pi_m uC_0 \connsum_{t_2}^u \cdots \connsum_{t_{n}}^u \pm \pi_m u^{n} C_0 \bigr) ,$$
so the proof can be completed almost exactly as in the final paragraph of \cref{UsualConnSumnotZ2Case-inS0} (by constructing another connected sum in which $\pi_m u^{n} C_0$ is replaced with $\pi_m u^{n} C_1$).

\begin{ProofCase} \label{AllucNontriv}
Assume $[u,c]$ projects nontrivially to~$\ZZ_2$, for every $u \in S \sm \{c\}$.
\end{ProofCase}
In particular, $[d,c]$ projects nontrivially to~$\ZZ_2$, for every $d \in S \sm \bigl( S_0 \cup \{c\} \bigr)$. Since we may assume that \cref{UsualConnSumnotZ2Case} does not apply with~$d$ in the place of~$c$, we conclude that we may assume\refnote[15]{udNontriv}
	\begin{align} \label{AllNontrivZ2}
	 \text{$[u,d]$ projects nontrivially to $\ZZ_2$, for all $d \in S \sm S_0$ and $u \in S \sm \{d\}$} 
	 . \end{align}
Choose a hamiltonian path $(s_i)_{i=1}^n$ in $\Cay(\quot{G}/\langle \quot{S_0} \rangle; S \sm S_0)$. Any connected sum 
	$$C = C_0 \connsum_{t_1}^{s_1} -\pi_1C_0 \connsum_{t_2}^{s_2} \cdots \connsum_{t_n}^{s_n} \pm\pi_n C_0$$
is a hamiltonian cycle in $\Cay(\quot{G}; S)$. Calculating modulo $\ZZ_p$, and letting $z$~be the nontrivial element of~$\ZZ_2$, we have
	\begin{align*}
	\voltage C
	&\equiv \voltage C_0 \cdot [s_1,t_1] \cdot \voltage (-\pi_1C_0) 
	\ \cdots \ 
	[s_n,t_n] \cdot \voltage (\pm\pi_n C_0)
		&& \text{(\cref{VoltageOfConnSumModZp})}
	\\&\equiv \voltage C_0 \cdot z \cdot \voltage C_0 \cdots z \cdot \voltage C_0
		&& \text{(\fullcref{BasicVoltage}{translate} and \pref{AllNontrivZ2})}
	\\&= (\voltage C_0)^{n+1} \cdot z^n 
	\\&\equiv z
		&& \text{($n$ is odd)}
	.\end{align*}

The proof is now completed exactly as in the final paragraph of \cref{UsualConnSumnotZ2Case-inS0}.
\end{proof}

\begin{cor} \label{UsualConnSumCor}
Let $S_0 \subseteq S$, $g \in G$, and $s \in S_0$.
Assume $C_0$ and~$C_1$ are oriented hamiltonian cycles in $\Cay \bigl( \langle \quot {S_0} \rangle ; S_0 \bigr)$, such that 
	\begin{itemize}
	\item $(\voltage C_0)^{-1} (\voltage C_1)$ is a nontrivial element of\/~$\ZZ_p$, 
	\item $C_0$ and $C_1$ both contain the oriented edge $[\quot g](s)$,
	\item for every $x \in S_0$, $C_0$ contains at least two edges that are labelled either $x$ or~$x^{-1}$, 
	and
	\item $\ZZ_2 \nsubseteq \langle S_0 \rangle'$.
	\end{itemize}
Then there is a hamiltonian cycle~$C$ in $\Cay(\quot G; S)$, such that $\langle \voltage C \rangle = G'$, so \cref{FGL} yields a hamiltonian cycle in $\Cay(G;S)$.
\end{cor}

\begin{proof}
We may assume $[c,t] \in \ZZ_p$, for all $c \in S$ and $t \in S_0$. 
(Otherwise, we see from \cref{Cent->Divides} that \fullcref{UsualConnSum}{even} applies.)
Choose $c,d \in S$, such that $[c,d] \notin \ZZ_p$, let $S_0^+ = S_0 \cup \{d\}$, and let $r = |\langle S_0^+ \rangle : \langle S_0 \rangle|$. Any connected sum of the following form is a hamiltonian cycle in $\Cay \bigl( \langle S_0^+ \rangle; S_0^+ \bigr)$:
	$$ C = C_0 \connsum_{s_1}^d -dC_0 \connsum_{s_2}^d \cdots \connsum_{s_{r-1}}^d \pm d^{r-1}C_0 .$$
We may assume $s_1 = s$, and that the connected sum $C_0 \connsum_{s_1}^d -dC_0$ is formed by using the oriented edge $[\quot g](s)$ that is also in~$C_1$. Therefore, a second hamiltonian cycle $C'$ can be constructed by replacing the leftmost occurrence of~$C_0$ with~$C_1$. Then \cref{Cent->Divides} implies that \fullcref{UsualConnSum}{even} applies (with $S_0^+$, $d$, $d$, $C$, and~$C'$ in the roles of $S_0$, $s$, $t$, $C_0$, and~$C_1$, respectively).
\end{proof}

\section{Case with \texorpdfstring{$\quot s = \quot t$}{sG' = tG'}}

\begin{CASE} \label{s=t}
Assume there exist $s,t \in S \cup S^{-1}$ with $\quot s = \quot t$ and $s \neq t$.
\end{CASE}

\begin{proof}
Write $t = s \gamma$ with $\gamma \in G'$. We may assume $\langle \gamma \rangle = G'$, for otherwise $|\gamma|$ is prime, so \cref{DoubleEdge} applies with $N = \langle \gamma \rangle$. Note that the irredundance of~$S$ implies $\langle S \sm \{s\} \rangle$ and $\langle S \sm \{t\} \rangle$ do not contain~$\ZZ_p$. This implies that every element of $S \sm \{s, t\}$ centralizes~$\ZZ_p$.\refnote{CentralizeOrNot} 
So $s$ and~$t$ do not centralize~$\ZZ_p$.\refnote{sNotCentralize}

Let $m = |\quot{t}|$ and $n = |\quot G|/m$. 

\begin{subcase}
Assume $|\quot t| > 2$.
\end{subcase}
Since $\quot G$ is abelian, it is easy to find a hamiltonian cycle $C = (t_i)_{i=1}^{mn}$ in $\Cay \bigl( \quot G ; S \sm \{s\} \bigr)$, such that $t_1 = t_2 = \cdots = t_{m-1} = t$.\refnote{Startt} Since $\langle \voltage C \rangle \subseteq \langle S \sm \{s\} \rangle$, and $\ZZ_p \nsubseteq \langle S \sm \{s\} \rangle$, we must have $\voltage C \in \ZZ_2$. 

For each subset~$I$ of $\{1,\ldots,m-1\}$, we define $C_I$ to be the hamiltonian cycle constructed from~$C$ by changing $t_i$ to~$s$ for all $i \in I$. The proof is completed by noting that $I$ may be chosen such that $\voltage C_I$ generates~$G'$, so \cref{FGL} applies:
	\begin{itemize}
	\item If $\voltage C = e$, let $I = \{1\}$.
	\item If $\voltage C$ is the nontrivial element of~$\ZZ_2$, and $t$ does not invert~$\ZZ_p$, then we may let $I = \{1,2\}$.
	\item If $\voltage C$ is the nontrivial element of~$\ZZ_2$, and $t$~inverts~$\ZZ_p$, then $|\quot t|$ is even, so we must have $|\quot t| \ge 4$. We may let $I = \{1,3\}$. 
	\end{itemize}

\begin{subcase}
Assume $|\quot t| = 2$.
\end{subcase}
(Since $t$ does not centralize $\ZZ_p$, this implies that $t$ inverts~$\ZZ_p$.)
%
%
Choose a hamiltonian cycle $(s_i)_{i=1}^n$ in $\Cay \bigl( \quot{G}/\langle \quot t \rangle ; S \sm \{s,t\}  \bigr)$, and let
	$$C_0 =  ( t,  s_i )_{i=1}^n = (t_j)_{j=1}^{2n} .$$
Since $n = |\quot G|/2$ is even \csee{DivBy4} and $S \sm \{s\}$ is an irredundant generating set\refnote{MinGenSet} of~$\quot G$, it is easy to see that $C_0$ is a hamiltonian cycle in $\Cay\bigl(\quot G ; S \sm \{s\} \bigr)$.\refnote{t=2HamCyc} Note that $t_i = t$ whenever $i$~is odd, and that $\voltage C_0 \in \ZZ_2$ (because $\ZZ_p \nsubseteq \langle S \sm \{s\} \rangle$).\refnote{VoltC0inZ2}

We may assume $n \ge 6$ (for otherwise $|G| = 4np \le 20p$, so \cref{|G|small} applies).
We construct a hamiltonian cycle $C_1$ from~$C_0$:
	\begin{itemize}
	\item If $\voltage C_0 = e$, construct $C_1$ by changing $t_1$ to~$s$.
	\item If $\voltage C_0 \neq e$, 
	construct $C_1$ by changing both~$t_1$ and~$t_5$ to~$s$.
	\end{itemize}
In each case, $\voltage C_1$ generates~$G'$. (To see this in the second case, note that $t_2 t_3 t_4 t_5 = s_1 t s_2 t$ centralizes~$G'$, because $t$~inverts~$G'$, and each $s_i$ centralizes~$G'$.) Therefore, \cref{FGL} applies.
\end{proof}

\section{Cases with \texorpdfstring{$|\quot a| > 2$ and $\quot b \notin \langle \quot a \rangle$}{|aG'| > 2 and bG' not in <aG'>}} \label{a>2+bNotinaSect}

Recall that the elements $a$ and~$b$ of~$S$ satisfy \pref{aDefn} and~\pref{bDefn}.

\begin{CASE} \label{generic(aS)=G}
Assume $|\quot a| > 2$, $\quot b \notin \langle \quot a  \rangle$, and there exists $c \in S$, such that $\ZZ_2 \subseteq \langle [a,c] \rangle$.
\textup(It may be the case that $b = c$.\textup)
\end{CASE}

\begin{proof}
Let $m = |\quot{a}|$ and $n = | \quot{G}: \langle \quot{a} \rangle|$. Since $\quot b, \quot c \notin \langle \quot a \rangle$ (and $\quot G / \langle \quot a \rangle$ is abelian), it is easy to find a hamiltonian cycle $(s_i)_{i=1}^n$ in $\Cay \bigl( \quot G / \langle \quot a \rangle; S \sm \{a\} \bigr)$, such that $s_n \in \{c^{\pm 1}\}$, and $s_k = b$ for some $k < n$. 
Since $\ZZ_2 \subseteq \langle [a,c] \rangle$, we know $m$ and~$n$ are both even \csee{Cent->Divides}. Since $n$~is even, we have the following (well-known) hamiltonian cycle~$C_0$
in $\Cay(\quot{G}; \quot{S})$:\refnote[21]{WellKnownCycle}
	\begin{align} \label{C0Even}
	 C_0 = \bigl( a, ( a^{m-2}, s_{2i-1}, a^{-(m-2)},  s_{2i} )_{i=1}^{n/2}\#, a^{-1}, (s_{n-j}^{-1})_{j=1} ^{n-1} \bigr) 
	 . \end{align}

Letting $\widehat G = G/\ZZ_p$, we have $\widehat G' = \ZZ_2$, so $\widehat a^{m-2} \in Z(\widehat G\,)$\refnote{amCent}
 (because $m$~is even). Therefore 
	$$a^{m-2} s_{2i-1} a^{-(m-2)} \equiv s_{2i-1} \quad \pmod{\ZZ_p} ,$$
so, calculating modulo~$\ZZ_p$, we have
	$$ \voltage C_0
	\equiv  a \cdot \left( \prod_{i=1}^{n-1}  s_j \right) \cdot  a^{-1} \cdot \left( \prod_{i=1}^{n-1}  s_j \right)^{-1} 
	\equiv a \cdot s_n^{-1} \cdot a^{-1} \cdot s_n
	= [ a^{-1} , s_n]
	= [ a^{-1} , c^{\pm1}]
	, $$
which is nontrivial (mod~$\ZZ_p$).

Recall that $s_k = b$.
Let $g = \prod_{i=1}^{k-1} s_i$ and $\delta = (-1)^{k+1}$. Then $C_0$ contains both the oriented edge $[\quot{gb}](b^{-1})$ and the oriented path $[\quot{ga^{-2\delta}}](a^\delta, b, a^{-\delta})$.
So \cref{StandardAlteration} (with $s = a^\delta$, $t = b$, $u = a^\delta$ and $h = g$)
provides a hamiltonian cycle~$C_1$, such that $(\voltage C_0)^{-1} (\voltage C_1)$ is conjugate to $[b^{-1}, a^\delta] [a^\delta, b^{-1}]^{a^\delta}$. Since $a$ centralizes~$\ZZ_2$, but not~$\ZZ_p$, this voltage is a generator of~$\ZZ_p$.\refnote{generic(aS)=GEndpt}

Thus, either $\voltage C_0$ or $\voltage C_1$ generates $\ZZ_2 \times \ZZ_p = G'$, so \cref{FGL} provides a hamiltonian cycle in $\Cay(G;S)$.
\end{proof}

\begin{CASE} \label{generic(aS)neqG}
Assume $|\quot a| > 2$, $\quot b \notin \langle \quot a \rangle$, and there does not exist $c \in S$, such that $\ZZ_2 \subseteq \langle [a,c] \rangle$.
\end{CASE}

\begin{proof} 
Choose $c,d \in S$ with $\ZZ_2 \subseteq \langle [c,d] \rangle$. 
Let 
	$$ \text{$m = | \quot a |$, \ $n = | \langle \quot S \sm \{\quot d\} \rangle | / m$, \ and \ $r = |\quot G|/(mn)$} .$$
By assumption, we know $a \notin \{c,d\}$. Also, we may assume $d \neq b$ (after interchanging $c$ and~$d$ if necessary). Then \cref{Slessamin} tells us $r > 1$.
Furthermore, we see from \cref{Cent->Divides} that the image of~$c$ in $\quot G/ \langle \quot a \rangle$ has even order,\refnote{cEven}
 so $n$~is even.

\begin{subcase} \label{generic(aS)neqG-n>2}
Assume $n > 2$.
\end{subcase} 
It is not difficult to construct a hamiltonian cycle $(s_i)_{i=1}^n$ in $\Cay \bigl( \langle \quot S \sm \{\quot d\} \rangle / \langle \quot a \rangle ; \quot S \sm \{\quot a , \quot d\} \bigr)$, such that $s_1 = b$ and $s_k = c^{\pm1}$\refnote{sk=c}
 for some $k \notin \{1,n\}$. Then, since $n$ is even, we may define $C_0$ as in \pref{C0Even}, so $C_0$ is a hamiltonian cycle in $\Cay \bigl( \langle \quot S \sm \{\quot d \} \rangle ; S \sm \{ d \} \bigr)$.

Let $g = s_1 s_2 \cdots s_k$, and note that $C_0$ contains the oriented edges $[\quot e](a)$ and $[\quot g](c^{\mp1})$. Since $\ZZ_2 \subseteq \langle [c,d] \rangle$, but $\ZZ_2 \nsubseteq \langle [a,d] \rangle$, we see from \cref{VoltageOfConnSum} that there is a connected sum
	$$ C = C_0 \connsum_{t_1}^d -dC_0 \connsum_{t_2}^d \cdots \connsum_{t_{r-1}}^d \pm d^{r-1}C_0 , $$
with $t_1 \in \{a,c^{\pm1}\}$, such that $\ZZ_2 \subseteq \langle \voltage C \rangle$. Note that $C$ is a hamiltonian cycle in $\Cay(\quot{G}; S)$.

The cycle $C_0$ contains both $[\quot b](b^{-1})$ and $[\quot a^{-2}](a,b,a^{-1})$, and neither of these paths contains either the edge $[\quot e](a)$ or the edge $[\quot g](c^{\mp1})$. Therefore, $C$ also contains both of these paths, so \cref{StandardAlteration} (with $s = a$, $t = b$, $u = a$, and $h = e$)
provides a hamiltonian cycle~$C'$ in $\Cay( \quot G; S)$, such that $\bigl( \voltage C \bigr)^{-1} \bigl( \voltage C ' \bigr)$ is a conjugate of $[b^{-1}, a] \, [a, b^{-1}]^a$, which  is a generator of~$\ZZ_p$ (since $a$ centralizes~$\ZZ_2$, but not~$\ZZ_p$).
Then either $\voltage C$ or $\voltage C'$ generates~$G'$, so \cref{FGL} applies.

\begin{subcase} \label{noc+n=2+r>2}
Assume $n = 2$ and $r > 2$.
\end{subcase}
Since $n = 2$ (and $\quot b \notin \langle \quot a \rangle$), we have $\langle \quot a, \quot b, \quot d \rangle = \quot G$, so \cref{Slessamin} implies $S = \{a,b,d\}$. (Therefore $b = c$, which means $\ZZ_2 \subseteq \langle [b,d] \rangle$.)
We have the following hamiltonian cycle in $\Cay \bigl( \langle \quot a , \quot b \rangle; \quot a, \quot b \,\bigr)$:
	$$ C_0 = [\quot e] ( a^{m-1}, b, a^{-(m-1)}, b^{-1} ) .$$
Using the oriented edge $[\quot e](a)$, we can form the connected sum $C_0 \connsum_a^d -d C_0$. Then, since $dC_0$ contains both $[\quot{db}](b^{-1})$ and $[\quot{dab}](a^{-1})$, we can extend this to a connected sum
	$$ C = C_0 \connsum_{a}^d -dC_0 \connsum_{t_2}^d \cdots \connsum_{t_{r-1}}^d \pm d^{r-1}C_0 , $$
with $t_2 \in \{a,b\}$, such that $\ZZ_2 \subseteq \langle \voltage C \rangle$ \csee{VoltageOfConnSumModZp}.
Since $C$ contains both $[\quot b](b^{-1})$ and $[\quot a^{-2}](a,b,a^{-1})$, we may argue as in the last paragraph 
of \cref{generic(aS)neqG-n>2}. Namely, \cref{StandardAlteration} (with $s = a$, $t = b$, $u = a$, and $h = e$) provides a hamiltonian cycle~$C'$ in $\Cay( \quot G; S)$, such that $\bigl( \voltage C \bigr)^{-1} \bigl( \voltage C ' \bigr)$ is a conjugate of $[b^{-1}, a] \, [a, b^{-1}]^a$, which  is a generator of~$\ZZ_p$.
Then either $\voltage C$ or $\voltage C'$ generates~$G'$, so \cref{FGL} applies.

\begin{subcase}
Assume $n = r = 2$.
\end{subcase}
As in \cref{{noc+n=2+r>2}}, we must have $S = \{a,b,d\}$ and $b = c$ (so $\ZZ_2 \subseteq \langle [b,d] \rangle$).

\begin{subsubcase}
Assume $m \neq 3$. 
\end{subsubcase}
Since $m = |\quot a| > 2$ (by an assumption of this \lcnamecref{generic(aS)neqG}), we have $m \ge 4$.
We have the following hamiltonian cycle in $\Cay( \quot G; S)$:\refnote[15]{n=r=2aid}
	$$ C_0 = \bigl( d, b, a, b^{-1}, d^{-1}, a^{m-2}, d, a^{-(m-3)}, b, a^{m-3}, d^{-1}, a^{-(m-1)}, b^{-1} ) .$$ 
Since $a$ is central in $G/\ZZ_p$ (by an assumption of this \lcnamecref{generic(aS)neqG}), we know that 
	$$ \voltage C_0 \equiv d b b^{-1} d^{-1} d b d^{-1} b^{-1}
	= d b d^{-1} b^{-1} = [d^{-1},b^{-1}] \equiv [d,b] = [d,c] \  \pmod{\ZZ_p} ,$$
so $\ZZ_2 \subseteq \langle \voltage C_0 \rangle$.

Note that $C_0$ contains both $[\quot{dab}](b^{-1})$ and $[\quot{d}\quot{a}^3](a^{-1}, b, a)$ (because $m \ge 4$), so applying \cref{StandardAlteration} (with $s = a^{-1}$, $t = b$, $u = a^{-1}$ and $h = da$) yields a hamiltonian cycle $C_1$ in $\Cay(G;S)$, such that $\bigl( \voltage C_0 \bigr)^{-1} \bigl( \voltage C_1 \bigr)$ is a conjugate of $[b^{-1}, a^{-1}] \, [a^{-1}, b^{-1}]^{a^{-1}}$, which  is a generator of~$\ZZ_p$.
Then either $\voltage C$ or $\voltage C'$ generates~$G'$, so \cref{FGL} applies.

\begin{subsubcase} \label{m=3notcentralize}
Assume $m = 3$ and $d$~does not centralize $G'$. 
\end{subsubcase}
%
Since the walk $(a^{-2}, b^{-1}, a^2)$ is a hamiltonian path in $\Cay \bigl(  \langle \quot a, \quot b \rangle; a,b \bigr)$, we have the following hamiltonian cycle in $\Cay( \quot G; S)$:
	$$ C = ( a^{-2}, b^{-1}, a^2,  d^{-1}, a^{-2}, b, a^2,  d) .$$
Note that
	$$ \voltage C
	= (a^{-2} b^{-1} a^2) \,  d^{-1} (a^{-2} b a^2) \,  d
 	= (b^{a^2})^{-1} \,  d^{-1} (b^{a^2}) \,  d
	= [b^{a^2}, d] .$$
Since $a^2$ does not invert~$G'$, we know that $b^{a^2} \not\equiv b^{a^{-2}} \pmod{\ZZ_2}$.\refnote{ba2}
Therefore, since $d$~does not centralize~$G'$, we may assume $[b^{a^2}, d] \not\equiv e \pmod{\ZZ_2}$\refnote{ba2dnote} 
(by replacing $a$ with its inverse if necessary). Also, since $G'$ is central modulo~$\ZZ_p$, we have $[b^{a^2}, d] \equiv [b,d] \not\equiv e \pmod{\ZZ_p}$. Therefore, $\voltage C$ generates~$G'$, so \cref{FGL} applies.

\begin{subsubcase}
Assume $m = 3$ and $d$ centralizes $G'$. 
\end{subsubcase}
%
Suppose $[b,d] \in \ZZ_2$.  Let $\widehat G = G/\ZZ_2$ and $\widehat H = \langle \widehat a, \widehat b \rangle$. From \cref{G'=p}, we know there is a hamiltonian cycle in $\Cay( \widehat H; a , b \bigr)$. Deleting an edge labeled $b^{\pm1}$ (and passing to the reverse and/or a translate if necessary) yields a hamiltonian path $L = (t^i)_{i=1}^{2mp-1}$ in $\Cay( \widehat H ; a , b \bigr)$ from~$\widehat e$ to~$\widehat b$. 
Let
	$$ C = (L^{-1}, d^{-1}, L, d) .$$
Then
	$$ \voltage C 
	= \bigl[ \,\prod\nolimits_{i=1}^{2mp-1} t_i, d]
	\in [ b \ZZ_2, d]
	= \{[b,d]\}
	, $$
because $\ZZ_2$ is in the center of~$G$. 
Since $[b,d] \in \ZZ_2$, this calculation implies that $C$ is a closed walk in $G/\ZZ_2 = \widehat G$. So $C$ is a hamiltonian cycle in $\Cay( \widehat G; S)$.
The calculation also implies that \cref{FGL} applies, because $\langle [b,d] \rangle = \ZZ_2$.

We may now assume $[b,d] \notin \ZZ_2$. Therefore, since $d$ centralizes~$G'$, and $p^2 \nmid 12 = |\quot G|$, we see from \cref{Cents->Homo} that $b$~does not centralize~$G'$.\refnote{relprime}
Also, we may assume $[a,d] \neq e$, for otherwise \cref{Durnberger-commuting} applies with $s = d$ and $t = a$. However, we know $\ZZ_2 \nsubseteq \langle [a,d] \rangle$ (by an assumption of this \lcnamecref{generic(aS)neqG}). Therefore $\langle [a,d] \rangle = \ZZ_p$. So \cref{m=3notcentralize} applies after interchanging $b$ and~$d$.
\end{proof}

\section{Cases with \texorpdfstring{$\quot b \in \langle \quot a \rangle$}{bG' in <aG'>}}

\begin{CASE} \label{binaNotinvert}
Assume $\quot b \in \langle \quot a \rangle$ and $a$~does not invert~$G'$.
\end{CASE}

\begin{proof}
Let $m = |\quot a|$.
We may assume (perhaps after replacing $b$ with its inverse) that we may write $b = a^k \gamma$ with $1 \le k \le m/2$ and $\gamma \in G'$. Assume $k \ge 2$, for otherwise \cref{{s=t}} applies.
This implies $m - 1 \ge k + 1$ (since $m = |\quot a|\ge 2k \ge k + 2$).

\begin{subcase} \label{binaNotinvert-c}
Assume there exists $c \in S$, such that $\ZZ_2 \subseteq \langle [a,c] \rangle$.
\end{subcase}
Let $n = | \quot{G}: \langle \quot{a} \rangle|$. 
Note that \cref{Cent->Divides} implies $m$ and~$n$ are even, and $c \notin \langle \quot a \rangle$ (so $c \neq b$).

Choose a hamiltonian cycle $(s_i)_{i=1}^n$ in $\Cay \bigl( \quot G / \langle \quot a \rangle ; S \sm \{a,b\} \bigr)$, such that $s_n = c$, and define $C_0$ as in \pref{C0Even}. Then $\langle \voltage C_0 \rangle$ contains~$\ZZ_2$ by the same calculation as in \cref{generic(aS)=G}.

Since $m - 1 \ge k + 1$, we may construct a hamiltonian cycle $C_1$ in $\Cay(\quot G ; S)$ by replacing the path $(a^{k+1})$ at the start of~$C_0$ with $(b, a^{-(k-1)} , b)$. Then
	$$ \bigl( \voltage C_1 \bigr)  \bigl( \voltage C_0 \bigr) ^{-1}
	=  b a^{-(k-1)} b a^{-(k+1)}
= (a^k \gamma) a^{-(k-1)} (a^k \gamma) a^{-(k+1)}
= a^{k+1} \gamma^a \, \gamma a^{-(k+1)}
. $$
This is a generator of~$\ZZ_p$, since $a$~inverts~$\ZZ_2$, but not~$\ZZ_p$. 
Hence, either $\voltage C_0$ or $\voltage C_1$ generates~$G'$, so \cref{FGL} provides a hamiltonian cycle in $\Cay(G;S)$.

\begin{subcase} \label{binaNotinvert-Noc}
Assume there does not exist $c \in S$, such that $\ZZ_2 \subseteq \langle [a,c] \rangle$.
\end{subcase}
Choose $c,d \in S$, such that $\ZZ_2 \subseteq \langle [c,d] \rangle$. (It is possible that $b \in \{c,d\}$, but we know, by the assumption of this \lcnamecref{binaNotinvert-Noc}, that $a \notin \{c,d\}$.)
Let $n = | \langle \quot a, \quot d \rangle: \langle \quot{a} \rangle|$ and $r = | \quot G|/(mn)$.
From \cref{Cent->Divides} (and the assumption of this \lcnamecref{binaNotinvert-Noc}),
we know $n$ and~$r$ are even.

We have the following hamiltonian cycle in $\Cay \bigl(  \langle \quot a, \quot d \rangle ; a,d \bigr)$:%
\refnote[16]{binaNotinvert-Noc-aid}
	$$ C_0 =  \bigl( (a, (a^{m-2}, d, a^{-(m-2)}, d)^{n/2}\#, a^{-1}, d^{-(n-1)} \bigr) .$$
As in the final paragraph of \cref{binaNotinvert-c}, another hamiltonian cycle $C_1$ can be constructed by replacing the path $(a^{k+1})$ at the start of~$C_0$ with $(b, a^{-(k-1)} , b)$, and the calculation in \cref{binaNotinvert-c} shows that $ ( \voltage C_1 ) (\voltage C_0 )^{-1}$ generates~$\ZZ_p$. Therefore, since $[c,d] \notin \ZZ_p$, but $[c,a] \in \ZZ_p$, we see that \fullcref{UsualConnSum}{notZ2} applies (with $S_0 = \{a,b,d\}$, $g = a^{-1}$, $s = t = d$, and $u = a$).
\end{proof}

\begin{CASE}
Assume $\quot b \in \langle \quot a \rangle$ and $a$~inverts~$G'$.
\end{CASE}

\begin{proof}
As in \cref{binaNotinvert}, we let $m = |\quot a|$ and write $b = a^k \gamma$ with $2 \le k \le m/2$ and $\gamma \in G'$.
We now consider the same five subcases as in \cite[pp.~60--62]{Durnberger-semiprod}.

\begin{subcase}
Assume $2 < k < m/2$ and $k$ is even.
\end{subcase}
Let $C_1 = (a^m)$. The proof in the last paragraph of \cite[p.~60]{Durnberger-semiprod} provides a hamiltonian cycle
	$$C_0 = \bigl( b, a^{-(k-4)}, b, a^{m-2k-2}, b, a^{-1}, b, a^2, b^{-2}, a^{k-3} \bigr) $$
in $\Cay \bigl( \langle \quot a \rangle ; a,b \bigr)$, such that $(\voltage C_0)^{-1} (\voltage C_1)$ is a generator of~$\ZZ_p$.\refnote{Durnberger60}
Therefore, \cref{UsualConnSumCor} applies (with $S_0 = \{a,b\}$), because $C_0$ and~$C_1$ both contain the oriented edge $[\quot a^{-1}](a)$.

\begin{subcase}
Assume $2 < k < m/2$ and $k$ is odd.
\end{subcase}
Let
	\begin{align*}
	 C_0 &= \bigl( (b,a,b^{-1},a)^{(k-1)/2}, b, a^{m-2k+1} \bigr) \\
\intertext{and}
	C_1 &= \bigl(  (b, a^{-1}, b^{-1}, a^{-1})^{(k-1)/2}, b^2, a^{m-2k-1} , b \bigr)
	. \end{align*}
Calculations in \cite[p.~61]{Durnberger-semiprod} show that $(\voltage C_0)^{-1} (\voltage C_1)$ is a generator of~$\ZZ_p$. Therefore, \cref{UsualConnSumCor} applies (with $S_0 = \{a,b\}$), because $C_0$ and~$C_1$ both contain the oriented edge $[\quot e](b)$.

\begin{subcase}
Assume $k = m/2$ and $k$ is even.
\end{subcase}
We follow the argument of \cite[Subcase~iii, p.~97]{KeatingWitte}. Since $k$~is even, we know $a^k$ centralizes~$G'$, so 
	$$b^2 = (a^k \gamma)^2 = a^{2k} \gamma^2 = a^{m} \gamma^2 \in \ZZ_2 \cdot \gamma^2 \not\ni e .$$
Therefore \cref{DoubleEdge} applies (with $s = b$ and $t = b^{-1}$).

\begin{subcase}
Assume $k = m/2$ and $k$ is odd.
\end{subcase}
Choose $c \in S$ so that $\ZZ_2 \subseteq \langle [a,c] \rangle$, if such $c$~exists. Otherwise, choose $c$ so that there exists $d \in S$, such that $\ZZ_2 \subseteq \langle [c,d] \rangle$. In either case,
\cref{Cent->Divides} implies $c \in S \sm \{a,b\}$, and $|\langle \quot a, \quot c \rangle : \langle \quot a \rangle|$ is even.

We may assume $b^2 = e$, for otherwise \cref{DoubleEdge} applies (with $s = b$ and $t = b^{-1}$). Therefore, noting that $a^k$ inverts~$G'$ (since $k$ is odd), we have
	$$ e = b^2  = (a^k \gamma)(a^k \gamma) = a^{2k} \cdot \gamma^{-1} \gamma = a^m  .$$

\begin{subsubcase}
Assume $|\quot G : \langle \quot a \rangle| > 2$.
\end{subsubcase}
It suffices to find a hamiltonian cycle $C_*$ in $\Cay( \quot G; S )$, such that $\voltage C_*$ projects nontrivially to~$\ZZ_2$, and $C_*$ contains the paths $[\quot{a^{k-3}}](a,b,a^{-1})$ and  $[\quot{a^{k-1} b}](b^{-1})$. For then \cref{StandardAlteration} (with $s = a$, $t = b$, $u = a$, and $h = a^{k-1}$) provides a hamiltonian cycle $C_*'$, such that $\langle (\voltage C_*)^{-1}(\voltage C_*') \rangle = \ZZ_p$. Therefore, either $\voltage C_*$ or $\voltage C_*'$ generates~$G'$, so \cref{FGL} applies.

Note that \refnote[14]{Ga>2}
	$$ C = ( a^{k-2}, b, a^{-(k-2)}, c, a^{k-1}, c^{-1}, b^{-1}, c, a^{-(k-1)}, c^{-1} ) $$
is a cycle through the vertices of $\Cay( \quot G; \{a,b,c\} )$ in $\langle \quot a \rangle \cup c \langle \quot a \rangle$.
 A connected sum of translates of~$C$ yields a hamiltonian cycle~$C_0$ in $\Cay( \quot G; S )$.

If $\ZZ_2 \nsubseteq \langle [ a,  c] \rangle$, then the connected sum defining~$C_0$ can be chosen so that $\ZZ_2 \subseteq \langle \voltage C_0 \rangle$ (see the proof of \cref{UsualConnSum}). So we may let $C_* = C_0$.

We may now assume $\ZZ_2 \subseteq \langle [ a,  c] \rangle$. Construct a hamiltonian cycle $C_1$ in $\Cay( \quot G; S )$ by replacing the rightmost translate of~$C$ in the connected sum with\refnote[15]{Ga>2replace}
	$$C' = (a^{k-1}, b, a^{-(k-1)}, c, a^{k-1},  b^{-1}, a^{-(k-1)}, c^{-1}) .$$
A straightforward calculation shows that $(\voltage C)^{-1}(\voltage C') \notin \ZZ_p$,\refnote{Ga>2voltage}
 so we have $\ZZ_2 \subseteq \langle \voltage C_i \rangle$ for some $i \in \{0,1\}$.
Let $C_* = C_i$.

\begin{subsubassumps}
We may now assume $|\quot G : \langle \quot a \rangle| = 2$, so the irredundance of~$S$ implies $S = \{a,b,c\}$. Since $\quot b \in \langle \quot a \rangle$, the irredundance of~$S$ also implies $\langle [a,c] \rangle = \ZZ_2$.\refnote{acisZ2}
Furthermore, we may also assume that $c$ either centralizes~$G'$ or inverts~$G'$. (Otherwise, a preceding case applies after interchanging $a$ with~$c$.)\refnote{acInterchange}
\end{subsubassumps}

\begin{subsubcase}
Assume $c$ inverts~$G'$. 
\end{subsubcase}
Let
	$$ \text{$L = 
	\begin{cases}
	(a,b)^k\#
	& \text{if $p \mid k$} \\
	 (b,a)^k\#
	& \text{if $p \nmid k$} 
	  \end{cases} $
	\qquad and \qquad
	$ C = (L^{-1}, c^{-1}, L, c )$.} $$
Then $L$ is a hamiltonian path in $\Cay \bigl( \langle \quot a \rangle; a,b \bigr)$,\refnote{abkham}
 so $C$ is a hamiltonian cycle in $\Cay(\quot G; S)$. Since $(ab)^k = \gamma^k$,\refnote{abk=gammak}
  we have\refnote[17]{voltage(abk)}
\makeatletter\@nobreaktrue\makeatother 
	$$ \voltage L = 
	\begin{cases}
	\gamma^k b^{-1} = \gamma^{k-1} a^{-k}
	& \text{if $p \mid k$} , \\
	\gamma^{-k} a^{-1}
	& \text{if $p \nmid k$} 
	 . \end{cases} $$
Thus, in either case, we have $\voltage L = \gamma^y a^z$, where $p \nmid y$ and $z$ is odd, so
	\begin{align*}
	 \voltage C &= (\voltage L)^{-1} c^{-1} (\voltage L) c = [\voltage L,c] = [\gamma^y a^z, c] 
	 \\&= [\gamma^y, c]^{a^z} \cdot [a^z, c] = (\gamma^{-2y})^{a^z} \cdot [a, c]^z = \gamma^{2y} \cdot [a,c] 
	 . \end{align*}
This generates $G'$, so \cref{FGL} applies.

\begin{subsubcase}
Assume $c$ centralizes~$G'$ and $k \ge 5$.
\end{subsubcase}
Let 
	$$C_0 = (L, c, L^{-1}, c^{-1} ) ,$$
where $L = (b,a)^k\#$. Since $C_0$ contains both $[\quot e](b,a,b)$ and $[\quot a \quot b \quot c] (a^{-1})$, and also contains both $[\quot a^2](b,a,b)$ and  $[\quot a^3 \quot b \quot c] (a^{-1})$ we can apply \cref{StandardAlteration} twice (first with $s = b$, $t = a$, $u = c$, and $h = bc$, and then with $s = b$, $t = a$, $u = c$, and $h = a^2 bc$), to obtain a hamiltonian cycle $C_2$, such that\refnote[15]{k>5voltage}
	$$ (\voltage C_0)^{-1}(\voltage C_2)
	= [a^{-1},b]^2 
	,$$
which generates~$\ZZ_p$.
Then, since \refnote[16]{k>5voltageC0}
	$$ \voltage C_0 
	= \bigl( (ba)^k a^{-1} \bigr)  c  \bigl( (ba)^k a^{-1} \bigr)^{-1} c^{-1}
	= [a,c] $$
is a generator of~$\ZZ_2$, we conclude that $\voltage C_2$ generates~$G'$, so \cref{FGL} applies.

\begin{subsubcase}
Assume $c$ centralizes~$G'$ and $k = 3$.
\end{subsubcase}
Assume, for the moment, that $\gamma \notin \ZZ_p$. Let\refnote[15]{k=3notZp}
	$$ C = (c, b, c^{-1}, a, b^{-1}, c, b, a, b^{-1}, c^{-1}, b, a ) .$$
Then $C$ is a hamiltonian cycle in $\Cay(\quot G; S)$, and a straightforward calculation shows that $ \voltage C = ba^3 = \gamma^{-1}$\refnote{voltage(k=3notZp)} generates~$G'$, so \cref{FGL} applies.

Now, suppose that $p \ge 5$, and, because of the preceding paragraph, 
that $\gamma \in \ZZ_p$. Let\refnote[13]{p>5aid}
		$$ C = (b, a, b^{-1}, a, b, c, a^{-5}, c^{-1} ) 
		 .$$
		Then $C$ is a hamiltonian cycle in $\Cay(\quot G; S)$ and\refnote[16]{p>5voltage}
			$$ \voltage C 
			= b a b^{-1} a b c a c^{-1} 
			= b a b^{-1} a b a [a,c^{-1}] 
			= \gamma^{-3} [a,c] . $$
		Therefore $\langle \voltage C \rangle = G'$ (since $p \neq 3$ and $\gamma$~projects trivially to~$\ZZ_2$), so \cref{FGL} applies.

We may now assume $p = 3$ (so $|G| = 72$), and that $\gamma \in \ZZ_p$. Let $\widehat G = G/\ZZ_p$. We have the following hamiltonian cycle in $\Cay(\widehat G;S)$: \refnote[15]{p=3aid}
	\begin{align*}
	C = (a^{2},c,a^{5},c^{-1},a^{-2},b,a^{2},c,a^{-5},c^{-1},a^{-2},b ) 
	. \end{align*}
Calculating modulo $\ZZ_2$ (so $c$ is in the center), we have
	$$ \voltage C
	= a^{2}ca^{5}c^{-1}a^{-2}ba^{2}ca^{-5}c^{-1}a^{-2}b
	\equiv a^{2}a^{5}a^{-2}ba^{2}a^{-5}a^{-2}b
	= a^{-1}bab
	= [a,b]
	= \gamma^2
	. $$
This is nontrivial (mod~$\ZZ_2$), so $\voltage C$ must be nontrivial. Therefore $\voltage C$ generates $\ZZ_p$, so \cref{FGL} applies.

\begin{subcase}
Assume $k = 2 < m/2$.
\end{subcase}

\begin{subsubcase}
Assume $|\quot{G} : \langle \quot a \rangle| > 2$.
\end{subsubcase}
Note that\refnote[15]{k=2aid}
	$$ C =  \bigl( b, a, b^{-1}, c, b, a^{-1}, b, c^{-1}, (a, c, a, c^{-1})^{(m-4)/2} \bigr) $$
is a cycle through the vertices of $\Cay( \quot G; \{a,b,c\} )$ in $\langle \quot a \rangle \cup c \langle \quot a \rangle$. A connected sum of translates of~$C$ yields a hamiltonian cycle~$C_0$ in $\Cay( \quot G; S )$. Since $k$ is even, we know that $\ZZ_2 \nsubseteq \langle [ b,  c] \rangle$, so it is easy to choose the connected sum in such a way that $\ZZ_2 \subseteq \langle \voltage C_0 \rangle$ (see the proof of \cref{UsualConnSum}). 

The cycle~$C$ contains the paths $[\quot e](b, a, b^{-1})$ and $[\quot b^2](a)$.
By taking just a bit of care in the creation of~$C_0$ (namely, not using any of these edges for the first connected sum), we may assume that $C_0$ also contains these paths.
Then \cref{StandardAlteration} (with $s = b$, $t = a$, $u = b$, and $h = b^2$) provides a hamiltonian cycle $C_1$, such that $(\voltage C_0)^{-1}(\voltage C_1) = [a,b]^2$ \refnote{k=2Ga>2voltage}
(because $b$ centralizes~$G'$).
 This is a generator of~$\ZZ_p$, so either $\voltage C_0$ or $\voltage C_1$ generates~$G'$. Therefore, \cref{FGL} applies.

\begin{subsubcase}
Assume $|\quot{G} : \langle \quot a \rangle| = 2$.
\end{subsubcase}
The irredundance of~$S$ implies that $S = \{a,b,c\}$ \csee{Slessamin}. We have the following hamiltonian cycle in $\Cay(\quot G; S)$:\refnote[15]{Ga=2}
	$$ C = ( b^2, a^{m-5}, c, a^{-(m-4)}, c^{-1}, b^{-1}, c, a, b^{-1}, c^{-1} ) .$$
Since $\quot b \in \langle \quot a \rangle$,
 the irredundance of~$S$ implies $\langle [a,c] \rangle = \ZZ_2$.\refnote{acisZ2}
 So $m$~is even \csee{Cent->Divides}. However, $\ZZ_2 \nsubseteq \langle [b,c] \rangle$, because $k = 2$ is even. So\refnote[15]{Ga=2modp}
	$$ \voltage C = b^2 (a^{m-5} c a^{-(m-4)} c^{-1}) (b^{-1} c a b^{-1} c^{-1})
	\equiv b^2 (a^{-1}) (b^{-2} a [a,c] )
	\equiv [a,c] \pmod{\ZZ_p}
	, $$
which generates~$\ZZ_2$.
We may also assume that $c$ either centralizes~$G'$ or inverts~$G'$ (for otherwise a preceding case applies after interchanging $a$ with~$c$).\refnote{acInterchange}
 Therefore\refnote[17]{Ga=2mod2}
	\begin{align*}
	 \voltage C 
	&= b^2 (a^{m-5} c a^{-(m-4)} c^{-1}) (b^{-1} c a b^{-1} c^{-1})
	\equiv a^4 \gamma^2 (a^{-1}) (\gamma^{-1} a^{-2} c a \gamma^{-1} a^{-2} c^{-1})
	\\&= \gamma^3 \cdot (\gamma^{-1})^{c}
	= \gamma^3 \cdot \gamma^{\pm1}
	\in \{\gamma^2, \gamma^4\} 
	\pmod{\ZZ_2}
	, \end{align*}
which generates~$\ZZ_p$.
We now know that $\voltage C$ projects nontrivially to both $\ZZ_2$ and~$\ZZ_p$, so it generates~$G'$. Therefore, \cref{FGL} applies.
\end{proof}

\section{Cases with \texorpdfstring{$|\quot a| = 2$ and $\#S = 2$}{|aG'| = 2 and |S| = 2}} \label{a=S=2}

\begin{assump}
In this \lcnamecref{a=S=2}, we assume 
	\begin{itemize}
	\item $|\quot a| = 2$, for all $a \in S$, such that $a$~does not centralize~$G'$,
	and
	\item $\#S = 2$.
	\end{itemize}
We may assume $|a| = 2$, for otherwise \cref{s=t} applies with $s = a$ and $t = a^{-1}$.

We may also assume that $b$ centralizes~$G'$, for otherwise we must have $|\quot b| = 2$, so $|G| = 8p$, so \cref{|G|small} applies.
Since $a$ does not centralize~$G'$, this implies $\quot a \notin \langle \quot b \rangle$.
Let 
	$$n = |\quot G : \langle \quot a \rangle| = |\quot G| / 2 = |\quot b| .$$
\end{assump}

\begin{CASE}
Assume $n \not\equiv 1 \pmod{p}$.
\end{CASE}

\begin{proof}
Let $C = (a^{-1}, b^{-(n-1)}, a, b^{n-1})$, so $C$ is a hamiltonian cycle in $\Cay(\quot G; S)$ with $\voltage C = [a, b^{n-1}] = [a,b]^{n-1}$, since $b$~centralizes~$G'$. Note that $n$~is even \csee{Cent->Divides}, and, by assumption, $n \not\equiv 1 \pmod{p}$. Therefore, $n-1$ is relatively prime to $2p$, so $\voltage C$ generates~$G'$, so \cref{FGL} applies.
\end{proof}

\begin{CASE}
Assume $n \equiv 1 \pmod{p}$.
\end{CASE}

\begin{proof}
We claim that $\ZZ_p \subseteq \langle b \rangle$. Suppose not. Then $|\langle b, \ZZ_2 \rangle| = 2n$. Since $\gcd(2n,p) = 1$, the abelian group $\langle b, G' \rangle$ has a unique subgroup of order~$2n$, so we conclude that $\langle b, \ZZ_2 \rangle$ is normal in~$G$. This implies that 
	$$ \langle a \rangle \langle b, \ZZ_2 \rangle = \langle a , b, \ZZ_2 \rangle \supseteq \langle a,b \rangle = G ,$$
so 
	$$|G| \le |a| \cdot |\langle b, \ZZ_2 \rangle| = 2 \cdot 2n = 4n .$$
This contradicts the fact that $|G| = 4np$.

\begin{subcase}
Assume $\ZZ_2 \subseteq \langle b \rangle$.
\end{subcase}
Combining this assumption with the above claim, we see that $G' \subseteq \langle b \rangle$. This implies $\langle b \rangle \normal G$, so $G = \langle a \rangle \ltimes \langle b \rangle$. Since $|a| = 2$, this implies that $\Cay(G; a,b)$ is a generalized Petersen graph. Then the main result of \cite{Bannai-HamCycGenPet} tells us that $\Cay(G; a,b)$ has a hamiltonian cycle.\refnote{BannaiThm}

\begin{subcase} \label{a=2&p=1}
Assume $\ZZ_2 \nsubseteq \langle b \rangle$.
\end{subcase}
Since $\langle b, G'\rangle$ is abelian, $\gcd(n, p) = 1$, and $\ZZ_2 \nsubseteq \langle b \rangle$, we may write 
	$$\langle b, G'\rangle = \ZZ_2 \times \ZZ_p \times \ZZ_n .$$
Then $G = \langle a \rangle \ltimes (\ZZ_2 \times \ZZ_p \times \ZZ_n)$, and we may assume $b = (0,1,1)$ and $[a,b] = (1,2,0)$. For $\ul G = G/\langle b^2 \rangle = G/(\ZZ_p \times 2\ZZ_n)$, it is straightforward to check that $\bigl( (a,b)^4\#, b^{-1} \bigr)$ is a hamiltonian cycle in $\Cay(\ul G; a,b)$ whose voltage is $(0,-2,2)$.\refnote{Z2notinbaid}
 (This hamiltonian cycle is taken from the final paragraph of Case~1 of the proof of \cite[Prop.~6.1]{CurranMorrisMorris-16p}.) This voltage generates $\ZZ_p \times 2\ZZ_n$ (since $\gcd(p,n) = 1$), so \cref{FGL} applies.
\end{proof}

\section{Cases with \texorpdfstring{$|\quot a| = 2$ and $\#S = 3$}{|aG'| = 2 and |S| = 3}} \label{a=2S=3}

\begin{assump} \label{a=2S=3assump}
In this \lcnamecref{a=2S=3}, we assume 
	$$S = \{a,b,c\} ,$$
and
	$$ \text{$|\quot s| = 2$, for all $s \in S$, such that $s$~does not centralize~$G'$.} $$
We also assume \cref{s=t} does not apply. (So $|s| = 2$.) In particular, we have $|a| = 2$.

Note that $\quot a \notin \langle \quot b \rangle$. (If $\quot a \in \langle \quot b \rangle$, then~$b$, like~$a$, does not centralize~$G'$, so our assumption implies $|\quot b| = 2$. Then $\quot a = \quot b$, contradicting the fact that \cref{s=t} does not apply.)
\end{assump}

\begin{notation}
Let
	$$ n = |\quot b| = |\langle \quot a, \quot b \rangle : \langle \quot a \rangle| \ge 2
	 \hbox{\quad and\quad}
	 \ell = |\quot G: \langle \quot a, \quot b \rangle| = |\quot G|/(2n) \ge 2 .$$
The last inequality is because the irredundance of~$S$ implies $\quot c \notin \langle \quot a, \quot b \rangle$ \csee{Slessamin}. 
\end{notation}

\begin{CASE} \label{a=2&b=3}
Assume $|\quot b| = 3$.
\end{CASE}
\begin{proof}
Since $|\quot b| \neq 2$, \Cref{a=2S=3assump} implies that $b$ centralizes~$G'$. Also, since $|\quot b|$ is odd, \cref{Cent->Divides} implies that $[a,b]$ and $[b,c]$ project trivially to~$\ZZ_2$, so $[a,c]$ must project nontrivially (and $\ell$~must be even). 
We have the following hamiltonian path in $\Cay \bigl( \quot G/\langle \quot a \rangle; S \bigr)$:\refnote[15]{a=2+b=3L}
	$$ L = (c^{\ell-1}, b, c^{-(\ell-1)}, b, c^{\ell-1}) .$$
Then $C = (L, a, L^{-1} , a)$ is a hamiltonian cycle in $\Cay(\quot G; S)$.
Since $\ell-1$ is odd, it is easy to see that $\ZZ_2 \subseteq \langle \voltage C \rangle$.\refnote{a=2+b=3LaLa}

Since $C$ contains both $[\quot c^{\ell-2}](c, b, c^{-1})$ and $[\quot c^{\ell-1} \quot a \quot b](b^{-1})$, \cref{StandardAlteration} (with $s = c$, $t = b$, $u = a$, and $h = c^{\ell-1} a$) provides a hamiltonian cycle $C'$, such that $(\voltage C)^{-1} (\voltage C')$ is conjugate to $[t^{-1}, u] \, [s,t^{-1}]^u = [b^{-1}, a] \,[c, b^{-1}]^a = [a,b] \, [c,b]$. This is an element of~$\ZZ_p$.
If it generates~$\ZZ_p$, then either $\voltage C$ or $\voltage C'$ generates~$G'$, so \cref{FGL} applies.

Thus, we may assume $[a,b] \, [c,b]$ is trivial. Since $\ZZ_p \subseteq \langle [a,b] \rangle$ (see~\pref{bDefn}), this implies that $[c,b]$ is nontrivial. So we may assume that $c$~does not centralize~$\ZZ_p$ (for otherwise replacing $c$ with~$c^{-1}$ would replace $[c,b]$ with $[c,b]^{-1}$, which would not cancel $[a,b]$). 

Now, \cref{a=2S=3assump} implies $| \quot c| = 2$, so we have the hamiltonian cycle\refnote[15]{a=2+b=3C0}
	$$ C_0 = (b^2, a, b^2, c, a,b,a,b, a, c ) ,$$
in $\Cay(\quot G; S)$. This contains both the path $[bac](a,b,a)$ and the edge $[b](b)$, so applying \cref{StandardAlteration} (with $s = a$, $t = b$, $u = c$, and $h = b$) provides a hamiltonian cycle~$C_1$, such that $\bigl( \voltage C_0 \bigr)^{-1} \bigl( \voltage C_1 \bigr)$ is conjugate to $[u,t^{-1}] \, [s,t^{-1}]^u = [c,b^{-1}] \, [a, b^{-1}]^c$. This is not equal to $[a,b] \, [c,b]$ (which is trivial), because $[a, b^{-1}]^c = [a,b]$, but $[c,b^{-1}] = [c,b]^{-1} \neq [c,b]$. So $\bigl( \voltage C_0 \bigr)^{-1} \bigl( \voltage C_1 \bigr)$ is nontrivial, and therefore generates~$\ZZ_p$. Since a straightforward calculation shows that $\ZZ_2$ is contained in $\langle \voltage C_0 \rangle$,\refnote{a=2+b=3Z2}
 this implies that  either $\voltage C_0$ or $\voltage C_1$ generates~$G'$, so \cref{FGL} applies. 
\end{proof}

\begin{CASE} \label{a=c=2}
Assume $\ell = 2$.
\end{CASE}

\begin{proof}
We may assume $|\quot b| \ge 4$, for otherwise either $|\quot b| = 2$, so \cref{|G|small} applies (because $|G| = 16p$), 
or $|\quot b| = 3$, so \cref{a=2&b=3} applies.
Let\refnote[15]{a=c=2L}
	$$ \text{$L = ( a,b,a, b^{n-2}, a, b^{-(n-3)} )$
	\qquad and \qquad
	$C = (L, c, L^{-1}, c^{-1} )$} ,$$
so $L$ is a hamiltonian path in $\Cay \bigl( \langle \quot a, \quot b \rangle ; a,b \bigr)$ and $C$ is a hamiltonian cycle in $\Cay(\quot G; S)$.

\begin{subcase}
Assume $[a,c]$ and $[a,b][b,c]$ are not both in~$\ZZ_p$.
\end{subcase}
A straightforward calculation (using \cref{Cents->Homo}) shows that $\voltage C \equiv [a,c] \pmod{\ZZ_p}$.\refnote{a=c=2Cv}
 If this is in~$\ZZ_p$, then, by assumption, $[a,b][b,c] \notin\ZZ_p$, so applying \cref{StandardAlteration} to the paths $[\quot e](a,b,a)$ and $[\quot{abc}](b^{-1})$ in~$C$ (so $s = a$, $t = b$, $u = c$, and $h = ac$)  yields a hamiltonian cycle~$C'$, such that $\voltage C'$ projects nontrivially to~$\ZZ_2$.\refnote{a=c=2C'Z2}
 Therefore, we have a hamiltonian cycle (either~$C$ or~$C'$) whose voltage is not in~$\ZZ_p$.

Now, since $|\quot b| \ge 4$, we know that $C$ (and also~$C'$) contains the paths $[\quot{b^{-2}ac}] (b, a, b^{-1})$ and $[\quot{ac}](a)$. Furthermore, we know that $[b,a] [b,a]^b$ is a nontrivial element of~$\ZZ_p$ (because $b$ does not invert~$[a,b]$). Therefore, \cref{StandardAlteration} (with $s = b$, $t = a$, $u = b$, and $h = ac$) yields a hamiltonian cycle~$C_1$ (or~$C_1'$) whose voltage generates~$G'$, so \cref{FGL} applies.

\begin{subcase}
Assume $[a,c]$ and $[a,b][b,c]$ are both in~$\ZZ_p$.
\end{subcase}
Since $[a,c]$, $[a,b]$, and $[b,c]$ generate~$G'$, they cannot all be in~$\ZZ_p$, so this assumption implies that neither $[a,b]$ nor $[b,c]$ is in~$\ZZ_p$.
Also, we may assume $\langle [a,c] \rangle = \ZZ_p$, for otherwise $[a,c] = e$, so we could apply \cref{Durnberger-commuting} with $s = c$.

We have the following hamiltonian cycle in $\Cay( \quot G; S )$:\refnote[15]{bothinC0}
	$$ C_0 = ( b^{n-1}, c, b^{-(n-2)}, a, b^{n-2}, c^{-1}, b^{-(n-1)}, c, a, c^{-1} ) .$$
Then
	\begin{align*}
	\voltage C_0 
	&= b^{n-1} c \bigl( b^{-(n-2)} a b^{n-2} \bigr) c^{-1} b^{-(n-1)} c  a c^{-1}
	\\&=  b^{n-1} c \bigl( a [a,b]^{n-2} \bigr) c^{-1} b^{-(n-1)} c  a c^{-1}
	\\&= ([a,b]^{-(n-2)})^c \cdot b^{n-1} (c  a  c^{-1}) b^{-(n-1)} (c  a c^{-1})
	\\&= ([a,b]^{-(n-2)})^c \cdot [b, c  a c^{-1}]^{-(n-1)}
	\\&= ([a,b]^{-(n-2)})^c \cdot [b, a]^{-(n-1)} && \text{($cac^{-1} \in aG'$ and $G' \subseteq C_G(b)$)}
	\\&= ([a,b]^{-(n-2)})^c \cdot [a,b]^{n-1}
	.\end{align*}
If $c$ centralizes $\ZZ_p$, then $\voltage C_0 = [a,b]$ generates~$G'$, so \cref{FGL} applies.

We may now assume $c$ does not centralize~$\ZZ_p$. Then \cref{a=2S=3assump} tells us that $c$ inverts~$\ZZ_p$, so $\voltage C_0 = [a,b]^{2n-3}$ (and $|c| = 2$).
Hence, we may assume $2n \equiv 3 \pmod{p}$, for otherwise $\voltage C_0$ generates~$G'$, so \cref{FGL} applies.
We now consider the following hamiltonian cycle in $\Cay( \quot G; S )$:\refnote[15]{bothinCstar}
	$$ C_* = ( b^{n-3}, c, b^{-(n-4)}, a, b^{n-4}, c^{-1}, b^{-(n-3)}, c, (b^{-1}, c)^2, a, (c,b)^2, c^{-1}   ) .$$
We have
	\begin{align*}
	\voltage C_* 
	&= b^{n-3} c \bigl( b^{-(n-4)} a b^{n-4} \bigr) c^{-1} b^{-(n-3)} c \bigl( (b^{-1}c)^2 a (cb)^2 \bigr) c^{-1}
	.\end{align*}
Since $cb$ inverts~$G'$, we know that $(b^{-1}c)^2 a (cb)^2 = a$,\refnote{acb=a}
so $\voltage C_*$ is exactly the same as the voltage of~$C_0$, but with $n$ replaced by~$n-2$; that is, 
	$$\voltage C_* =  [a,b]^{2(n-2)-3} = [a,b]^{2n-7}. $$
Since $2n \equiv 3 \pmod{p}$, we have 
	$$ 2n - 7 \equiv 3 - 7 = -4 \not\equiv 0 \pmod{p} ,$$
so $\voltage C_*$ generates~$G'$, so \cref{FGL} applies.
\end{proof}

\begin{CASE} \label{bNot3&lnot2}
Assume $|\quot b| \neq 3$ and $\ell \neq 2$.
\end{CASE}

\begin{proof}
Since $\ell \neq 2$, we know $|\quot c| > 2$, so $c$ must centralize~$G'$ (by \cref{a=2S=3assump}). Also, \cref{Cent->Divides} implies that $|\quot b|$ and $\ell$ cannot both be odd.
	\begin{itemize}

	\item If $|\quot b|$ is odd (so $\ell$ is even), let\refnote[15]{bNot3+lnot2Lodd}
		\begin{align*}
		L &= \bigl( c^{\ell-1}, b, c^{-1}, b, c, b, (b^{n-4}, c^{-1}, b^{-(n-4)}, c^{-1})^{\ell/2}\#, 
		b^{-1}, c^{\ell-3}, b^{-1} , c^{-(\ell-3)} \bigr) 
		. \end{align*}

	\goodbreak 
	\item If $|\quot b|$ is even, let\refnote[15]{bNot3+lnot2Leven}
		\begin{align*}
		L &= \bigl( c^{\ell-1}, b^{n-1}, c^{-1}, (c^{-(\ell-2)}, b^{-1}, c^{\ell-2}, b^{-1})^{(n-2)/2}, c^{-(\ell-2)} \bigr)
		. \end{align*}
	\end{itemize}
In either case, $L$ is a hamiltonian path in $\Cay \bigl( \quot G / \langle \quot a \rangle ; \{b,c\} \bigr)$ from $\quot e$ to~$\quot b$. Now, let
	$$ C = (L, a, L^{-1}, a) 
	\quad \text{and} \quad
			 (g,\epsilon) = 
			\begin{cases}
			(c^{\ell-1}, -1) & \text{if $|\quot b| = 2$ or $|\quot b|$ is odd}, \\
			(ab^2, 1) & \text{if $|\quot b| > 2$ and $|\quot b|$ is even}
			, \end{cases} $$
so $C$ is a hamiltonian cycle in $\Cay(\quot G; S)$ that contains the paths
	$$ [\quot {bc}](c^{-1},a,c), 
	\quad
	[\quot{ca}](c^{-1}, a, c),
	\quad
	[\quot g](b),
	\text{\quad and\quad}
	[\quot{gbac^\epsilon}](c^{-\epsilon}, b^{-1}, c^{\epsilon}) .$$ 
Note that $[\quot {bc}](c^{-1},a,c)$ contains $[\quot b](a)$ and that $[\quot{ca}](c^{-1}, a, c)$ contains $[\quot a](a)$. Also note that all of these paths are vertex-disjoint (except for the vertices $\quot{ac}$ and $\{abc\}$ when $|\quot b| = 2$ and $\ell = 3$).
We introduce some terminology:
	\begin{itemize}
	\item Applying \cref{StandardAlteration} to the oriented paths $[\quot{ca}](c^{-1}, a, c)$ and $[\quot b](a)$ (so $s = c^{-1}$, $t = a$, $u = b$, and $h = ab$) will be called the ``$a$-transform\rlap.'' This multiplies the voltage by $\gamma_a$, where 
	$\gamma_a 
	= [a,b^{-1}] [c,a]$.\refnote{atransform}
	\item Applying \cref{StandardAlteration} to the oriented paths $[\quot g](b)$ and $[\quot{gbac^{\epsilon}}](c^{-\epsilon}, b^{-1}, c^{\epsilon})$ (so $s = c^{-\epsilon}$, $t = b^{-1}$, $u = a$, and $h = gb$) will be called the ``$b$-transform\rlap.'' This multiplies the voltage by a conjugate of $\gamma_b$, where $\gamma_b = [b,a][b,c^{-\epsilon}]$.\refnote{btransform}
	\end{itemize}

\begin{subcase} \label{PreciselyOne}
Assume precisely one of $\gamma_a$ and~$\gamma_b$ is in~$\ZZ_p$.
\end{subcase}
Write $\{a,b\} = \{x,y\}$, such that $\gamma_x \in \ZZ_p$ and $\gamma_y \notin \ZZ_p$.
We may assume $\langle \gamma_x \rangle = \ZZ_p$ (by replacing $c$ with its inverse, if necessary).
Choose $C'$ to be either $C$ or the $y$-transform of~$C$, such that $\voltage C'$ projects nontrivially to~$\ZZ_2$. 
Then choose $C''$ to be either $C'$ or the $x$-transform of~$C'$, such that $\voltage C''$ generates~$G'$, so \cref{FGL} applies.

\begin{subcase}
Assume $\gamma_a$ and~$\gamma_b$ are both in~$\ZZ_p$.
\end{subcase}
Since $[a,b]$, $[a,c]$, and $[b,c]$ cannot all be in~$\ZZ_p$, this assumption implies that none of them are in~$\ZZ_p$. Therefore, since the path~$L$ has odd length, we see that $\voltage C$ has nontrivial projection to~$\ZZ_2$.\refnote{bothZ2}

We may assume (by replacing $c$ with its inverse, if necessary), that $\gamma_a$ has nontrivial projection to~$\ZZ_p$, so $\langle \gamma_a \rangle = \ZZ_p$. Therefore, by choosing $C'$ to be either $C$ or the $a$-transform of~$C$, such that $\voltage C'$ generates~$G'$, we may apply \cref{FGL}.

\begin{subcase} \label{gammabcnontrivZ2}
Assume neither $\gamma_a$ nor~$\gamma_b$ is in~$\ZZ_p$, and $b$~centralizes~$G'$.
\end{subcase}
%
Note that the sum of the exponents of the occurrences of~$b$ in~$L$ is~$1$, and the sum of the exponents of the occurrences of~$c$ is~$0$.
Therefore, since $b$ and~$c$ centralize~$G'$, \cref{Cents->Homo} implies that $\voltage C = [a,b]$.\refnote{neitherC}
 Hence, we may assume $[a,b] \in \ZZ_p$ (for otherwise $\langle \voltage C \rangle = G'$, so \cref{FGL} applies).
Then, by the assumption of this \lcnamecref{gammabcnontrivZ2}, we conclude that $[a,c] \notin \ZZ_p$. So we may assume $\langle [a,c] \rangle = \ZZ_2$, for otherwise $b$ and~$c$ could be interchanged, resulting in a situation in which $[a,b] \notin \ZZ_p$, and which has therefore already been covered.
Also, since $[a,b] \in \ZZ_p$ and $[a,c] \notin \ZZ_p$, \cref{Cent->Divides} tells us that $\ell$ is even (and recall that $\ell \neq 2$).

Since $[a,b]$ is a nontrivial element of~$\ZZ_p$, and $b$~centralizes~$G'$, we see from \cref{Divbyp} that $|b|$ is divisible by~$p$. Therefore, $|b| \neq 2$, so we may assume $|\quot b| > 2$ (for otherwise \cref{s=t} applies with $s = b$ and $t = b^{-1}$). Since $|\quot b| \neq 3$ (by the assumption of this \lcnamecref{bNot3&lnot2}), this implies $n = |\quot b| \ge 4$, so we may let\refnote{L0}
	$$L_0 = \bigl( c^{\ell-1}, b, c^{-(\ell-1)}, b^2, (b^{n-4}, c, b^{-(n-4)}, c)^{\ell/2}\#, b^{-1}, c^{-(\ell-2)} \bigr) ,$$
so $L_0$ is a hamiltonian path  from $\quot e$ to~$\quot{b^2 c}$ in $\Cay \bigl( \quot G/\langle \quot a \rangle; \{b,c\} \bigr)$.
Note that the sum of the exponents of the occurrences of~$b$ in~$L$ is~$2$, and the sum of the exponents of the occurrences of~$c$ is~$1$.
Therefore, since $b$ and~$c$ centralize~$G'$, \cref{Cents->Homo} implies $\voltage (L_0, a, L_0^{-1}, a) = [a,b]^2 [a,c]$.
This generates~$G'$, so \cref{FGL} applies.
%

\begin{subcase} \label{gammas&invert}
Assume neither $\gamma_a$ nor~$\gamma_b$ is in~$\ZZ_p$, and $b$~does not centralize~$\ZZ_p$.
\end{subcase}
From \cref{a=2S=3assump}, we know $\quot b = 2$ (so $b$ must invert~$G'$).

We may assume $[a,c] \in \ZZ_2$, for otherwise \cref{a=c=2} could be applied by interchanging $b$ and~$c$. Then we may assume $[a,c]$ is the nontrivial element of~$\ZZ_2$, for otherwise the assumption that $\gamma_a \notin \ZZ_p$ implies $\langle [a,b] \rangle = G'$, so $\langle a,b \rangle \normal G$, and then \cref{Durnberger-commuting} applies with $s = c$.

By applying the same argument, with $a$ and~$b$ interchanged, we may assume $[b,c]$ is also the nontrivial element of~$\ZZ_2$. This implies $[a,b] \in \ZZ_p$, since $\gamma_b \notin \ZZ_p$.

Note that, since $a$ and~$b$ both have order~$2$ (and invert~$G'$), the image of $\langle a,b \rangle$ in $G/\ZZ_2$ is the dihedral group of order~$2p$. Also, the preceding two paragraphs 
imply that $c$ is in the center of $G/\ZZ_2$. Therefore, we have the following hamiltonian cycle in $\Cay \bigl( G/\ZZ_2 ; S \bigr)$:\refnote[15]{DpC}
	$$ C = \bigl( c, ( c^{\ell-2}, a, c^{-(\ell-2)}, b )^p\#, c^{-1}, (a^{-1},b^{-1})^p\# \bigr) .$$
Since $[a,b]$ projects trivially to~$\ZZ_2$, \cref{Cent->Divides} implies that $\ell$~is even, so,
calculating modulo~$\ZZ_p$, we have
	\begin{align*}
	 \voltage C
	&= c ( c^{\ell-2} a c^{-(\ell-2)}  b )^p b^{-1} c^{-1} (a^{-1}b^{-1})^p b
	\\&\equiv c ( a b )^p  b^{-1} c^{-1} (a^{-1}b^{-1})^p b
	&& \begin{pmatrix} \text{$\ell-2$ is even, so $c^{\ell-2}$} \\ \text{is central modulo~$\ZZ_p$} \end{pmatrix}
	\\&\equiv z^{2p-1} ( a b )^p b^{-1} (a^{-1}b^{-1})^p b
	&& \begin{pmatrix} \text{letting $z = [a,c] = [b,c]$ be} \\ \text{the nontrivial element of~$\ZZ_2$} \end{pmatrix}
	\\&\equiv z
	&& (\text{$z^2 = e$ and $[a,b] \in \ZZ_p$})
	.\end{align*}
Since this generates~$\ZZ_2$, \cref{FGL} applies.
\end{proof}

\section{Cases with \texorpdfstring{$|\quot a| = 2$ and $\#S \ge 4$}{|aG'| = 2 and |S| > 3}} \label{a=2S>3}

\begin{assump} \label{S>3Assump}
In this \lcnamecref{a=2S>3}, we assume 
	\begin{itemize}
	\item $\# S \ge 4$,
	and
	\item $|\quot s| = 2$, for all $s \in S$, such that $s$~does not centralize~$G'$.
	\end{itemize}
We also assume \cref{s=t} does not apply. (So $|s| = 2$.) 

Furthermore, we assume $\quot b \notin \langle \quot a \rangle$ (otherwise, \cref{s=t} applies). 
Then it is easy that we also have $\quot a \notin \langle \quot b \rangle$.\refnote{anotinb}
\end{assump}

\begin{outline}
This final 
\lcnamecref{a=2S>3} of the proof is longer than the others, so here is an outline of the cases and subcases that it considers. 
	\begin{itemize}

	\item[\ref{nocent}:]
	\emph{Assume no element of~$S$ centralizes~$G'$.}
		\begin{itemize}
		
		\item[\ref{>5}:]
		\emph{Assume $\#S \ge 5$.}
		
		\item[\ref{S=4NotCent}:]
		\emph{Assume $\#S = 4$.}
		
		\end{itemize}
	
	\item[\ref{S=4asNotZp}:]
	\emph{Assume there exists $s \in S$, such that $[a,s] \notin \ZZ_p$, and, in addition, either 
	$s = b$, 
	or
	$b$~centralizes~$G'$, 
	or
	$\ZZ_p \subseteq \langle S \sm \{a\}\rangle ' $.}
		\begin{itemize}
		
		\item[\ref{S=4asNotZpAbel}:]
		\emph{Assume $\ZZ_p \nsubseteq \langle S \sm \{a\} \rangle ' $.}
		
		\item[\ref{S=4asNotZpNonabel}:]
		\emph{Assume $\ZZ_p \subseteq \langle S \sm \{a\} \rangle ' $.}
		
		\end{itemize}
	
	\item[\ref{S>3bcent}:]
	\emph{Assume $b$ centralizes~$G'$.}
		\begin{itemize}
		
		\item[\ref{S>3bcent-somenontriv}:]
		\emph{Assume there exists $c \in S$, such that $[c,b] \notin \ZZ_p$.}

		\item[\ref{S>3bcent-alltriv}:]
		\emph{Assume $[c,b] \in \ZZ_p$ for all $c \in S$.}
		
		\end{itemize}
		
	\item[\ref{Leftovers}:]
	\emph{Assume that none of the preceding cases apply.}
	\\ Since \cref{nocent} does not apply, some element~$c$ of~$S$ centralizes~$G'$.
		\begin{itemize}
		
		\item[\ref{abinZpnotZ2}:]
		\emph{Assume $\langle [s,c] \rangle \neq \ZZ_2$, for some $s \in S \sm \{c\}$.}
		
		\item[\ref{abinZp=Z2}:]
		\emph{Assume $\langle [s,c] \rangle = \ZZ_2$, for all $s \in S \sm \{c\}$.}
		
		\end{itemize}
	
	\end{itemize}
\end{outline}
%

\begin{notation}
Let $ n = |\quot b|$ and $ \ell = | \quot G : \langle \quot a, \quot b \rangle| = |\quot G|/(2n) $.
\end{notation}

\begin{note} \label{minmodab}
The irredundance of~$S$ implies $S \sm \{a,b\}$ is an irredundant generating set for $\quot G / \langle \quot a, \quot b \rangle$ \csee{Slessamin}, so $\ell \ge 4$.
\end{note}

\begin{CASE} \label{nocent}
Assume no element of~$S$ centralizes~$G'$.
\end{CASE}

\begin{proof}
From \cref{S>3Assump}, we see that every element of~$S$ inverts~$G'$ (and has order~$2$).
We may assume no two elements of~$S$ commute, for otherwise it is not difficult to see that \cref{Durnberger-commuting} applies.\refnote{NoCommute}

Let $c,d \in S \sm \{a,b\}$, and let $\gamma = [a,b] \, [a,c]$. We claim that we may assume $\gamma \notin \ZZ_2$, by permuting $b,c,d$. To this end, first note that if $\gamma \in \ZZ_2$, then $\ZZ_p \subseteq \langle [a,c] \rangle$, so there is no harm in putting $c$ into the role of~$b$. Now, let us suppose $[a,b][a,c]$, $[a,c][a,d]$, and $[a,d][a,b]$ are all in~$\ZZ_2$. Then
	$$ [a,b] \equiv [a,c]^{-1} \equiv [a,d] \equiv [a,b]^{-1}  \pmod{\ZZ_2} ,$$
which contradicts the fact that $[a,b] \notin \ZZ_2$ (and $p$~is odd). 

Let\refnote[17]{cacb}
	$$ C = \bigl( (c,a,c,b)^2\#, d \bigr)^2 ,$$
so $C$ is a hamiltonian cycle in $\Cay \bigl( \langle \quot a, \quot b, \quot c, \quot d \rangle ; \{a,b,c,d\} \bigr)$ that contains the vertex-disjoint paths $[\quot e](c, a, c )$, $[\quot{abc}](a)$, 
$[\quot{bd}](c, a, c )$, and $[\quot{acd}](a)$. Applying \cref{StandardAlteration} to the paths $[\quot e](c, a, c )$ and $[\quot{abc}](a)$ (so $s = c$, $t = a$, $u = b$, and $h = bc$) will multiply the voltage by~$\gamma$.\refnote{nocentvoltage1}
Applying \cref{StandardAlteration} to the other two paths $[\quot {bd}](c, a, c )$ and $[\quot{acd}](a)$ (so $s = c$, $t = a$, $u = b$, and $h = cd$) will also multiply the voltage by~$\gamma$\refnote{nocentvoltage2}
(because $bc$ and $cd$ both centralize~$G'$). Therefore, applying \cref{StandardAlteration} twice yields a hamiltonian cycle $C''$, such that $(\voltage C)^{-1} (\voltage C'') = \gamma^2$, which is a generator of~$\ZZ_p$.

\begin{subcase} \label{>5}
Assume $\#S \ge 5$.
\end{subcase}
If there exist $s,t \in S$, such that $s \notin \{a,b,c\}$, and $[s,t] \notin \ZZ_p$, then the preceding paragraph 
 implies that \fullcref{UsualConnSum}{even} applies.

Thus, we may assume that the preceding condition does not apply (for any legitimate choice of $a$, $b$, and~$c$). Fix two elements $x,y \in S \sm \{a,b,c\}$. The failure of the condition implies $[x,S] \subseteq \ZZ_p$. In particular, $[x,y]$ must be a generator of~$\ZZ_p$ (because no two elements of~$S$ commute), so we may let $\{x,y\}$ play the role of $\{a,b\}$. So we may let $\{x,y,b,c\}$ play the role of $\{a,b,c,d\}$. Then, since $a \notin \{x,y,b,c\}$, the failure of the condition implies $[a,S] \subseteq \ZZ_p$. Similarly, $[b,S]$ and $[c,S]$ are also in $\ZZ_p$. So $[s,t] \subseteq \ZZ_p$ for all $s,t \in S$. This contradicts the fact that $\langle [S,S] \rangle = G' \nsubseteq \ZZ_p$.

\begin{subcase} \label{S=4NotCent}
Assume $\#S = 4$.
\end{subcase}
For convenience, in this \lcnamecref{S=4NotCent} (and only in this \lcnamecref{S=4NotCent}), we drop our standing assumption that $\langle [a,b] \rangle$ contains~$\ZZ_p$. Instead, choose $b,d \in S$, such that $[b,d]$ projects nontrivially to~$\ZZ_2$. 
A straightforward calculation (using the fact that $a$, $b$, $c$, and~$d$ all invert~$G'$) shows that\refnote[15]{S=4NotCentvoltage}
	$$ \voltage C = [c,d]^4 [d,a]^2 [d,b] .$$
Since $[d,b]$ projects nontrivially to~$\ZZ_2$, but $[c,d]^4$ and $[d,a]^2$ have even exponents, so they obviously do not, we see that $\ZZ_2 \subseteq \langle \voltage C \rangle$. Therefore, we may assume $\voltage C \in \ZZ_2$, for otherwise \cref{FGL} applies. 

We may assume $\gamma \in \ZZ_2$, for otherwise applying \cref{StandardAlteration} twice (as in the paragraph immediately before 
 \cref{>5}) yields a hamiltonian cycle whose voltage generates~$G'$, so \cref{FGL} applies. By the definition of~$\gamma$, this means $[a,b][a,c] \in \ZZ_2$. And we may assume the same is true when $b$ and~$d$ are interchanged, which means $[a,d] [a,c] \in \ZZ_2$. So 
 	$$[a,b] \equiv [a,c]^{-1} \equiv [a,d] \pmod{\ZZ_2} .$$
By interchanging $a$ and~$c$, we conclude that we may also assume 
	$$[c,b] \equiv [c,a]^{-1} \equiv [c,d] \pmod{\ZZ_2} .$$
So 
	$$[c,d] \equiv[c,a]^{-1} = [a,c] \equiv [a,d]^{-1} = [d,a] \pmod{\ZZ_2}  .$$
Therefore 
	$$[d,a]^6 [d,b] = [d,a]^4 [d,a]^2 [d,b] \equiv [c,d]^4 [d,a]^2 [d,b] = \voltage C \equiv 0 \pmod{\ZZ_2} .$$

If $p \neq 3$, then, since we may assume the same is true when we interchange $a$ and~$c$, we conclude that $[d,c] \equiv [d,a] \pmod{\ZZ_2}$.\refnote{a<>c} 
Since we also have $[c,d] \equiv  [d,a] \pmod{\ZZ_2}$, we conclude that $[c,d]$ and $[a,d]$ are in~$\ZZ_2$. This implies $[b,d] \notin \ZZ_2$ (since $d$ does not centralize~$\ZZ_p$, and is therefore not in the center of $G/\ZZ_2$), so 
	$$\voltage C = [c,d]^4 [d,a]^2 [d,b] \equiv e^4  e^2 [d,b] = [d,b] \not\equiv 0 \pmod{\ZZ_2} .$$
This contradicts the fact that $\voltage C \in \ZZ_2$.

We now assume $p = 3$. Then the equation $[d,a]^6 [d,b] \equiv 0 \pmod{\ZZ_2}$ implies $[d,b] \in \ZZ_2$. This conclusion came from assuming only that $[d,b] \notin \ZZ_p$. Therefore, for all $s,t \in S$, the commutator $[s,t]$ must be in either $\ZZ_2$ or~$\ZZ_p$. However, 
	$$ [a,b] \equiv [c,a] \equiv [a,d] \equiv [b,c] \equiv [d,c] \pmod{\ZZ_2} ,$$
and $[a,b] \notin \ZZ_2$. Therefore, we conclude all five of these other commutators are in~$\ZZ_p$. (Therefore, the stated congruences between these commutators are actually equalities.)

Now, interchanging $a \leftrightarrow b$ and $c \leftrightarrow d$ in~$C$ yields a hamiltonian cycle~$C^*$, such that 
	$$\voltage C^* = [d,c]^4 [c,b]^2 [c,a] = [d,c] [b,c] [c,a] = [c,a]^3 = e $$
(because $p = 3$).
Let $\gamma\,^* = [b,a] \, [b,d]$, so $\gamma\,^*$ is obtained from $\gamma = [a,b]  [a,c]$ by interchanging $a \leftrightarrow b$ and $c \leftrightarrow d$. Then, since applying \cref{StandardAlteration} to~$C$ can multiply the voltage by $\gamma = [a,b] \, [a,c]$, we know that applying \cref{StandardAlteration} to~$C^*$ can multiply the voltage by~$\gamma\,^*$, which generates~$G'$. So \cref{FGL} applies.
\end{proof}

\begin{CASE} \label{S=4asNotZp} 
Assume there exists $s \in S$, such that $[a,s] \notin \ZZ_p$, and:
	$$ \text{either 
	\  $s = b$, 
	\ or \ 
	$b$~centralizes~$G'$, 
	\ or \ 
	$\ZZ_p \subseteq \langle S \sm \{a\}\rangle ' $}
	.$$
\end{CASE}

\begin{proof}
Let $S_0 = S \sm \{a\}$. Note that the irredundance of~$S$ implies $a \notin \langle S_0 \rangle \ZZ_2$ \csee{Z2inFrattini}.

\begin{subcase} \label{S=4asNotZpAbel}
Assume $\ZZ_p \nsubseteq \langle S_0\rangle ' $.
\end{subcase}
If $[a,b] \notin \ZZ_p$, we assume that $s = b$.
Let
	$$ g = \begin{cases}
		s & \text{if $[s,a] \notin \ZZ_2$}, \\
		sb^2 & \text{if $[s,a] \in \ZZ_2$}
		. \end{cases} $$
Note that $\langle [g,a] \rangle = G'$.\refnote{S=4asNotZpAbelga}

Let $H^* = \langle S_0 \rangle\ZZ_2/\ZZ_2$.
From the assumption of this \lcnamecref{S=4asNotZpAbel}, we know that $H^*$ is abelian. Therefore, \cref{ChenQuimpoOddEndpt} provides a hamiltonian path $L = (s_i)_{i=1}^r$ in $\Cay( \quot{H^*} ; S_0 )$, such that $s_1s_2 \cdots s_r \in g \ZZ_2$. Then $(L^{-1}, a, L, a)$ is a hamiltonian cycle in $\Cay(\quot G; S)$, and
	$$\voltage C = [ s_1s_2 \cdots s_r, a ] \in [g\ZZ_2,a] = \{[g,a]\} $$
(since $\ZZ_2$ is in the center of~$G$). This voltage generates~$G'$, so \cref{FGL} applies.

\begin{subcase} \label{S=4asNotZpNonabel}
Assume $\ZZ_p \subseteq \langle S_0\rangle ' $.
\end{subcase}
Suppose $w,x,y \in S^{\pm1} \sm \{a\}$, such that 
	\begin{align} \label{ascending}
	 \langle \quot w \rangle \subsetneq \langle \quot w, \quot x \rangle \subsetneq \langle \quot w, \quot x, \quot y \rangle
	. \end{align}
It is easy to construct a hamiltonian cycle~$C_0$ in $\Cay( \langle \quot {S_0} \rangle; S_0 )$, such that $C_0$ contains the oriented paths $[\quot {hw^{-1}y^{-1}}](w, x, w^{-1})$ and $[\quot{hx}](x^{-1})$, for some $h \in G$.\refnote{S=4asNotZpNonabelC0}
Furthermore, if
	\begin{align} \label{xnots16}
	\text{either \ $x \notin \{s^{\pm1}\}$ \ or \ $| \quot G | > 16$}
	, \end{align}
then, for some $\epsilon \in \{\pm1\}$, it is not difficult to arrange that the hamiltonian cycle~$C_0$ contains the oriented edge $[\quot {s^{\epsilon}}](s^{-\epsilon})$,\refnote{not16}
 and that this edge is not in either of the above-mentioned paths.

Applying \cref{StandardAlteration} to the first two paths (so $s = w$, $t = x$, and $u = y$) yields a hamiltonian cycle~$C_1$, such that $(\voltage C_0)^{-1} (\voltage C_1)$ is conjugate to $[x^{-1}, y] \, [w, x^{-1}]^y$.
Removing the edge $[\quot {s^{\epsilon}}](s^{-\epsilon})$ yields hamiltonian paths $C_0\#$ and~$C_1\#$ from $\quot e$ to~$\quot s^\epsilon$. 

From \cref{Z2inFrattini} and the assumption of this \lcnamecref{S=4asNotZpNonabel}, we see that $\langle \quot{S_0} \rangle \neq \quot G$.\refnote{anotS0G'}
So 
	$$ \text{$C_0^+ = \bigl( C_0\#, a, (C_0\#)^{-1}, a \bigr)$ \ and \ $C_1^+ = \bigl( C_1\#, a, (C_1\#)^{-1}, a \bigr)$} $$
are hamiltonian cycles in $\Cay(\quot G; S)$. 
For $k = 0,1$, we have
	$$ \voltage C_k^+ = \bigl[ \bigl( ( \voltage C_k) s^\epsilon \bigr)^{-1}, a \bigr] .$$
Since $\voltage C_k \in G'$, and $G'$ is central modulo~$\ZZ_p$ (and from the choice of~$s$), we have 
	$$ \voltage C_k^+ \equiv [s^\epsilon, a] \not\equiv e \pmod{\ZZ_p} .$$

Furthermore, if $[x^{-1}, y] \, [w, x^{-1}]^y$ projects nontrivially to~$\ZZ_p$, then $(\voltage C_0^+)^{-1}(\voltage C_1^+)$ does not centralize~$a$ modulo~$\ZZ_2$, so $\voltage C_0^+$ and $\voltage C_1^+$ are not both in~$\ZZ_2$. This implies that $\voltage C_k^+$ generates~$G'$ for some~$k$, so \cref{FGL} applies.
Therefore (after replacing $x^{-1}$ with~$x$ for simplicity), we may assume 
	\begin{align} \label{IfAscending}
	 \text{$[w, x]^y \, [x, y] \in \ZZ_2$ \ for all $w,x,y \in S^{\pm1} \sm \{a\}$ that satisfy \pref{ascending} and~\pref{xnots16}} 
	 . \end{align}
We will show that this leads to a contradiction.

Assume, for the moment, that $b$ centralizes~$G'$.
Then $n = |\quot b| > 2$ (because \cref{Divbyp} implies that $|b| \neq 2$), so $|\quot G| = 2n\ell > 2 \cdot 2 \cdot 4 = 16$. Therefore \pref{xnots16} is automatically satisfied.
Let $x,y \in S_0 \sm \{b \}$, such that $x \neq y$. We see from \cref{minmodab} that \pref{ascending} is satisfied for $w = b^{\pm1}$, so \pref{IfAscending} tells us 
	$$ \text{$[b, x]^y \, [x, y]$ \ and \ $[b^{-1}, x]^y \, [x, y]$ \ are both in~$\ZZ_2$} .$$
However, we also know that $[b^{-1},x] = [b,x]^{-1}$ (because we are assuming in this paragraph 
that $b$ centralizes~$G'$). Therefore
	$$[b,x]^y \equiv [x, y]^{-1} \equiv [b^{-1}, x]^y = \bigl( [b,x]^{-1} \bigr)^y \pmod{\ZZ_2} ,$$
so $[b, x] \in \ZZ_2$ (for all $x \in S_0$). Then, since $[b, x]^y \, [x, y] \in \ZZ_2$, we conclude that $[x,y] \in \ZZ_2$, for all $x,y \in S_0$. This contradicts the assumption of this \lcnamecref{S=4asNotZpNonabel}.

Now assume $b$ does not centralize~$G'$.
We may assume \cref{nocent} does not apply, so $G'$ is centralized by some $t \in S$
 (and $t \neq b$).  
 Let $w,x \in S_0 \sm \{t\}$ with $w \neq x$. Combining the irredundance of~$S$ with the fact that $t \neq b$ implies that \pref{ascending} is satisfied for $y = t^{\pm1}$\refnote{Ascendtnotb}
 (unless $\quot w = \quot x$, when \cref{s=t} applies).
 We may assume $x \neq s$ (by interchanging $w$ and~$x$, if necessary), so \pref{xnots16} is satisfied. 
 Then \pref{IfAscending} tells us 
	$$ \text{$[w, x]^t \, [x, t]$ \ and \ $[w, x]^{t^{-1}} \, [x, t^{-1}]$ \ are both in~$\ZZ_2$} .$$
Since $t$ centralizes~$G'$, this implies $[x, t] \equiv [x, t^{-1}] = [x,t]^{-1} \pmod{\ZZ_2}$,
so $[x,t] \in \ZZ_2$ (for all $x \in S_0$). Since $[w, x]^t \, [x, t] \in \ZZ_2$, this implies $[w,x] \in \ZZ_2$ (for all $w,x \in S_0$). This contradicts the assumption of this \lcnamecref{S=4asNotZpNonabel}.
\end{proof}

\begin{CASE} \label{S>3bcent}
Assume $b$ centralizes~$G'$.
\end{CASE}

\begin{proof}
We consider two \lcnamecref{S>3bcent-somenontriv}s. 

\begin{subcase} \label{S>3bcent-somenontriv}
Assume there exists $c \in S$, such that $[c,b] \notin \ZZ_p$.
\end{subcase}
We use some of the arguments of \cref{bNot3&lnot2}.
We may assume $[a,s] \in \ZZ_p$ for all $s \in S$. (Otherwise, \cref{S=4asNotZp} applies, because $b$~centralizes~$G'$.) Therefore $c \neq a$.
Let $L = (s_i)_{i=1}^r$ be a hamiltonian path from $\quot e$ to $\quot b$ in $\Cay \bigl( \quot{G}/\langle \quot a \rangle; S \sm \{a\} \bigr)$, such that $s_1 = c = s_r^{-1}$, and $L$ contains a path of the form $[\quot {gc^\epsilon}](c^{-\epsilon}, b^\delta, c^\epsilon)$\refnote{S>3somenontrivL}
 (for some $\delta,\epsilon \in \{\pm1\}$) that is vertex-disjoint from $\{\quot e, \quot c, \quot b, \quot{bc} \}$. Now let  $C = (L, a, L^{-1}, a)$. Then $C$ contains vertex-disjoint paths of the form
	$$ [\quot b](a), 
	\quad
	[\quot{ca}](c^{-1}, a, c),
	\quad
	[\quot {gc^\epsilon}](c^{-\epsilon}, b^\delta, c^\epsilon),
	\text{\quad and\quad}
	 [\quot {gab^\delta}](b^{-\delta})
	 .$$ 
\begin{itemize}
	\item Applying \cref{StandardAlteration} to $[\quot b](a)$ and $[\quot{ca}](c^{-1}, a, c)$ (so $s = c^{-1}$, $t = a$, $u = b$, and $h = ab$) will be called the ``$a$-transform\rlap.'' It multiplies the voltage by\refnote[15]{S>3somenontriva}
	$$\gamma_a = [b,a] [a,c^{-1}] .$$
	\item Applying \cref{StandardAlteration} to $[\quot {gc^\epsilon}](c^{-\epsilon}, b^\delta, c^\epsilon)$ and $[\quot {gab^\delta}](b^{-\delta})$ 
	(so $s = c^{-\epsilon}$, $t = b^\delta$, $u = a$, and $h = ga$)
will be called the ``$b$-transform\rlap.'' It multiplies the voltage by a conjugate of\refnote[15]{S>3somenontrivb}
	$$\gamma_b = [a,b] [c^{-\epsilon}, b] .$$
\end{itemize}

Since $[a,b], [a,c] \in \ZZ_p$ and  $[b,c] \notin \ZZ_p$ we know $\gamma_a \in \ZZ_p$ and $\gamma_b \notin \ZZ_p$. 
Also, we may also assume $\gamma_a$ is nontrivial (by replacing $b$ with~$b^{-1}$ if necessary). 
Therefore, the argument of \cref{PreciselyOne} applies.
Namely, choose $C'$ to be either $C$ or the $b$-transform of~$C$, such that $\voltage C'$ projects nontrivially to~$\ZZ_2$. 
Then choose $C''$ to be either $C'$ or the $a$-transform of~$C'$, such that $\voltage C''$ generates~$G'$, so \cref{FGL} applies.

\begin{subcase} \label{S>3bcent-alltriv} 
Assume $[c,b] \in \ZZ_p$ for all $c \in S$. 
\end{subcase}
Choose $c,d \in S$, such that $[c,d] \notin \ZZ_p$. Assuming that \cref{S=4asNotZp,S>3bcent-somenontriv} do not apply, we have
	$$ \text{$[s,t] \in \ZZ_p$ for all $s \in \{a,b\}$ and $t \in S$} .$$
Therefore, $c,d \notin \{a,b\}$, and the element $\gamma = [a,b][d^{-1},a]$ is in~$\ZZ_p$, and we may assume (by replacing $b$ with its inverse, if necessary) that $\gamma$ generates~$\ZZ_p$. 

Let $S_0 = \{a,b,d\}$, and choose a hamiltonian cycle~$C_0$ in $\Cay \bigl( \langle \quot{S_0} \rangle; S _0 \bigr)$ that contains the oriented paths $[\quot d](d^{-1}, a, d)$\refnote{4-bCentsG'C0}
 and $[\quot{ab}](a)$, and has at least two edges labelled $x^{\pm1}$, for every $x \in S_0$.
\Cref{StandardAlteration} (with $s = d^{-1}$, $t = a$, $u = b$, and $h = b$) provides a hamiltonian cycle $C_1$, such that $(\voltage C_0)^{-1}(\voltage C_1)$ is conjugate to~$\gamma$,\refnote{4-bCentsG'gamma}
 and therefore generates~$\ZZ_p$.  Furthermore, $C_1$ contains all of the oriented edges of~$C_0$ that are not in these two above-mentioned paths, so \fullcref{UsualConnSum}{even} applies (with $g = b$ and $t = d$).
\end{proof}

\begin{CASE} \label{Leftovers}
Assume that none of the preceding cases apply. 
\end{CASE}

\begin{proof}
This implies that:
	\begin{enumerate} \renewcommand{\theenumi}{\#\arabic{enumi}}
	\item \emph{$[a,b] \in \ZZ_p$.} (Otherwise, \cref{S=4asNotZp} applies.)
	\item \label{Leftovers-invert}
	\emph{If $s \in S$, and there exists $t \in S$, such that $t$~inverts~$G'$ and $\ZZ_p \subseteq \langle [s,t] \rangle$, then $s$~inverts~$G'$.}
(If $s$ does not invert~$G'$, then we see from \cref{S>3Assump} that $s$ centralizes~$G'$, so \cref{S>3bcent} applies with $s$ and~$t$ in the roles of $b$ and~$a$, respectively.)
	\item \emph{There exists $c \in S$, such that $c$ centralizes~$G'$.}
	(Otherwise, \cref{nocent} applies.)
	From \pref{Leftovers-invert}, we know $[a,c] \in \ZZ_2$. 
	\end{enumerate}


\begin{subcase} \label{abinZpnotZ2}
Assume $\langle [s,c] \rangle \neq \ZZ_2$, for some $s \in S \sm \{c\}$.
\end{subcase}
Suppose, for the moment, that $s$ centralizes~$G'$.
Then \cref{Cents->Homo} implies $\bigl[ a, [s,c] \bigr] = \bigl[ [a,s], [a,c] \bigr] = e$ (because $G'$ is abelian), so $[s,c]$ projects trivially to~$\ZZ_p$. Since $\langle [s,c] \rangle \neq \ZZ_2$, we conclude from this that $[s,c] = e$, so \cref{NormalEasy} applies. 

We may now assume $s$ does not centralize~$G'$, so there is no harm in assuming that $s = a$. Since \pref{Leftovers-invert} implies that $[a,c] \in \ZZ_2$, we see that $[a,c]$ must be trivial. Let $H = \langle S \sm \{c\} \rangle$. We may assume $\ZZ_2 \nsubseteq H$, for otherwise $H \normal G$, so \cref{Durnberger-commuting} applies with $s = c$ and $t = a$. Therefore, $[x,y] \in \ZZ_p$ for all $x,y \in S \sm \{c\}$, but there is some $d \in S \sm \{c\}$, such that $[c,d]$ projects nontrivially to~$\ZZ_2$. 

Similarly, we may assume $\ZZ_p \nsubseteq \langle S \sm \{a\} \rangle$, for otherwise $\langle S \sm \{a\} \rangle \normal G$, so \cref{Durnberger-commuting} applies with $s = a$ and $t = c$. This means $[x,y] \in \ZZ_2$ for all $x,y \in S \sm \{a\}$. In particular, since $b$ and~$d$ are in both $S \sm \{a\}$ and $S \sm \{c\}$, we must have $[b,d] \in \ZZ_2 \cap \ZZ_p = \{e\}$.

Choose a hamiltonian cycle~$C_0$ in $\Cay \bigl(  \quot H ;  S \sm \{c\}  \bigr)$ that contains the oriented paths $[\quot {d}](d^{-1}, b, d)$ and $[\quot{ab}](b)$.\refnote{abinZpnotZ2C0}
If we apply \cref{StandardAlteration} to these paths (so $s = d^{-1}$, $t = b$, $u = a$, and $h = a$), then the voltage is multiplied by a conjugate of $[b,a] \, [b,d^{-1}]$,\refnote{abinZpnotZ2voltage}
 which is a generator of~$\ZZ_p$ (since $[a,b]$ generates~$\ZZ_p$ and $[b,d]$ is trivial). Therefore, \fullcref{UsualConnSum}{notZ2} applies with $s = t = d$ and $u = a$.

\begin{subcase} \label{abinZp=Z2}
Assume $\langle [s,c] \rangle = \ZZ_2$, for all $s \in S \sm \{c\}$.
\end{subcase}
For convenience, let $\widehat G = G/\ZZ_2$ and $\widehat H = \langle \widehat S \sm \{\widehat c\} \rangle$. Then $|\widehat H'| = p$ is prime, so \cref{G'=p} provides a hamiltonian path~$L$ in $\Cay \bigl( \widehat H; S \sm \{c\} \bigr)$. Since $\widehat c$ is central in~$\widehat G$, there is a spanning subgraph of $\Cay(\widehat G ; S)$ that is isomorphic to the Cartesian product $L \mathbin\Box (\widehat c^{\ell-1})$, where $\ell = |\quot G : \langle \quot {S \sm \{c\}} \rangle |$. Since $|\widehat G|$ is even, it is easy to find a hamiltonian cycle~$C$ in $L \mathbin\Box (\widehat c^{\ell-1})$ \csee{ChenQuimpoEvenGrid}, and this yields a hamiltonian cycle~$\widehat C$ in $\Cay(\widehat G ; S)$. 

To complete the proof, we carry out a straightforward (and well-known) calculation to verify that $\voltage \widehat C$ is nontrivial, so \cref{FGL} applies.

If we view the Cartesian product $L \mathbin\Box (\widehat c^{\ell-1})$ as a grid of squares, then the interior of the hamiltonian cycle~$C$ is a union of squares of the grid.
Graph theoretically, this means $C$~is the connected sum of some number~$N$ of digons of the form $[g](t,t^{-1})$ (where $t \in S^{\pm1}$). 
Note that if $\mathcal{C}$ is an $r$-cycle (with $r \ge 2$), then $\mathcal{C} \connsum^s_t (t,t^{-1})$ is an $(r+2)$-cycle. Therefore, since the length of~$C$ is $|\widehat G|$, we have $2N = |\widehat G| \equiv 0 \pmod{4}$, so $N$~is even.

Now, each $4$-cycle in $L \mathbin\Box (\widehat c^{\ell-1})$ is of the form $[\widehat{g}](s^{-1}, t^{-1}, s, t)$, where one of $s$ and~$t$ is in $\{c^{\pm1}\}$, and the other is in $S^{\pm1} \sm \{c^{\pm1}\}$. This means that in any connected sum $\mathcal{C} \connsum^s_t [g](t,t^{-1})$, one of $s$ and~$t$ is in $\{c^{\pm1}\}$, and the other is in $S^{\pm1} \sm \{c^{\pm1}\}$. By the assumption of this \lcnamecref{abinZp=Z2}, we conclude that  $[s,t] = z$, where $z$~is the generator of~$\ZZ_2$. Therefore
	\begin{align*} 
	\voltage C
	&= \voltage \Bigl(  [\widehat{g_1}](t_1, t_1^{-1}) 
		\ \connsum_{t_2}^{s_2} \ [\widehat{g_2}](t_2, t_2^{-1}) 
		\ \connsum_{t_3}^{s_3} \ \cdots 
		\ \connsum_{t_{N}}^{s_{N}} \ [\widehat{g_N}](t_N, t_N^{-1}) 
		 \Bigr)
	\\&\equiv \prod\nolimits_{i = 2}^N [s_i, t_i]
		&& \begin{pmatrix} \text{\cref{VoltageOfConnSumModZp}} \\ \text{and $\voltage (t, t^{-1}) = e$} \end{pmatrix}
	\\&= z^{N-1}
	\\&\not\equiv e
	\qquad \pmod{\ZZ_p} 
		&& \text{($N-1$ is odd)}
	. \qedhere \end{align*}
\end{proof}


\AtEndDocument{

\newpage

\refereeheaders

\addtocontents{toc}{\medskip}

\begin{appendix}

\section{Notes to aid the referee}

\thispagestyle{plain}

\vfill

\begin{aid} \label{BasicVoltageRef}
Write $C = [N v] (s_i)_{i=1}^n$.

\medskip

\pref{BasicVoltage-welldef}
Suppose $C$ has another representation: $C = [Nw](t_j)_{j=1}^n$. Since $Nw$ is a vertex on~$C$, there is some~$\ell$, such that $N v s_1 s_2 \cdots s_\ell = Nw$. Then $t_j = s_{j-\ell}$ for all~$j$ (with subscripts read modulo~$n$).
Also, letting $g = (s_1 s_2\cdots s_\ell)^{-1}$, we have $Nwg = Nv$, so (since $N$ is normal) there is some $h \in N$, such that $wg = vh$. Therefore
	\begin{align*}
	\textstyle {\vphantom{\Bigl|}}^w \! {\Bigl( \prod_{j=1}^n t_j \Bigr)}
	&\textstyle= {\vphantom{\Bigl|}}^w \! {\Bigl( \bigl( \prod_{i= \ell + 1}^n s_i  \bigr)  \bigl( \prod_{i=1}^\ell s_i  \bigr) \Bigr)}
	\\&\textstyle= {\vphantom{\Bigl|}}^w \! {\Bigl( g \, \bigl( \prod_{i=1}^\ell s_i  \bigr) \bigl( \prod_{i= \ell + 1}^n s_i  \bigr)  g^{-1} \Bigr)}
	\\&\textstyle= {\vphantom{\Bigl|}}^w \! {\Bigl( g \, \bigl( \prod_{i=1}^n s_i  \bigr)  g^{-1} \Bigr)}
	\\&\textstyle= {\vphantom{\Bigl|}}^{wg} \! {\Bigl( \prod_{i=1}^n s_i \Bigr)}
	\\&\textstyle= {\vphantom{\Bigl|}}^{vh} \! {\Bigl( \prod_{i=1}^n s_i \Bigr)}
	\\&\textstyle= {\vphantom{\Bigl|}}^{v} \! {\Bigl( \prod_{i=1}^n s_i \Bigr)}
	&& \begin{pmatrix}
	\text{since $N$ is abelian, we have} \\ \text{${}^h x = x$ for all $x \in N$} \end{pmatrix}
	. \end{align*}
This means that the two representations $[Nw](t_j)_{j=1}^n$ and $[N v] (s_i)_{i=1}^n$ yield the same value for the voltage, so the voltage is well defined.

\medskip

\pref{BasicVoltage-translate} 
We have $gC = [N gv] (s_i)_{i=1}^n$, so
	\begin{align*}
	\textstyle
	\voltage gC
	 = {}^{gv} \! \bigl( \prod_{i=1}^n s_i  \bigr) 
	 = {\vphantom{\Bigl|}}^{g} \! \Bigl( {}^v \! \bigl( \prod_{i=1}^n s_i  \bigr)  \Bigr)
	 = {\vphantom{\Bigl|}}^{g} \! \Bigl(\voltage C  \Bigr)
	.\end{align*}

\medskip

\pref{BasicVoltage-reverse} 
We have $-C = [N v] (s_n^{-1}, s_{n-1}^{-1}, \ldots, s_1^{-1} )$, so
	\begin{align*}
	\textstyle
	\voltage (-C)
	 = {}^v (s_n^{-1} s_{n-1}^{-1} \cdots s_1^{-1} )
	 = {}^v \bigl( (s_1 s_2 \cdots s_n)^{-1} \bigr)
	 = \bigl( {}^v ( s_1 s_2 \cdots s_n) \bigr)^{-1}
	 = \bigl(\voltage C  \bigr)^{-1}
	.\end{align*}
\end{aid}

 \begin{aid} \label{ChenQuimpoOddEndptvoltage}
We have
 	\begin{align*}
	 \voltage (s_i)_{i=1}^r
	&= \prod_{i=1}^{k\ell/2} \bigl( (\voltage L) \, t_{2i-1} \, (\voltage L)^{-1} \, t_{2i}  \bigr) t_{k\ell}^{-1}
	\\&= \prod_{i=1}^{k\ell/2} \bigl( t_{2i-1}  t_{2i}  \bigr) t_{k\ell}^{-1}
		&& \text{($H$ is abelian)}
	\\&= \prod_{i=1}^{k\ell-1} t_i
	\\&= x^p y^q
	. \end{align*}
\end{aid}

\begin{aid} \label{StandardAlterationPf}

\setcounter{case}{0}

\begin{case}
Assume that $C_0$ contains $[\quot{h}](t)$. 
\end{case}
Construct $C_1$ by replacing:
	\begin{itemize}
	\item the oriented edge $[\quot{h}](t)$ with the oriented path $[\quot{h}](u^{-1}, t, u)$,
	and
	\item the oriented path $[\quot{h s^{-1} u^{-1}}](s,t,s^{-1})$ with the oriented edge $[\quot{h s^{-1} u^{-1}}](t)$.
	\end{itemize}
	\centerline{$C_0$: \lower3cm\hbox{\includegraphics[scale=0.8]{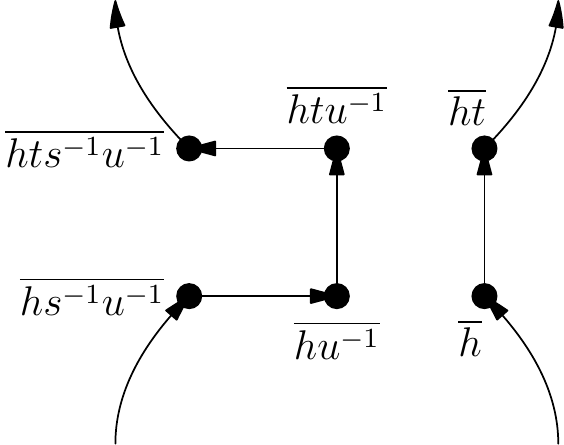}} \hfil $C_1$: \lower3cm\hbox{\includegraphics[scale=0.8]{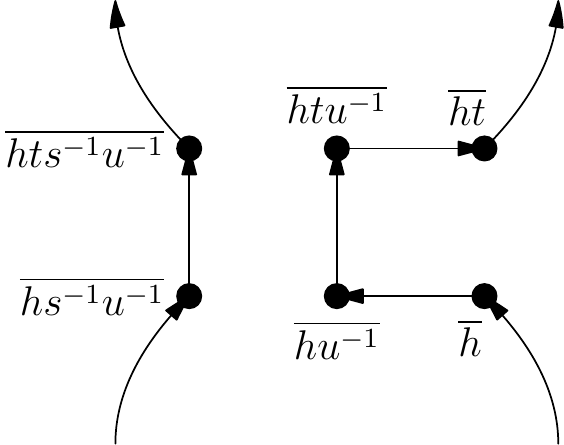}}}
	
To calculate the voltage of~$C_1$, write $C_0 = [\quot{h}](s_1,\ldots,s_n)$. Then $s_1 = t$ and there is some~$\ell$ with $\quot{s_1}\cdots \quot{s_\ell} = \quot{u}^{-1}$, so $(s_\ell, s_{\ell+1}, s_{\ell+2}) = (s,t,s^{-1})$, and we have
	$$ C_1 = [\quot{h}]\bigl( u^{-1}, t, u, \,  (s_i)_{i=2}^{\ell-1} , \, t , (s_i)_{i=\ell+3}^n \bigr) .$$
Note that if we let $\pi = \prod_{i=1}^\ell s_i$, then $\pi \equiv u^{-1} \pmod{N}$, so ${}^\pi x = x^u$ for all $x \in N$ (since $N$ is commutative). Therefore
	\begin{align*}
	\bigl(\voltage C_1 \bigr)^h
	&= (u^{-1} t u) \left( \prod_{i=2}^{\ell-1} s_i \right) t  \left( \prod_{i=\ell+3}^n s_i \right) 
	\\&= (u^{-1} t u) t^{-1} \left( \prod_{i=1}^{\ell} s_i \right) s^{-1} \   t  \  s t^{-1} \left( \prod_{i=\ell+1}^n s_i \right) 
	\\&= [u, t^{-1}] \left( \prod_{i=1}^{\ell} s_i \right)  [s, t^{-1}] \left( \prod_{i=\ell+1}^n s_i \right) 
	\\&= [u, t^{-1}]  \  {}^\pi[s, t^{-1}] \left( \prod_{i=1}^{\ell} s_i \right)\left( \prod_{i=\ell+1}^n s_i \right) 
	\\&= [u, t^{-1}]  [s, t^{-1}]^u \, (\voltage C_0)^h 
	. \end{align*}

\begin{case}
Assume that $C_0$ contains $[\quot{ht}](t^{-1})$. 
\end{case}
Construct $C_1$ by replacing:
	\begin{itemize}
	\item the oriented edge $[\quot{ht}](t^{-1})$ with the oriented path $[\quot{ht}](u^{-1}, t^{-1}, u)$,
	and
	\item the oriented path $[\quot{h s^{-1} u^{-1}}](s,t,s^{-1})$ with the oriented edge $[\quot{h s^{-1} u^{-1}}](t)$.
	\end{itemize}
	\centerline{$C_0$: \lower3cm\hbox{\includegraphics[scale=0.8]{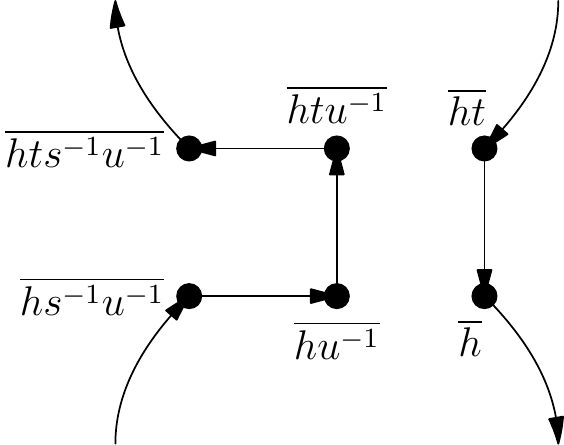}} \hfil $C_1$: \lower3cm\hbox{\includegraphics[scale=0.8]{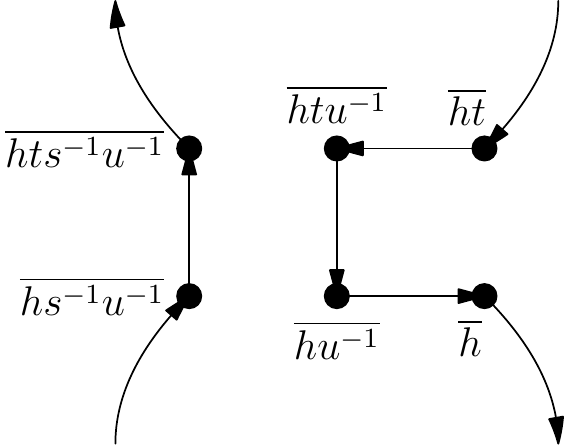}}}
	
To calculate the voltage of~$C_1$, write $C_0 = [\quot{ht}](s_1,\ldots,s_n)$. Then $s_1 = t^{-1}$ and there is some~$\ell$ with $\quot{s_1}\cdots \quot{s_\ell} = \quot{t^{-1}u^{-1}}$, so $(s_\ell, s_{\ell+1}, s_{\ell+2}) = (s,t,s^{-1})$, and we have
	$$ C_1 = [\quot{ht}]\bigl( u^{-1}, t^{-1}, u, \,  (s_i)_{i=2}^{\ell-1} , \, t , (s_i)_{i=\ell+3}^n \bigr) .$$
Then
	\begin{align*}
	\bigl(\voltage C_1 \bigr)^{ht}
	&= (u^{-1} t^{-1} u) \left( \prod_{i=2}^{\ell-1} s_i \right) t  \left( \prod_{i=\ell+3}^n s_i \right) 
	\\&= (u^{-1} t^{-1} u) t \left( \prod_{i=1}^{\ell} s_i \right) s^{-1} \   t  \  s t^{-1} \left( \prod_{i=\ell+1}^n s_i \right) 
	\\&= [u,t] \left( \prod_{i=1}^{\ell} s_i \right)  [s, t^{-1}] \left( \prod_{i=\ell+1}^n s_i \right) 
	\\&= [u, t]  \  {}^\pi[s, t^{-1}] \left( \prod_{i=1}^{\ell} s_i \right)\left( \prod_{i=\ell+1}^n s_i \right) 
	\\&= [u, t]  [s, t^{-1}]^{tu} \, (\voltage C_0)^{ht} 
	. \end{align*}
Conjugating both sides by~$t^{-1}$ yields
	$$ \bigl(\voltage C_1 \bigr)^{h} = [u, t]^{t^{-1}} \   [s, t^{-1}]^u \, (\voltage C_0)^{h}  .$$
Now note that
	$$  [u, t]^{t^{-1}} = t (u^{-1} t^{-1} ut) t^{-1} = t u^{-1} t^{-1} u = [t^{-1}, u] .$$
%
%
%
%
%
%
%
%
%
%
%
%
%
%
%
%
\end{aid}

\begin{aid} \label{KWUsesFGL}
Case~4.5 of \cite{KeatingWitte} (on page 95) considers certain groups of order~$27$.
Near the start of \cite[\S4]{KeatingWitte} (on page 92), it is stated that ``In every case except 4.5, we use the Factor Group Lemma~2.3 on $G/G'$\rlap.'' Replacing $G$ with~$\widehat G$, this means there is a hamiltonian cycle in $\Cay(\widehat G/\widehat G'; S)$ whose voltage generates~$\widehat G'$ (unless $|\widehat G| = 27$, which we have ruled out).
\end{aid}

\begin{aid} \label{GbarDivbyq}
Fix some $\widehat g \in \widehat G \sm Z(\widehat G)$, and define $\varphi \colon \widehat G \to \widehat G'$ by $\varphi(x) = [x,g]$. From \cref{Cents->Homo}, we know that $\varphi$ is a homomorphism. Since $\widehat g \notin Z(\widehat G)$, this homomorphism is nontrivial, so it must be surjective (since $|\widehat G'| = q$ is prime). Therefore $|\widehat G : \ker \varphi| = |\widehat G'| = q$. Also, we have $\widehat G' \subseteq Z(\widehat G) \subseteq \ker \varphi$. So $|\widehat G : \widehat G'|$ is divisible by~$q$.
\end{aid}

\begin{aid} \label{MinEnough}
If $T$ is a subset of~$S$, then it is obvious that $\Cay(G;T)$ is a subgraph of $\Cay(G;S)$. Therefore, in order to show that every connected Cayley graph on~$G$ has a hamiltonian cycle, it suffices to consider only the irredundant generating sets.
\end{aid}

\begin{aid} \label{ZpZ}
Suppose $\ZZ_p \cap Z(G)$ is nontrivial. Since $\ZZ_p$ has prime order, this implies $\ZZ_p \subseteq Z(G)$. However, $\ZZ_2$ is a normal subgroup that has no automorphisms, so $\ZZ_2 \subseteq Z(G)$. Therefore, $Z(G)$ contains both $\ZZ_2$ and $\ZZ_p$, and therefore contains all of~$G'$. This contradicts the fact that $G$ is not nilpotent.
\end{aid}

\begin{aid} \label{CommutatorGenZp}
If $\ZZ_p \nsubseteq \langle [a,b] \rangle$, then, since $ [a,b] \in G' = \ZZ_2 \times \ZZ_p$, we must have $[a,b] \in \ZZ_2$. If this is true for all $b \in S$, then $[a,g] \in \ZZ_2$ for all $g \in G$ (because $\langle S \rangle = G$ and $\ZZ_2 \normal G$). In particular, $[a, \ZZ_p] \subseteq \ZZ_2$. However, we also have $[a, \ZZ_p] \subseteq \ZZ_p$, because $\ZZ_p \normal G$. Therefore, $[a, \ZZ_p] \subseteq \ZZ_2 \cap \ZZ_p = \{e\}$. This contradicts \pref{aDefn}.
\end{aid}

\begin{aid} \label{Cents->HomoPf}
We have
	$$ xy \, [xy,z] = (xy)^z = x^z \, y^z = x \, [x,z] \cdot y \, [y,z] = xy \, [x,z] \, [y,z] .$$
\end{aid}

\begin{aid} \label{DivbypAid}
Let $\widehat G = G/\ZZ_2$.
We have $ [ \widehat x,  \widehat y^p ] = [ \widehat x,  \widehat y]^p = \widehat e$, so $\widehat y^p \in Z( \widehat G)$. If $p \nmid |y|$, this implies $\widehat y \in Z( \widehat G)$, which contradicts the fact that $[ \widehat x,  \widehat y]$ is nontrivial (because $\ZZ_p \subseteq \langle [x,y] \rangle$. 
\end{aid}

\begin{aid} \label{Cent->DividesPf}
Let $\widehat G = G/\ZZ_p$. By assumption, there exists $s \in S_0$, such that $\langle [\widehat g, \widehat s] \rangle = \widehat{\ZZ_2} = \widehat {G'}$. Every element of~$\widehat G$ centralizes $\ZZ_2 = \widehat G'$, so \cref{Cents->Homo} tells us that the map $\varphi(x) = [x,s]$ is a homomorphism from $\langle \widehat g, \widehat {S_0} \rangle$ to $\ZZ_2$. Since $\ZZ_2 \nsubseteq \langle S_0 \rangle'$ (and $s \in S_0$), we know $\widehat{S_0}$ is contained in the kernel of~$\varphi$. But $\langle \varphi(\widehat g) \rangle = \ZZ_2$, so the kernel of~$\varphi$ is a subgroup of index~$2$ in $\langle \widehat g, \widehat {S_0} \rangle$. Therefore
	\begin{align*}
	\frac{|\langle \quot{g} , \quot{S_0}\rangle}{| \langle \quot{S_0}\rangle|}
	&= 	\frac{|\langle \quot{g} , \quot{S_0}\rangle}{|\ker \varphi|} \cdot \frac{|\ker \varphi|}{ \langle \quot{S_0}\rangle|}
	= 	2 \cdot \frac{|\ker \varphi|}{| \langle \quot{S_0}\rangle|}
	\qquad
	\text{is even}
	. \end{align*}
\end{aid}

\begin{aid} \label{ConnSumAid}
The connected sum $C_1 \connsum_t^s C_2$ joins $C_1$ and~$C_2$ into a single large cycle by replacing the two white edges labelled $t$ and~$t^{-1}$ with the two black edges labelled $s$ and~$s^{-1}$.

	\centerline{\includegraphics{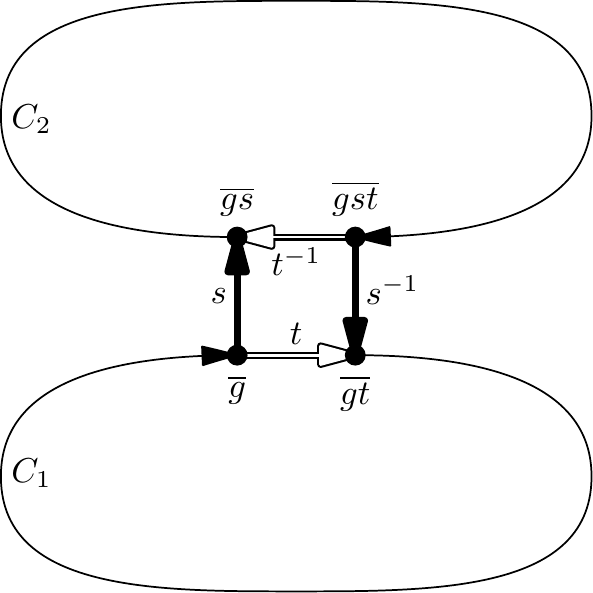}}
%
%
%
%
%
%
%
%
%
%
\end{aid}

\begin{aid} \label{UsualConnSumWitht}
Let $S_0^* = S_0 \cup \{t\}$. We verify the hypotheses of \cref{UsualConnSum} with $S_0^*$ in the role of~$S_0$ and with $t$ in the role of~$s$.
	\begin{itemize}
	\item $(\voltage C_0)^{-1} (\voltage C_1)$ is a nontrivial element of~$\ZZ_p$ (by assumption).
	\item By construction of the connected sum, $C_0'$ and $C_1'$ both contain the oriented edge $[\quot g](t)$.
	\item By construction of the connected sum (and the fact that $s_1 = s$), $C_0'$ contains the oriented edges $[\quot g](t)$ and $[\quot gs](t^{-1})$. Also, for every $x \in S_0$, $C_0$ contains at least two edges $[\quot v](x^{\pm1})$ and $[\quot w](x^{\pm1})$ that are labelled either $x$ or~$x^{-1}$. Then the subgraph of~$C_0'$ induced by~$C_0$ contains at least one of these edges, and the subgraph of~$C_0'$ induced by~$t^{n-1}C_0$ contains either $[\quot {t^{n-1}v}](x^{\pm1})$ or $[\quot {t^{n-1}w}](x^{\pm1})$; so $C_0'$ contains at least two edges that are labelled either $x$ or~$x^{-1}$.
	\item We know $c \notin S_0$ and $\ZZ_2 \subseteq \langle [c,t] \rangle$. The latter implies $c \neq t$, so $c \notin S_0 \cup \{t\} = S_0^*$.
	\item We are letting $s = t$.
	\item By assumption, either
	\begin{enumerate}
	\item there exists $u \in S \sm \{c\}$, such that $\ZZ_2 \nsubseteq \langle [u,c] \rangle$,
	or
	\item $|\quot{G} : \langle \quot{S_0}, \quot t \rangle|$ is even.
	\end{enumerate}
	The first condition makes no mention of~$S_0$, $s$, or~$t$, so remains true with $S_0^*$ in the role of~$S_0$ and with $t$ in the role of~$s$. Since $S_0^* = S_0 \cup \{t\}$, we have $S_0^* \cup \{t\} = S_0 \cup \{t\}$. So the second condition tells us that $|\quot{G} : \langle \quot{S_0^*} \rangle|$ is even. 
	\end{itemize}
\end{aid}

\begin{aid} \label{t1=t}
Suppose $|\quot{G} : \langle \quot{S_0} \rangle| = 2$, so $C = C_0 \connsum_{t_1}^c -cC_0$. Then, calculating mod~$\ZZ_p$, we have
	\begin{align*}
	 0
	&\not\equiv \voltage C 
	\\&\equiv  \voltage C_0 \cdot \voltage (-c C_0) \cdot [c,t_1] 
		&& \text{(\cref{VoltageOfConnSumModZp})}
	\\&\equiv  \voltage C_0 \cdot \voltage C_0 \cdot [c,t_1] 
		&& \text{(\fullcref{BasicVoltage}{translate})}
	\\&\equiv  [c,t_1] 
		&& \text{($x^2 \in \ZZ_p$ for all $x \in G'$, since $G' = \ZZ_2 \times \ZZ_p$)}
	. \end{align*}
By the definition of~$u$, this implies $u \neq t_1$. So $t_1 = t$.
\end{aid}

\begin{aid} \label{Cans}
The choice of the oriented edge $[\quot{g_i}](t_i)$ of~$C_0$ that is used in the connected sum is arbitrary, except that $t_1$ was chosen to make the projection of $\voltage C$ to~$\ZZ_2$ is nontrivial. Therefore, if $n > 1$, then we may use any edge that we want in order to make the connected sum $(-1)^{n-1} \pi_{n-1}C_0 \connsum_s^{s_n} (-1)^n\pi_n C_0$.  

So we may now assume $n = 1$. This means $|\quot{G} : \langle \quot{S_0} \rangle| = 2$. Therefore, by assumption, we must have $s = t$. Also, as was mentioned in the proof, we must have $t_1 = t$. So $t_1 = s$. Therefore, we may assume that the connected sum $C_0 \connsum_{t_1}^c - c C_0$ is relative to the oriented edge $[\quot {gc}](s)$ of $c C_0$ that is also in~$cC_1$.
\end{aid}

\begin{aid} \label{udNontriv}
Suppose there exist $d \in S \sm S_0$ and $u \in S \sm \{d\}$, such that $[u,d]$ projects trivially to $\ZZ_2$. Note that, by the assumption of this \lcnamecref{AllucNontriv}, we must have $d \neq c$.
\begin{enumerate}
	\item By applying the assumption of this \lcnamecref{AllucNontriv} with $d$ in the place of~$u$ (and noting that $d \neq c$), we see that $\ZZ_2 \subseteq [d,c]$.
	\item \label{udNontriv-Z2notin}
	By the choice of~$u$, we know that $\ZZ_2 \nsubseteq \langle [u,d] \rangle$.
	\end{enumerate}
Therefore, the hypotheses of the \lcnamecref{UsualConnSum} are satisfied with $d$ and~$c$ in the roles of $c$ and~$t$, respectively. Furthermore \pref{udNontriv-Z2notin} tells us that \cref{UsualConnSumnotZ2Case} applies.
\end{aid}

\begin{aid} \label{CentralizeOrNot}
Suppose $\langle S \sm \{t\} \rangle$ contains~$\ZZ_p$. Note that $\langle S \sm \{t\} \rangle$ also contains~$s$. Therefore, we have
	\begin{align*}
	\langle S \sm \{t\}, \ZZ_2 \rangle 
	&= \langle S \sm \{t\} , s, \ZZ_p, \ZZ_2 \rangle
	= \langle S \sm \{t\} , s, G' \rangle
	= \langle S \sm \{t\} , s, \gamma\rangle
	\\&\supseteq \langle S \sm \{t\} , s \gamma \rangle
	= \langle S \sm \{t\} , t \rangle
	= \langle S \rangle
	= G 
	. \end{align*}
So \cref{Z2inFrattini} tells us that $\langle S \sm \{t\} \rangle = G$.
This contradicts the irredundance of~$S$, so we conclude that $\langle S \sm \{t\} \rangle$ does not contain~$\ZZ_p$. A similar argument shows that $\langle S \sm \{s\} \rangle$ does not contain~$\ZZ_p$.

Suppose $u$ is an element of~$S \sm \{s,t\}$ that does not centralize~$\ZZ_p$. Then $u$~is not in the center of $G/\ZZ_2$, so there is some $x \in S$, such that $[x,u] \notin \ZZ_2$. We may assume (perhaps after interchanging $s$ and~$t$) that $x \neq s$, so $x \in S \sm \{s\}$. Then $u$ and~$x$ are both in $S \sm \{s\}$, so the commutator subgroup of $\langle S \sm \{s\} \rangle$ is not contained in~$\ZZ_2$. Since $G' = \ZZ_2 \times \ZZ_p$, this implies that the commutator subgroup of $\langle S \sm \{s\} \rangle$ contains~$\ZZ_p$. So $\langle S \sm \{s\} \rangle$ contains~$\ZZ_p$. This contradicts the preceding paragraph, so we conclude that every element of~$S \sm \{s,t\}$ that centralizes~$\ZZ_p$.
\end{aid}

\begin{aid} \label{sNotCentralize}
Suppose $s$ centralizes~$\ZZ_p$. Since $t = s \gamma$ and it is obvious that $\gamma$ centralizes~$\ZZ_p$ (because $\gamma \in G' = \ZZ_2 \times \ZZ_p$), we conclude that $t$~also centralizes~$\ZZ_p$. From the conclusion of the preceding paragraph, we conclude that every element of~$S$ centralizes~$\ZZ_p$. Since $S$ generates~$G$, this implies that every element of~$G$ centralizes~$\ZZ_p$. This contradicts \pref{aDefn}, so we conclude that $s$ does not centralize~$\ZZ_p$. A similar argument shows that $t$ does not centralize~$\ZZ_p$. 
\end{aid}

\begin{aid} \label{Startt}
Since $n$ is even, we may let
	$$ C = \bigl( t, (t^{m-2}, s_{2i-1}, t^{-(m-2)}, s_{2i})_{i=1}^{n/2} \#, t^{-1}, (s_{n-j}^{-1})_{j=1}^{n-1} \bigr) .$$
	\centerline{\includegraphics{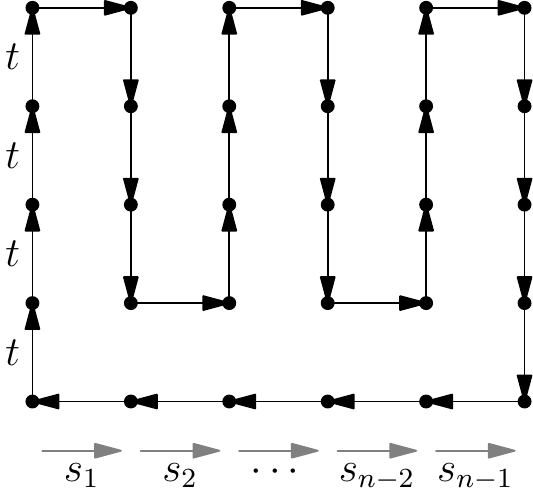}}
%
%
%
%
%
%
%
%
%
%
\end{aid}

\begin{aid} \label{MinGenSet}
Since $t \in S \sm \{s\}$ and $\quot t = \quot s$, we have $\quot{S \sm \{s\}} = \quot S$. Therefore $S \sm \{s\}$ generates~$\quot G$.

Now, we claim that $S \sm \{s\}$ is an irredundant generating set of~$\quot G$.
If not, then some proper subset~$T$ of $S \sm \{s\}$ generates~$\quot G$. This means $\langle T, G' \rangle = G$. From \cref{Z2inFrattini}, we conclude that $\langle T, \ZZ_p \rangle = G$. 
Since $\ZZ_p \normal G$, this implies $\langle T \rangle \ZZ_p = G$, so $|G : \langle T \rangle| = p$~is prime. Therefore, if we choose $x$ to be any element of $S \sm \langle T \rangle$, then $\langle T, x \rangle = G$. However, $T \cup \{x\}$ cannot be all of~$S$, because 
	$$ |T \cup \{x\}| \le |T| + 1 \le \bigl( |S| - 2 \bigr) + 1 < |S| .$$
This contradicts the irredundance of~$S$. 
\end{aid}

\begin{aid} \label{t=2HamCyc}
It is obvious that $C_0\#$ is a hamiltonian path in $\Cay\bigl(\quot G ; S \sm \{s\} \bigr)$.
	$$ \centerline{\raise1.3cm\hbox{$C_0\#$:} \quad \includegraphics{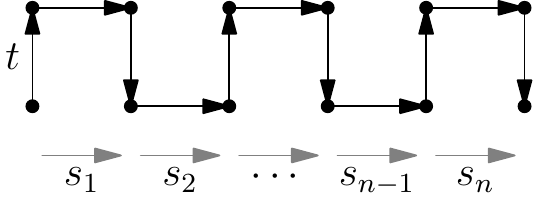}} $$
%
%
%
%
%
%
%
%
%
%
Therefore, we need only verify that $\quot{t_1} \, \quot{t_2} \cdots \quot{t_{2n}} = \quot e$. We have
	\begin{align*}
	 \quot{t_1} \, \quot{t_2} \cdots \quot{t_{2n}}
	&= \quot t \, \quot{s_1} \ \quot t \, \quot{s_2} \cdots \quot t \, \quot{s_n} 
	\\&= \quot t ^n \, \quot{s_1 s_2 \cdots s_n} 
	\\&= \quot{s_1 s_2 \cdots s_n} 
	, \end{align*}
so we wish to show $\quot{s_1 s_2 \cdots s_n}$ is trivial.

Suppose $\quot{s_1 s_2 \cdots s_n}$ is nontrivial.
Since $C_0 = (s_i)_{i=1}^n$ is a (hamiltonian) cycle in the graph $\Cay \bigl( \quot{G}/\langle \quot t \rangle ; S \sm \{s,t\}  \bigr)$, we know $\quot{s_1 s_2 \cdots s_n} \in \langle \quot t \rangle$. Therefore,  $\quot{s_1 s_2 \cdots s_n}$ must be~$\quot t$ (since this is the only nontrivial element of $\langle \quot t \rangle$). However, we also know $\quot{s_1 s_2 \cdots s_n} \in \quot{S \sm \{s,t\}}$ (because each $s_i$ is in $S \sm \{s,t\}$), so this implies that $\quot t \in \langle \quot{S \sm \{s,t\}} \rangle$. Therefore
	\begin{align*}
	\langle \quot{S \sm \{s,t\}} \rangle
	&= \langle \quot{S \sm \{s,t\}} , \quot t \rangle
	\\&= \langle \quot{S \sm \{s,t\}} , \quot s, \quot t \rangle
		&& \text{(because $\quot s = \quot t$)}
	\\&= \langle \quot{S} \rangle
	\\&= \quot G
	. \end{align*}
This contradicts the fact that $S \sm \{s\}$ is an irredundant generating set of~$\quot G$.
\end{aid}

\begin{aid} \label{VoltC0inZ2}
Since $C_0$ is a hamiltonian cycle in $\Cay\bigl(\quot G ; S \sm \{s\} \bigr)$, we know $\voltage C_0 \in G' = \ZZ_2 \times \ZZ_p$.  

On the other hand, it was pointed out in the first paragraph 
of the proof of this \lcnamecref{s=t} that $\langle S \sm \{s\} \rangle$ does not contain~$\ZZ_p$. Since $\voltage C_0 \in \langle S \sm \{s\} \rangle$, we conclude that $\ZZ_p \nsubseteq \langle \voltage C_0 \rangle$. So $\voltage C_0 \in \ZZ_2$.
\end{aid}

\begin{aid} \label{WellKnownCycle}
$$\includegraphics{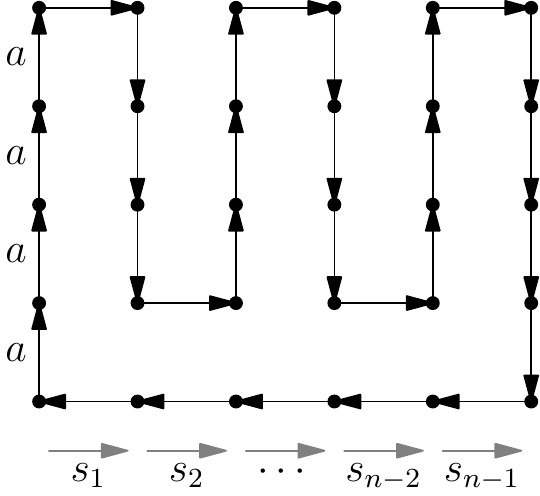}$$
%
%
%
%
%
%
%
%
%
%
\end{aid}

\begin{aid} \label{amCent}
Since $\widehat G' = \ZZ_2$, we have $\widehat G' \subseteq Z(\widehat G\,)$, which implies $[x,yz] [x,y] [x,z]$ for all $x,y,z \in \widehat G$. Therefore $[x, \widehat a^{m-2}] = [x, \widehat a]^{m-2}$. Since $m-2$~is even and $[x, \widehat a] \in \widehat G' = \ZZ_2$, this implies $[x, \widehat a^{m-2}]$ is trivial. Since $x$ is an arbitrary element of~$\widehat G$, this implies $\widehat a \in Z(\widehat G \,)$.
\end{aid}

\begin{aid} \label{generic(aS)=GEndpt}
We have $[a^\delta, b^{-1}] = [b^{-1}, a^\delta]^{-1}$. Therefore:
	\begin{itemize}
	\item Calculating modulo~$\ZZ_p$, we have 
		$$ [a^\delta, b^{-1}]^{a^\delta} \equiv [a^\delta, b^{-1}] = [b^{-1}, a^\delta]^{-1} ,$$
since $a$~centralizes~$\ZZ_2$. Therefore $[b^{-1}, a^\delta] [a^\delta, b^{-1}]^{a^\delta}$ is trivial modulo~$\ZZ_p$. In other words, $[b^{-1}, a^\delta] [a^\delta, b^{-1}]^{a^\delta} \in \ZZ_p$.
	\item We have $[a^\delta, b^{-1}]^{a^\delta} \neq [b^{-1}, a^\delta]^{-1}$, since $a^\delta = a^{\pm1}$ does not centralize~$\ZZ_p$. This implies $[b^{-1}, a^\delta] [a^\delta, b^{-1}]^{a^\delta}$ is nontrivial.
	\end{itemize}
Combining these two observations tells us that $[b^{-1}, a^\delta] [a^\delta, b^{-1}]^{a^\delta}$ is a generator of~$\ZZ_p$.
\end{aid}

\begin{aid} \label{cEven}
By assumption, $a$ is in the center of $G/\ZZ_p$, so $\langle a, \ZZ_p \rangle \normal G$. Let $\widehat G = G/\langle a, \ZZ_p \rangle$. Since $\ZZ_2 \subseteq \langle [\widehat c, \widehat d] \rangle$, \cref{Cent->Divides} tells us that $| \quot c |$ is even.
\end{aid}

 \begin{aid} \label{sk=c}
 Let $\widehat H = \langle \quot S \sm \{\quot d\} \rangle / \langle \quot a \rangle$, and let $w = |\widehat c|$. It has been pointed out that the image of~$c$ in $\quot G/ \langle \quot a \rangle$ has even order, which means that $w$~is even. 
   
Suppose $\widehat b \notin \langle \widehat c \rangle$. We may let $(t_j)_{j=1}^{\ell}$ be a hamiltonian path in $\Cay \bigl( \widehat H / \langle \widehat c \rangle ; S \sm \{c,d\} \bigr)$, such that $t_1 = b$. Then we may take $(s_i)_{i=1}^n$ to be the following hamiltonian cycle:
	$$ \includegraphics{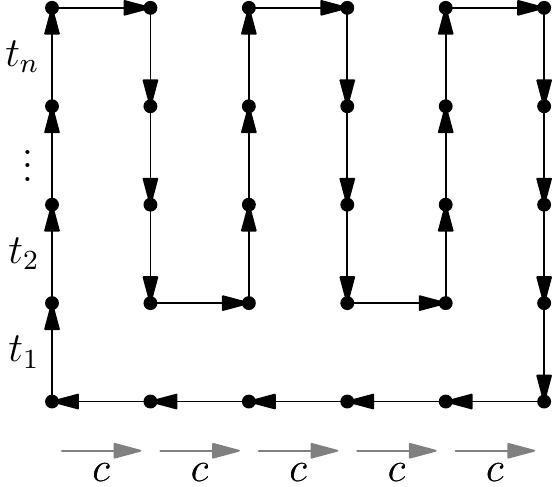} $$
%
%
%
%
%
%
%
%
%
%
%

If $\langle \widehat c \rangle = \widehat H$, then we may write $\widehat b = \widehat c^q$ with $1 \le q \le w - 1$, so we let 
	$$(s_i)_{i=1}^n = (b, c^{-(q-1)}, b, c^{w - q -1}) .$$

We now assume $\widehat b \in \langle \widehat c \rangle \neq \widehat H$. If $\widehat b = \widehat c$, then we modify the above-pictured hamiltonian cycle, by replacing a single occurrence of $c$ with~$b$. Otherwise, we write $\widehat b = \widehat c^{-q}$ with $1 \le q \le w - 2$, and replace the path $(c^{-(w-1)})$ at the end of the above-pictured hamiltonian cycle with $(b, c^{q-1}, b, c^{-(w-q-2)})$.
 \end{aid}

\begin{aid} \label{n=r=2aid}
$$ \includegraphics{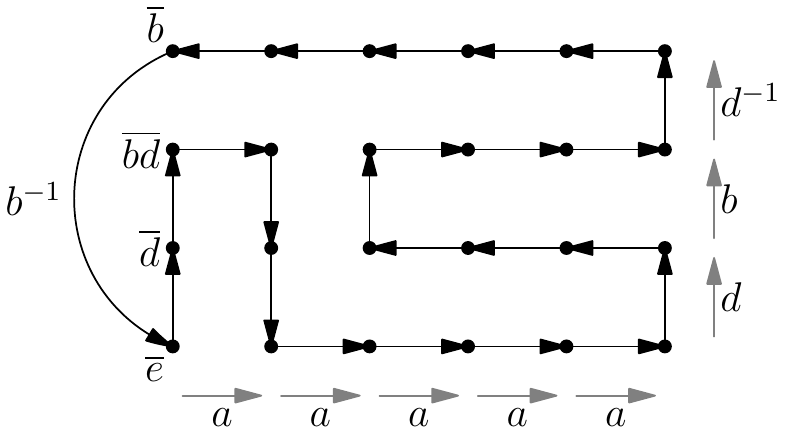} $$
\end{aid}
%
%
%
%
%
%
%
%
%
%
%
%

\begin{aid} \label{ba2}
It is a general group-theoretic fact that if $x$ does not invert $[x,y]$, then $x^2$ does not commute with~$y$. This is because
	$$ y^{x^2} = (y^x)^x = \bigl( y [y,x] \bigr)^x = y^x \, [y,x]^x = y \, [y,x] \,[y,x]^x ,$$
so $y^{x^2} = y$ iff $[y,x]^x = [y,x]^{-1}$.

Let $\widehat G = G/\ZZ_2$.
Since $|\quot a| = 3$ (and $a$ does not centralize~$G'$), we know that $a$ acts on~$\ZZ_p$ via an automorphism~$\varphi$ of order~$3$.  So $\widehat a$ does not invert $[\widehat a, \widehat b]$ (because $[\widehat a, \widehat b]$ is a generator of~$\ZZ_p$). Therefore, the general fact tells us that $\widehat a^2$ does not commute with $\widehat b$, so $[\widehat a^2,\widehat b]$ is a generator of~$\ZZ_p$. Since $a^2$ acts on~$\ZZ_p$ via $\varphi^2$, which is an automorphism of order~$3$, we know that $\widehat a^2$ does not invert $[\widehat a^2,\widehat b]$. 
So the general fact tells us that $(\widehat a^2)^2$ does not commute with~$\widehat b$. This means $b^{(a^2)^2} \not\equiv b \pmod{\ZZ_2}$. Conjugating by $a^{-2}$, we conclude that $b^{a^2} \neq b^{a^{-2}} \pmod{\ZZ_2}$.
\end{aid}

\begin{aid} \label{ba2dnote}
Suppose $[b^{a^2}, d]$ and $[b^{a^{-2}}, d]$ are both in~$\ZZ_2$. This means that $b^{a^2}$ and $b^{a^{-2}}$ both commute with~$d$ (mod~$\ZZ_2$), so the product $b^{a^2}(b^{a^{-2}})^{-1}$ also commutes with~$d$ (mod~$\ZZ_2$). For convenience, call this product~$\gamma$. Then 
	$$ \gamma = b^{a^2}(b^{a^{-2}})^{-1} \equiv b \cdot b^{-1} = e \pmod{G'} ,$$
so $\gamma \in G'$. Furthermore, we have observed that $b^{a^2} \not\equiv b^{a^{-2}} \pmod{\ZZ_2}$, so $\gamma \notin \ZZ_2$. Therefore $\ZZ_p \subseteq \langle \gamma \rangle$. Since $\gamma$ commutes with~$d$ (mod~$\ZZ_2$), we conclude that $d$ centralizes~$\ZZ_p$, and therefore centralizes~$G'$. This is a contradiction.
\end{aid}

\begin{aid} \label{relprime}
If $b$ centralizes~$G'$, then $\ZZ_p$ is in the center of $\langle b,d \rangle$. Therefore, by using \cref{Cents->Homo} as in the proof of \cref{DivBy4} (but replacing $2$ with~$p$), we see that $|\langle \quot b,d \quot \rangle|$ is divisible by~$p^2$. Since $p \neq 2$ and $|\quot G| = 12$, this is a contradiction.
\end{aid}

\begin{aid} \label{binaNotinvert-Noc-aid}
$$ \includegraphics{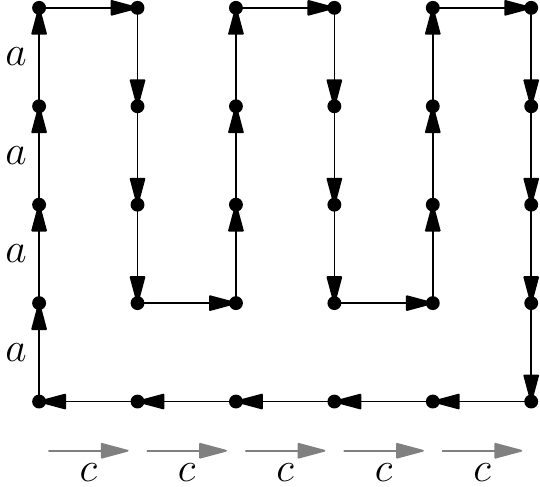} $$
\end{aid}

\begin{aid} \label{Durnberger60}
%
To see that $(\voltage C_0)^{-1} (\voltage C_1)$ is a generator of~$\ZZ_p$, we repeat the calculation in the last paragraph of \cite[p.~60]{Durnberger-semiprod} (using our notation). 
However, it is important to note that $a^k$ centralizes~$\gamma$ (because $a$~inverts $G'$ and~$k$~is even).
We have:
	\begin{align*}
	(\voltage C_0) (\voltage C_1)^{-1}
	&= (b a^{-(k-4)} b  a^{m-2k-2} b a^{-1} b a^2 b^{-2} a^{k-3})^{-1} (a^m)^{-1}
	\\&=  (b a^{-k}) a^4 (ba^{-k})  a^{m-2} (a^{-k} b) a^{-1} (b a^2 b^{-1}) (b^{-1} a^k) a^{-3} a^{-m}
	\\&= \gamma a^4 \gamma a^{m-2} \gamma a^{-1} a^2  \gamma^{-1} a^{-3} a^{-m}
	\\&= \gamma^4
		\qquad   \text{(since $a$ inverts~$\gamma$ and $m$~is even)}
	.\end{align*}
This is a generator of~$\ZZ_p$. Therefore, $(\voltage C_0)^{-1} (\voltage C_1)$  is also a nontrivial element of~$\ZZ_p$, since it is conjugate to the inverse of $(\voltage C_0) (\voltage C_1)^{-1}$.
\end{aid}

\begin{aid} \label{Ga>2}
	$$ \includegraphics{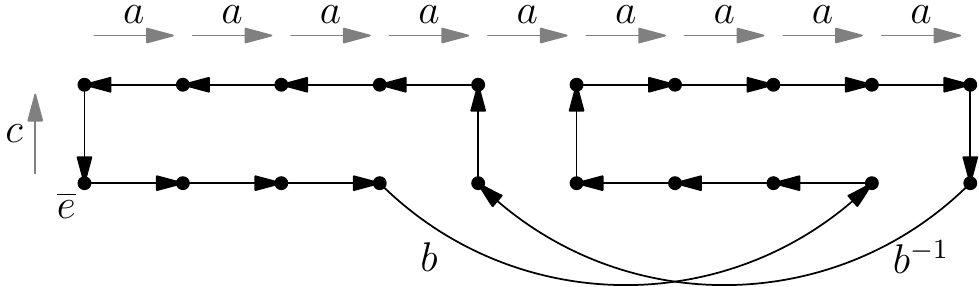} $$
%
%
%
%
%
%
%
%
%
%
%
%
%
\end{aid}

\begin{aid} \label{Ga>2replace}
	$$ \includegraphics{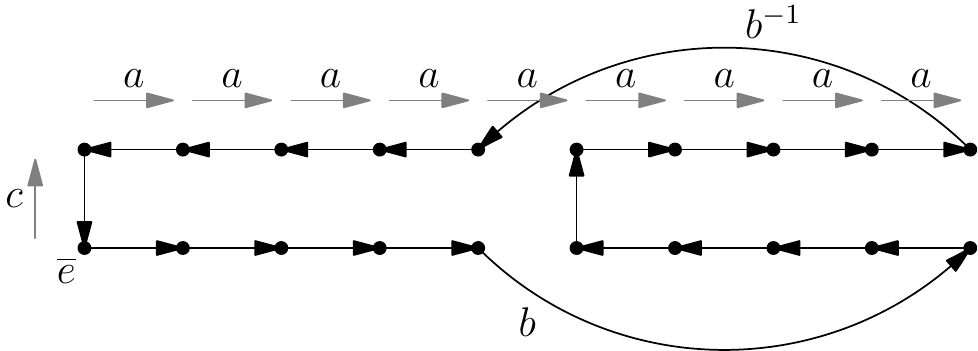} $$
%
%
%
%
%
%
%
%
%
%
%
%
%
\end{aid}

\begin{aid} \label{Ga>2voltage}
We have $[a,b] \in \ZZ_p$ and $[a^{k-1}, G] \subseteq \ZZ_p$ (since $k - 1$ is even). Therefore
	$$ \voltage C 
	=  a^{k-2}  b  a^{-(k-2)} c a^{k-1} c^{-1} b^{-1} c a^{-(k-1)} c^{-1}
	\equiv  b  c  c^{-1} b^{-1} c c^{-1}
	= e
	\pmod{\ZZ_p} $$
and
	$$ \voltage C'
	= a^{k-1} b a^{-(k-1)} c a^{k-1} b^{-1} a^{-(k-1)} c^{-1}
	\equiv b  c b^{-1} c^{-1}
	= [b^{-1}, c^{-1}]
	\equiv [a^k, c]
	\not\equiv e
	\pmod{\ZZ_p} .$$
So $(\voltage C)^{-1}(\voltage C') \notin \ZZ_p$.
\end{aid}

\begin{aid} \label{acisZ2}
Since $\quot b \in \langle \quot a \rangle$, we see from the contrapositive of \cref{Cent->Divides} that $\ZZ_2 \nsubseteq \langle [a,b] \rangle$. Therefore, we must have $\ZZ_2 \subseteq \langle [a,c] \rangle$.
On the other hand, since $\langle \quot a, \quot c \rangle = \quot G$, but $\{a,c\} \neq S$, we see from the contrapositive of \cref{Slessamin} that $\ZZ_p \nsubseteq \langle [a, c] \rangle$. Therefore $[a,c]$ is the nontrivial element of~$\ZZ_2$.
\end{aid}

\begin{aid} \label{acInterchange}
Suppose $c$ neither centralizes~$G'$ nor inverts~$G'$. Since $c$ does not centralize~$G'$, we know that it does not centralize~$\ZZ_p$, so there exists $x \in \{a,b\}$, such that $\ZZ_p \subseteq \langle [c,x] \rangle$. 

If $\quot x \in \langle \quot c \rangle$, then we may apply \cref{binaNotinvert} with $c$ and~$x$ in the roles of~$a$ and~$b$.

Suppose, now, that $\quot x \notin \langle \quot c \rangle$. Since $c$ neither centralizes nor inverts~$G'$, we know $|\quot c| > 2$. So one of the cases of \cref{a>2+bNotinaSect} applies with $c$ and~$x$ in the roles of~$a$ and~$b$.
\end{aid}

\begin{aid} \label{abkham}
The path $(a,b)^k$ is a hamiltonian cycle in $\langle \quot a \rangle$:
	$$ \includegraphics[scale=0.3]{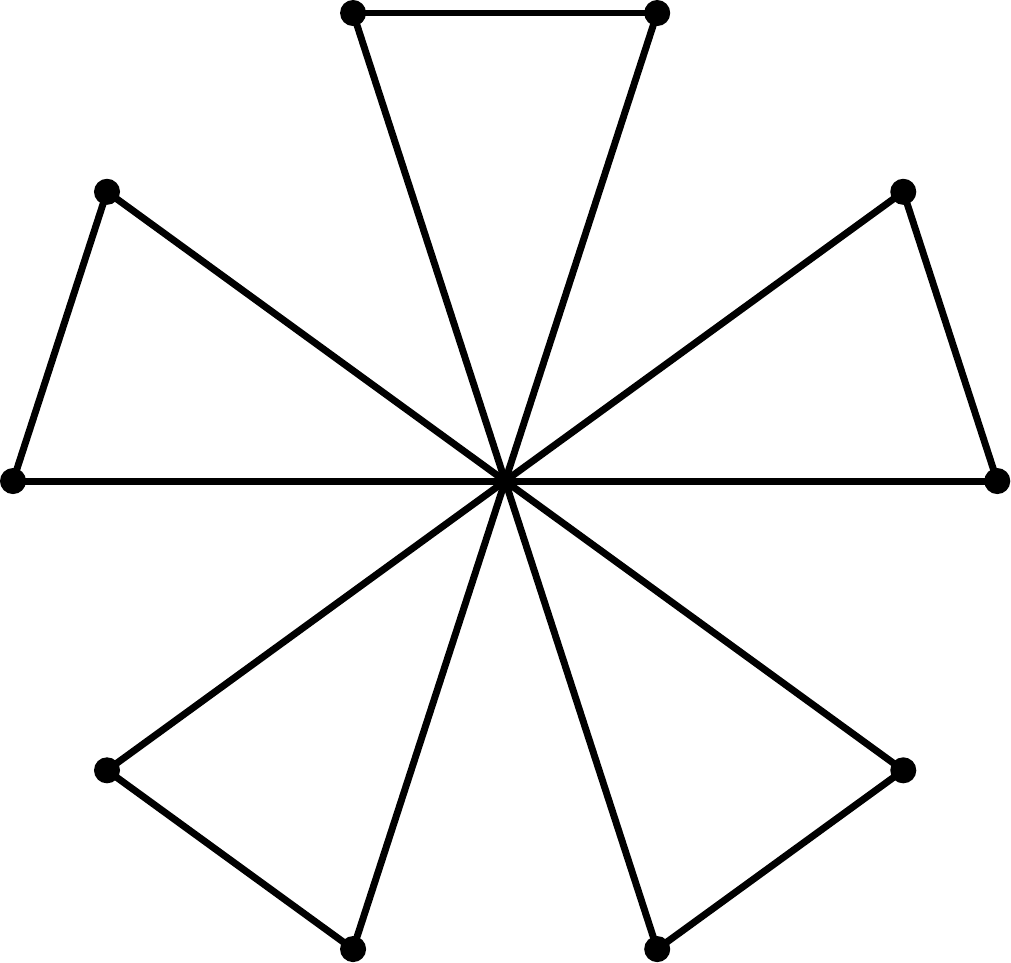} $$
Removing a single edge from this hamiltonian cycle yields the hamiltonian path~$L$.
%
%
%
%
%
\end{aid}

\begin{aid} \label{abk=gammak}
Since $k + 1$ is even (and $a$ inverts~$G'$), we know that $a^{k+1}$ centralizes~$\gamma$.  So
	$$ (ab)^k = \bigl(a \, (a^k \gamma) \bigr)^k = (a^{k+1}\gamma)^k = a^{k(k+1)} \gamma^k 
	= (a^m)^{(k+1)/2} \gamma^k = e^{(k+1)/2} \gamma^k = \gamma^k .$$
\end{aid}

\begin{aid} \label{voltage(abk)}
If $p \mid k$, then
	$$ \voltage L = (ab)^k b^{-1} = \gamma^k \cdot \gamma^{-1} a^{-k} = \gamma^{k-1} a^{-k} .$$
If $p \nmid k$, then
	$$ \voltage L = (ab)^k a^{-1} = \gamma^k a^{-1} .$$
\end{aid}

\begin{aid} \label{k>5voltage}
Let $C_1$ be the hamiltonian cycle obtained from~$C_0$ by applying \cref{StandardAlteration} to $[\quot e](b,a,b)$ and $[\quot a \quot b \quot c] (a^{-1})$ (so $s = b$, $t = a$, $u = c$, and $h = bc$). Then
	$$ \bigl( (\voltage C_0)^{-1}(\voltage C_1) \bigr)^{bc}
	= [t^{-1}, u] \, [s,t^{-1}]^u
	= [a^{-1}, c] [b,a^{-1}]^c .$$
Since $b$ inverts~$G'$, but $c$~centralizes~$G'$, this tells us
	$$ (\voltage C_0)^{-1}(\voltage C_1) = \bigl( [a^{-1}, c] [b,a^{-1}] \bigr)^{-1} = [c,a^{-1}] \, [a^{-1}, b] .$$

Also, since $C_2$ is obtained from~$C_1$ by applying \cref{StandardAlteration} to the path $[\quot a^2](b,a,b)$ and the edge $[\quot a^3 \quot b \quot c] (a^{-1})$ (so $s = b$, $t = a$, $u = c$, and $h = a^2 bc$), we have
	$$ \bigl( (\voltage C_1)^{-1}(\voltage C_2) \bigr)^{a^2bc}
	= [t^{-1}, u] \, [s,t^{-1}]^u
	= [a^{-1}, c] [b, a^{-1}]^c .$$
Since $a$ and $b$ invert~$G'$, but $c$~centralizes~$G'$, this tells us
	$$ (\voltage C_1)^{-1}(\voltage C_2)= \bigl( [a^{-1}, c] [b,a^{-1}] \bigr)^{-1} = [c,a^{-1}] \, [a^{-1},b] .$$

Putting these two calculations together tells us
	\begin{align*}
	 (\voltage C_0)^{-1}(\voltage C_2)
	&= \bigl( (\voltage C_0)^{-1}(\voltage C_1) \bigr) \bigl( (\voltage C_1)^{-1}(\voltage C_2) \bigr)
	\\&= \bigl( [c,a^{-1}] \, [a^{-1}, b] \bigr) \bigl( [c,a^{-1}] \, [a^{-1},b] \bigr)
	\\&= [c,a^{-1}]^2 \, [a^{-1}, b]^2
	\\&= [a^{-1}, b]^2
		&& \text{($[a, c] \in \ZZ_2$)}
	. \end{align*}
\end{aid}

\begin{aid} \label{k>5voltageC0}
Since $c$ centralizes~$G'$, the map $x \mapsto [x,c]$ is a homomorphism to~$\ZZ_2$ whose kernel contains~$G'$. Therefore
	\begin{align*}
	\voltage C_0 
	&= \bigl( (ba)^k a^{-1} \bigr)  c  \bigl( (ba)^k a^{-1} \bigr)^{-1} c^{-1}
	\\&= \bigl[ \bigl( (ba)^k a^{-1} \bigr)^{-1} , c^{-1} \bigr]
	\\&= [ a^{k(k+1) - 1} \gamma^k, c^{-1} ]
	\\&= [ a^{-1}, c^{-1} ]
	&& \text{($k(k+1)$ is even and $[\gamma,c] = e$)}
	\\&= [ a, c ]
	. \end{align*}
\end{aid}

\begin{aid} \label{k=3notZp}
	$$ \includegraphics{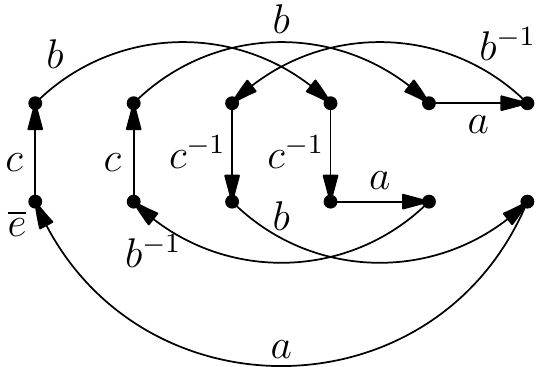} $$
\end{aid}

\begin{aid} \label{voltage(k=3notZp)}
Let $z$ be the nontrivial element of~$\ZZ_2$, so $[a,c] = [b,c] = z$. Then
	\begin{align*}
	\voltage C
		&= c b c^{-1} \cdot a b^{-1}  \cdot c b a b^{-1} c^{-1} \cdot b a
		\\&= {}^c\!b \cdot a b^{-1} \cdot  {}^c\!(b a b^{-1}) \cdot b a
		\\&= bz \cdot a b^{-1} \cdot (b a b^{-1})z \cdot b a
		\\&= b \cdot a b^{-1} \cdot (b a b^{-1}) \cdot b a \cdot z^2
			&& \text{($\ZZ_2$ is in the center of~$G$)}
		\\&= b  a^3
			&& \text{($z^2 = e$ because $z \in \ZZ_2$)}
		\\&= (a^k \gamma)  a^k
			&& \text{($k = 3$)}
		\\&= a^{2k} \gamma^{-1}
			&& \text{($a$ inverts~$G$ and $k = 3$ is odd)}
		\\&= \gamma^{-1}
			&& \text{($a^{2k} = a^m = e$)}
	. \end{align*}
\end{aid}

\begin{aid} \label{p>5aid}
	$$ \includegraphics{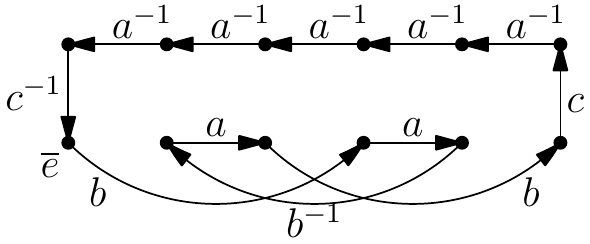} $$
\end{aid}

\begin{aid} \label{p>5voltage}
We have
	\begin{align*}
	\voltage C 
		&= (b a b^{-1} a b) (c a^{-5} c^{-1})
		\\& = (b a b^{-1} a b a) (a^{-1} c a c^{-1})
			&& \text{($a^6 = a^{2k} = e$, so $a^{-5} = a$)}
		\\& = (b a b^{-1} a b a) [a,c^{-1}] 
		. \end{align*}
Since $[a,c] \in \ZZ_2$, we have $[a,c^{-1}]  = [a,c]$. Also, we have
	\begin{align*}
	b a b^{-1} a b a 
	&= (a^3 \gamma) a (\gamma^{-1} a^{-3}) a (a^3 \gamma) a
	\\&= a^{3 + 1 -3 + 1 + 3 + 1} \gamma^{-3}
		&& \text{($a$~inverts~$\gamma$)}
	\\&= \gamma^{-3}
			&& \text{($a^6 = a^{2k} = e$)}
	. \end{align*}
Therefore $\voltage C = \gamma^{-3} [a,c]$.
\end{aid}

\begin{aid} \label{p=3aid}
Let $z = [a,c]$ be the nontrivial element of~$\ZZ_2$. Then every element of $\widehat G = G/\ZZ_p$ can be written uniquely in the form $a^i c^j w^k$ with $0 \le i \le 5$ and $j,k \in \{0,1\}$. The hamiltonian cycle~$C$ visits the vertices of  $\Cay(\widehat G;S)$ in the following order: 
$$\begin{array}{ccccccccccc}
&&e\bya{a}\bya{a^{2}}\byc{a^{2}c} \\
\bya{a^{3}cz}\bya{a^{4}c}\bya{a^{5}cz}\bya{c} \\
\bya{acz}\byc{az}\byai{z}\byai{a^{5}z} \\
\byb{a^{2}z}\bya{a^{3}z}\bya{a^{4}z}\byc{a^{4}cz} \\
\byai{a^{3}c}\byai{a^{2}cz}\byai{ac}\byai{cz} \\
\byai{a^{5}c}\byc{a^{5}}\byai{a^{4}}\byai{a^{3}} \\
\byb{e}
\end{array}$$
\end{aid}

\begin{aid} \label{k=2aid}
	$$ \includegraphics{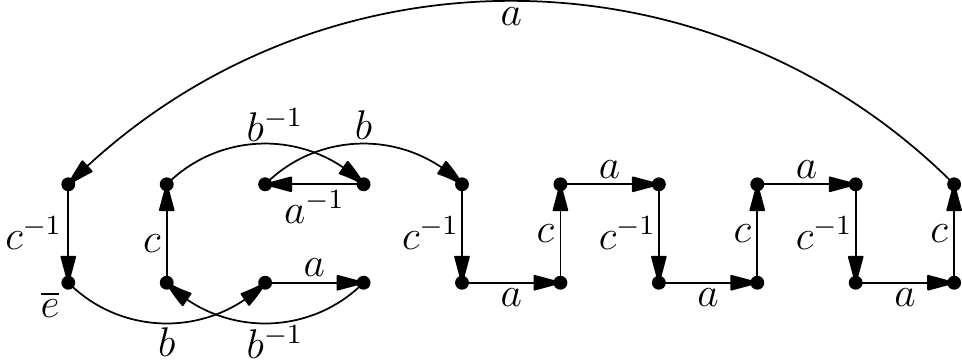} $$
\end{aid}

 \begin{aid} \label{k=2Ga>2voltage}
 \Cref{StandardAlteration} tells us
 	$$ \bigl( (\voltage C_0)^{-1}(\voltage C_1) \bigr)^{b^2}
		= [u,t^{-1}] \, [s, t^{-1}]^u
		= [b, a^{-1}] [b, a^{-1}]^b .$$
Since $b$ centralizes~$G'$, we have 
	$$\bigl( (\voltage C_0)^{-1}(\voltage C_1) \bigr)^{b^2} = (\voltage C_0)^{-1}(\voltage C_1)$$
and 
	$$[b, a^{-1}]^b = [b, a^{-1}] = [b, a]^{-1} = [a,b].$$
Therefore 
	$ (\voltage C_0)^{-1}(\voltage C_1) = [a,b]^2$.
 \end{aid}

\begin{aid} \label{Ga=2}
	$$ \includegraphics{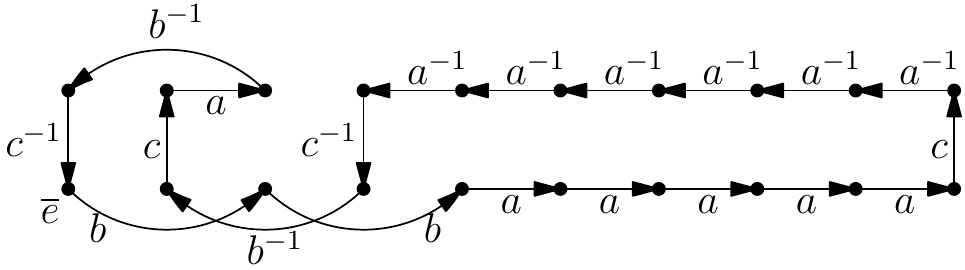} $$
\end{aid}

\begin{aid} \label{Ga=2modp}
Calculating modulo~$\ZZ_p$, we have
	\begin{align*}
	\voltage C 
	&= b^2 (a^{m-5} c a^{-(m-4)} c^{-1}) (b^{-1} c a b^{-1} c^{-1})
	\\&\equiv b^2 (a^{-1}) (b^{-1} a b^{-1} [a,c] )
		&& \begin{pmatrix}
		\text{$m-4$ is even, but $1-k = -1$ is odd,}
		\\ \text{so $c$ commutes with $a^{-(m-4)}$}
		\\ \text{but not with $ab^{-1}$ \quad (mod~$\ZZ_p$)}
		\end{pmatrix}
	\\&\equiv [a,c] 
		&& \bigl(\text{$a$ commutes with~$b$ \  (mod~$\ZZ_p$)}\bigr)
	. \end{align*}
\end{aid}

\begin{aid} \label{Ga=2mod2}
Calculating modulo~$\ZZ_2$, we have
	\begin{align*}
	 \voltage C 
	&= b^2 (a^{m-5} c a^{-(m-4)} c^{-1}) (b^{-1} c a b^{-1} c^{-1})
	\\&= (a^2 \gamma)^2 (a^{m-5} c a^{-(m-4)} c^{-1}) \bigl( (\gamma^{-1} a^{-2}) c a (\gamma^{-1} a^{-2}) c^{-1} \bigr)
		&& \text{($b = a^k \gamma = a^2 \gamma$)}
	\\&\equiv \gamma^2 a^{-1} \gamma^{-1} c a \gamma^{-1} c^{-1})
		&& \hskip-2cm
		\begin{pmatrix} \text{$a^2$ commutes with~$\gamma$ and}
		\\ \text{$a$ commutes with~$c$ $\pmod{\ZZ_2}$}
		\end{pmatrix}
	\\&\equiv \gamma^3 \cdot {}^{c}\!(\gamma^{-1})
		&& \hskip-2cm
		\begin{pmatrix} \text{$a$ inverts~$\gamma$ and}
		\\ \text{commutes with~$c$ $\pmod{\ZZ_2}$}
		\end{pmatrix}
	\\&= \gamma^3 \cdot \gamma^{\pm1}
		&& \hskip-2cm
		\text{($c$ either centralizes or inverts~$G'$)}
	\\&\in \{\gamma^2, \gamma^4\} 
	. \end{align*}
\end{aid}

\begin{aid} \label{BannaiThm}
The main result of \cite{Bannai-HamCycGenPet} states that the generalized Petersen graph $GP(m,k)$ has a hamiltonian cycle if $\gcd(m,k) = 1$ and 
	$$ \text{either \ $m \not\equiv 5 \pmod{6}$ \ or \ $k \notin \{2, (m-1)/2, (m+1)/2, m-2\}$} .$$
(More generally, see \cite[Thm.~3]{AlspachGPHam} for a complete determination of which generalized Petersen graphs  have hamiltonian cycles.)

$\Cay(G;a,b)$ is the generalized Petersen graph $GP(2n, k)$, where $b^a = b^k$ and $1 \le k < 2n$. We must have $\gcd(2n,k) = 1$ (because $b^k$, like~$b$, must generate $\langle b \rangle$), and it is obvious that $2n \not\equiv 5 \pmod{6}$ (since $2n$~is even), so Bannai's theorem provides a hamiltonian cycle in $\Cay(G;a,b)$.

\vbox{ 

}
\end{aid}

\begin{aid} \label{Z2notinbaid}
For the reader's convenience, we translate the calculations in the last paragraph of Case~1 of the proof of \cite[Prop.~6.1]{CurranMorrisMorris-16p} into our notation.

Let $w$ be the nontrivial element of~$\ZZ_2$. Then every element of $\ul G$ can be written uniquely in the form $\ul a^i \ul b^j \ul w^k$ with $i,j,k \in \{0,1\}$. To see that $\bigl( (a,b)^4\#, b^{-1} \bigr)$  is a hamiltonian cycle, note that it visits the vertices of $\Cay(\ul G, a,b)$ in the following order:
	$$ \ul e, \ \ul a, \ \ul a \ul b, \  \ul b \ul w, \  \ul w, \  \ul a \ul w, \  \ul a \ul b \ul w, \  \ul b, \  \ul e .$$
Since 
	$$ (ab)^2 = b(b^{-1}aba) b = b [b,a] b = [a,b]^{-1} b^2 = (-1,-2,0) \cdot (0,2,2) = (-1,0,2) ,$$
the voltage of this hamiltonian cycle is
	$$ (ab)^4 b^{-2} 
	= \bigl( (ab)^2 \bigr)^2 b^{-2} 
	=  (-1,0,2)^2 \cdot b^{-2}
	= (0,0,4) \cdot (0,-2,-2)
	= (0,-2,2)
	. $$
\end{aid}

\begin{aid} \label{a=2+b=3L}
	$$ \includegraphics{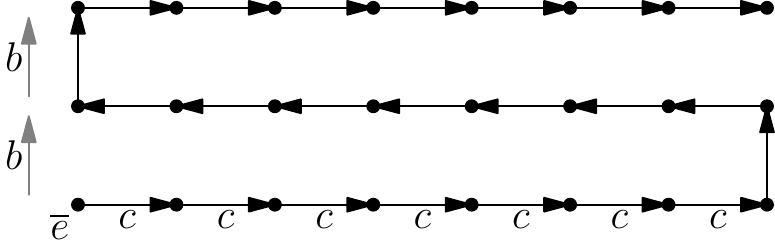} $$
%
%
%
%
%
%
%
%
%
%
%
\end{aid}

\begin{aid} \label{a=2+b=3LaLa}
Calculating modulo~$\ZZ_p$, we have
	\begin{align*}
	\voltage C
	&= (c^{\ell-1} b c^{-(\ell-1)} b c^{\ell-1}) \, a \, (c^{\ell-1} b c^{-(\ell-1)} b c^{\ell-1})^{-1} \, a
	\\&= [ (c^{\ell-1} b c^{-(\ell-1)} b  c^{\ell-1})^{-1}, a ]
	\\&\equiv [c,a]^{-(\ell-1)} [b,a]^{-1} [c,a]^{\ell-1} [b,a]^{-1} [c, a]^{-(\ell-1)}
		&& \text{($G'/\ZZ_p$ is in the center of~$G/\ZZ_p$)}
	\\&\equiv [c,a] \,  [b,a] \, [c,a] \, [b,a] \, [c, a]
		&& \text{($\ell - 1$ and $-1$ are odd)}
	\\&\equiv [c,a] 
	. \end{align*}
This is nontrivial (mod~$\ZZ_p$), so $\voltage C$ projects nontrivially to~$\ZZ_2$.
\end{aid}

\begin{aid} \label{a=2+b=3C0}
	$$ \includegraphics{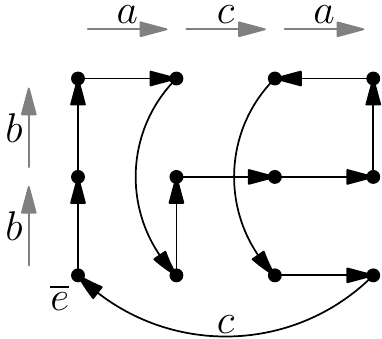} $$
\end{aid}

\begin{aid} \label{a=2+b=3Z2}
Calculating modulo~$\ZZ_p$, we have
	\begin{align*}
	\voltage C_0
	&= b^2 a b^2 c a b a b  a c
	\\&= b^6 a c a c
		&& \text{($[a,b]$ and $[b,c]$ are in~$\ZZ_p$, and $a^2 = e$)}
	\\&= e \cdot [a,c]
		&& \text{($|\quot b| = 3 \Rightarrow b^3 \in \ZZ_2 \pmod{\ZZ_p}$)}
	\\&= [a,c]
	.\end{align*}
This is nontrivial (mod~$\ZZ_p$).	
\end{aid}

\begin{aid} \label{a=c=2L}
	$$ \raise 1cm\hbox{$L$:} \qquad \includegraphics{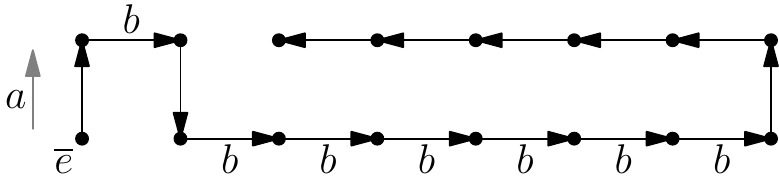} $$
%
%
%
%
%
%
%
%
%
%
\end{aid}

 \begin{aid} \label{a=c=2Cv}
 Calculating modulo~$\ZZ_p$, we have
 	\begin{align*}
	\voltage C
	&= (aba b^{n-2} a b^{-(n-3)}) \, c \, (aba b^{n-2} a b^{-(n-3)})^{-1} \, c^{-1}
	\\&= [(aba b^{n-2} a b^{-(n-3)})^{-1}, c^{-1}]
	\\&\equiv [aba b^{n-2} a b^{-(n-3)}, c]
		&& \hskip-1cm\text{(\cref{Cents->Homo} implies $[x^i,y^j] \equiv [x,y]^{ij}$)}
	\\& \equiv [a, c] \, [b,c]\,  [a,c] \, [b,c]^{n-2} [a,c] \,  [b, c]^{-(n-3)}
		&& \hskip-1cm\text{(\cref{Cents->Homo})}
	\\& = [a,c]^3 [b,c]^2 
		&& \hskip-1cm\text{($G'$ is abelian)}
	\\& \equiv [a,c] 
		&& \hskip-1cm\text{($3$ is odd and $2$ is even)}
	. \end{align*}
 \end{aid}

 \begin{aid} \label{a=c=2C'Z2}
Since $\voltage C \in \ZZ_p$, \cref{StandardAlteration} tells us that
	$$ \voltage C' \equiv [b^{-1}, c] [a, b^{-1}]^c \equiv [b,c] [a,b] \quad \pmod{\ZZ_p} .$$
By assumption, this projects nontrivially to~$\ZZ_2$.
 \end{aid}

\begin{aid} \label{bothinC0}
	$$ \includegraphics{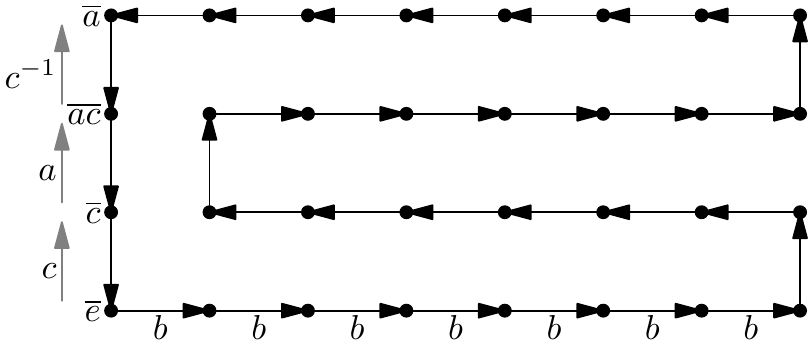} $$
%
%
%
%
%
%
%
%
%
%
\end{aid}

\begin{aid} \label{bothinCstar}
	$$ \includegraphics{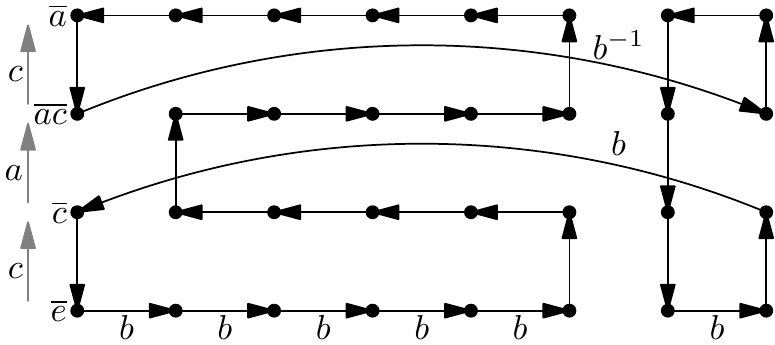} $$
\end{aid}

\begin{aid} \label{acb=a}
Write $a^{cb} = a \gamma$, with $\gamma \in G'$. Then
	$$(b^{-1}c)^2 a (cb)^2 
	=(cb)^{-2} a (cb)^2 
	= a^{(cb)^2}
	= \bigl( a^{cb} \bigr)^{cb}
	= \bigl( a \gamma \bigr)^{cb}
	= a^{cb} \, \gamma^{cb}
	= (a \gamma) \gamma^{-1} 
	= a .$$
\end{aid}

\begin{aid} \label{bNot3+lnot2Lodd}
	$$ \includegraphics{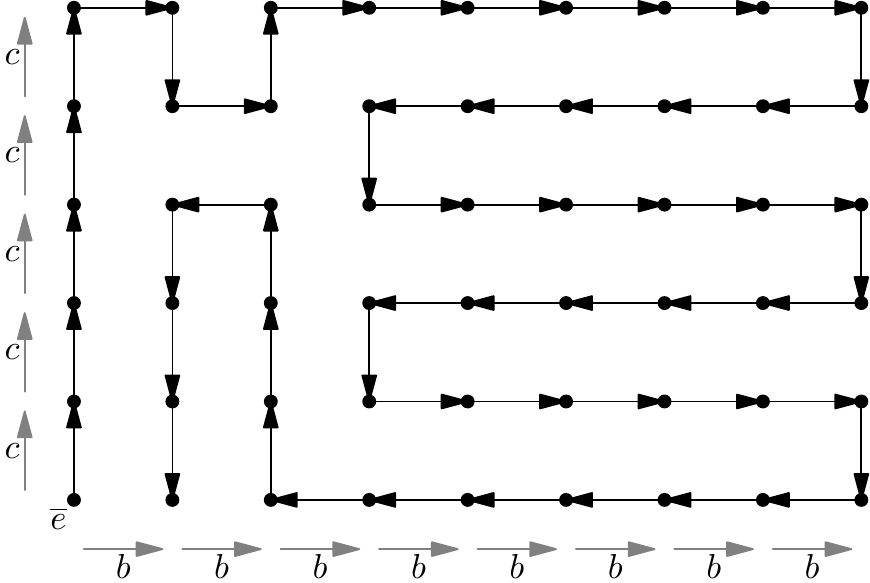} $$
\end{aid}

\begin{aid} \label{bNot3+lnot2Leven}
	$$ \includegraphics{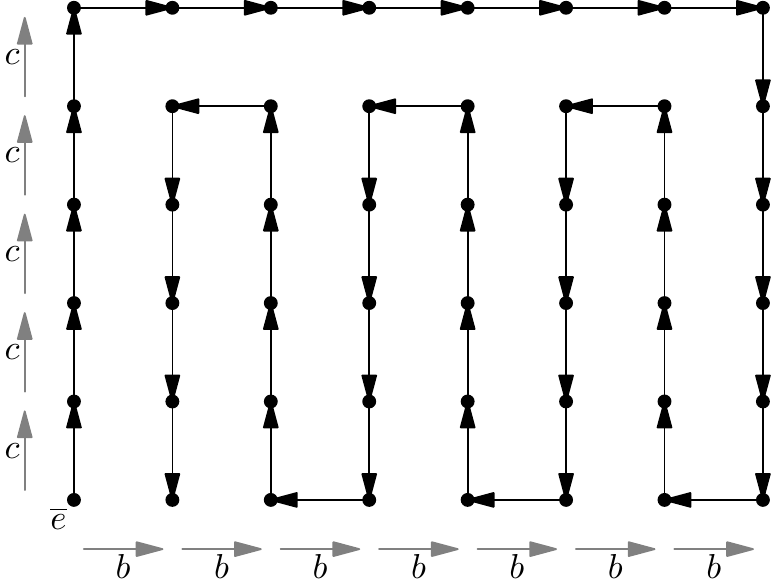} $$
\end{aid}

\begin{aid} \label{atransform}
Let $C'$ be the hamiltonian cycle obtained by applying the $a$-transform. To show that the $a$-transform multiplies the voltage by~$\gamma_a$, we wish to show 
	$$(\voltage C)^{-1} (\voltage C') = [a,b^{-1}]  [c,a] .$$
\Cref{StandardAlteration} tells us
	\begin{align*}
	\bigl( (\voltage C)^{-1} (\voltage C') \bigr)^{ab}
	&= [t^{-1}, u] \, [s, t^{-1}]^u
	= [a^{-1}, b] [c^{-1}, a^{-1}]^b
	= [a,b] [c^{-1}, a]^b
	, \end{align*}
because $a^{-1} = a$.
So
	$$
	(\voltage C)^{-1} (\voltage C')
	= \bigl( [a,b] [c^{-1}, a]^b \bigr)^{(ab)^{-1}}
	= [a,b]^{b^{-1} a} \  [c^{-1}, a]^a
	= [a,b^{-1}]  \, [c,a]
	, $$
because
	$$ [a,b]^{b^{-1} a} 
	= (ab) (a b^{-1} ab)(b^{-1} a)
	= abab^{-1}
	= [a, b^{-1}] $$
and 
	\begin{align*}
	[c^{-1}, a]^a 
	&= [a,c^{-1}] && \text{($a$ inverts $G'$)}
	\\&= [a,c]^{-1} && \text{($c$ centralizes $G'$)}
	\\&= [c,a]
	. \end{align*}
\end{aid}

\begin{aid} \label{btransform}
\Cref{StandardAlteration} tells us that the $b$-transform multiplies the voltage by a conjugate of
	\begin{align*}
	[t^{-1}, u] \, [s,t^{-1}]^u
	&= [b, a] \, [c^{-\epsilon}, b]^a
	= [b,a] \, [b,c^{-\epsilon}]
	= \gamma_b
	. \end{align*}
\end{aid}

\begin{aid} \label{bothZ2}
Let $z$ be the nontrivial element of~$\ZZ_2$, and write $L = (s_i)_{i=1}^r$, so $s_i \in \{b^{\pm1}, c^{\pm1}$ and $r$~is odd. Then, calculating mod~$\ZZ_p$, we have $[s_i,a] \equiv z$ for all~$i$, so
	$$ \voltage C
	= \bigl( \prod\nolimits_{i=1}^r s_i \bigr) \, a \, \bigl( \prod\nolimits_{i=1}^r s_i \bigr)^{-1} \, a
	= \bigl[ \bigl( \prod\nolimits_{i=1}^r s_i \bigr)^{-1} , a \bigr]
	\equiv \prod\nolimits_{i=1}^r [s_i,a]^{-1}
	\equiv \prod\nolimits_{i=1}^r z
	= z^r
	= z .$$
\end{aid}

\begin{aid} \label{neitherC}
Since $b$ and $c$ centralize~$G'$, \cref{Cents->Homo} implies that if we let $\varphi(x) = [x,a]$, then $\varphi$ is a homomorphism from $\langle b,c \rangle$ to~$G'$. Note that, since the image of~$\varphi$ is a subgroup of the abelian group~$G'$, the kernel of~$\varphi$ must contain $\langle b,c \rangle'$.

Write $L = (s_i)_{i=1}^r$. Then 
	$\voltage C
	= \bigl[ \bigl(\prod\nolimits_{i=1}^r s_i)^{-1}, a]$.
Since the sum of the exponents of the occurrences of~$b$ in~$L$ is~$1$, and the sum of the exponents of the occurrences of~$c$ is~$0$, we know $\prod\nolimits_{i=1}^r s_i \equiv b \pmod{\langle b,c \rangle'}$. So the preceding paragraph tells us that 
	$$\bigl[ \bigl(\prod\nolimits_{i=1}^r s_i)^{-1}, a] = [b,a]^{-1} = [a,b] .$$
\end{aid}

\begin{aid} \label{L0}
	$$ \includegraphics{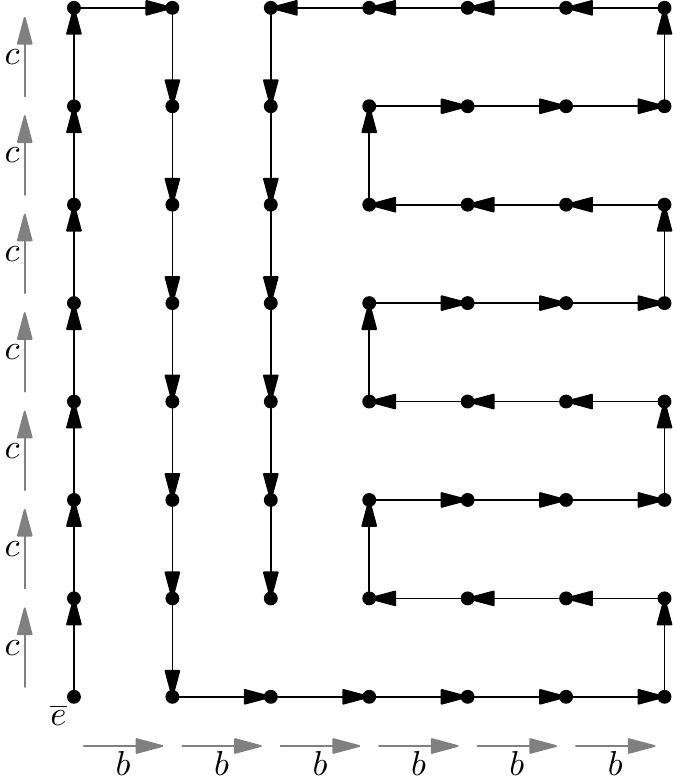} $$
\end{aid}


\begin{aid} \label{DpC}
	$$ \includegraphics{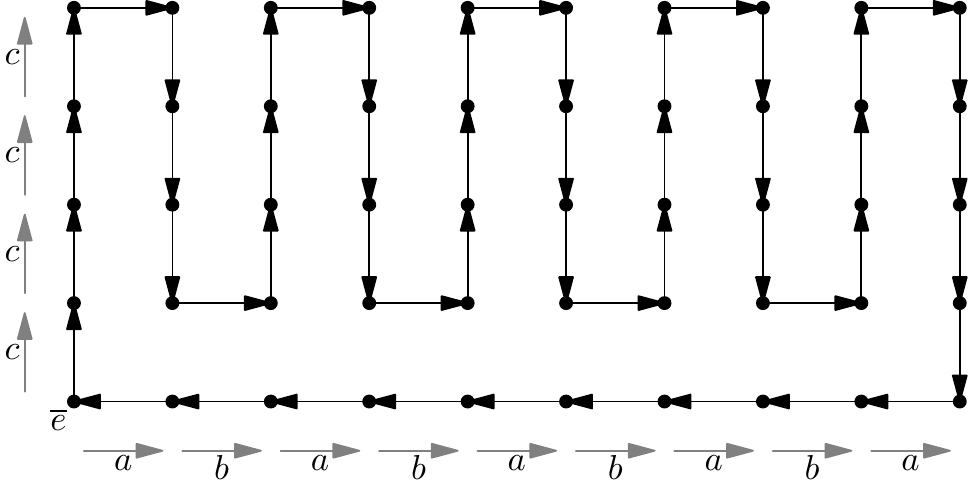} $$
%
%
%
%
%
%
%
%
%
%
\end{aid}

\begin{aid} \label{anotinb}
Suppose $\quot a \in \langle \quot b \rangle$. 

If $b$ does not centralize~$G'$, then \cref{S>3Assump} tells us that $|\quot b| = 2$, so we must have $\quot a = \quot b$, which contradicts the assumption that \cref{s=t} does not apply.

We now know that $b$ centralizes~$G'$. Since $G'$ is abelian (indeed, it is cyclic), this implies that $\langle b, G' \rangle$ is abelian. However, the fact that $\quot a \in \langle \quot b \rangle$ means that $a \in \langle b, G' \rangle$. Since $\langle b, G' \rangle$ is abelian, this implies that $a$ centralizes~$G'$. This contradicts~\pref{aDefn}.
\end{aid}

\begin{aid} \label{NoCommute}
Suppose $s,t \in S$, such that $s$ commutes with~$t$ (and $s \neq t$). There exist $x,y\in S$, such that $\ZZ_2 \subseteq \langle [x,y] \rangle$. Since $\{s,t\} \neq \{x,y\}$, we may assume $s \notin \{x,y\}$ (after interchanging $s$ and~$t$ if necessary). Since $t$ does not centralize~$G'$, there exists $u \in S$, such that $\ZZ_p \subseteq \langle [t,u] \rangle$. Then $\{x,y,t,u \} \subseteq S \sm \{s\}$, so
	$$ G' = \langle \ZZ_2, \ZZ_p \rangle \subseteq \langle [x,y], [t,u] \rangle \subseteq \langle S \sm \{s\} \rangle ,$$
so $\langle S \sm \{s\} \rangle \normal G$. Therefore \cref{Durnberger-commuting} applies
\end{aid}

\begin{aid} \label{cacb}
Note that $L_0 = (c,a,c)$ is a hamiltonian path in 
$\Cay \bigl( \langle \quot a, \quot c \rangle ; \{a,c\} \bigr)$
(because we have $\langle \quot a, \quot c, \rangle \iso \ZZ_2 \times \ZZ_2$). Therefore
	$$ L_1 = (c,a,c,b)^2\# = (L_0, b, L_0^{-1}) $$
is a hamiltonian path in $\Cay \bigl( \langle \quot a, \quot b, \quot c \rangle ; \{a,b,c\} \bigr)$.
So 
	$$ C = \bigl( (c,a,c,b)^2\#, d \bigr)^2 = (L_1, d, L_1^{-1}, d) $$
 is a hamiltonian cycle in $\Cay \bigl( \langle \quot a, \quot b, \quot c, \quot d \rangle ; \{a,b,c,d\} \bigr)$.
\end{aid}

\begin{aid} \label{nocentvoltage1}
\Cref{StandardAlteration} tells us
	$$ \bigl( (\voltage C)^{-1}(\voltage C') \bigr)^{bc}
	= [t^{-1}, u] \, [s,t^{-1}]^u
	= [a, b] \, [c, a]^b
	= [a,b] \, [a,c]
	= \gamma
	. $$
Since $b$ and~$c$ both invert~$G'$, we know that $bc$ centralizes~$G'$, so $(\voltage C)^{-1}(\voltage C') = \gamma$.
\end{aid}

\begin{aid} \label{nocentvoltage2}
This is exactly the same calculation as in \cref{nocentvoltage1}, except that $(\voltage C')^{-1}(\voltage C'')$ is conjugated by $cd$, instead of~$bc$. Since $cd$, like~$bc$, centralizes~$G'$, this change does not affect the result of the calculation at all, so, as before, the voltage is multiplied by~$\gamma$.
\end{aid}

\begin{aid} \label{S=4NotCentvoltage}
Note that 
	\begin{itemize}
	\item if $y$ centralizes~$G'$, then $[xy, z] = [x,z] \, [y,z]$ and $[y^{-1},z] = [y,z]^{-1}$, but
	\item if $y$ inverts~$G'$, then $[xy, z] = [x,z]^{-1} \, [y,z]$ and $[y^{-1}, z] = [y,z]$.
	\end{itemize}
Therefore
	\begin{align*}
	\voltage C
	&= \bigl( (cacb)^2 b^{-1} d \bigr)^2
	\\&= [(cacb)^2 b^{-1} ,d ]
	\\&= [(cacb)^2 , d ]^{-1} \,  [b ,d ]
		&& \text{($b$ inverts~$G'$)}
	\\&= \Bigl( [ca, d]^2 [cb , d ]^2 \Bigr)^{-1} \,  [b ,d ]
		&& \text{($cb$ and $cb$ centralize~$G'$)}
	\\&= \Bigl( \bigl( [c,d]^{-1} [a, d] \bigr)^2 \bigl( [c,d]^{-1} [b , d ] \bigr)^2 \Bigr)^{-1} \,  [b ,d ]
		&& \text{($a$ and~$b$ invert~$G'$)}
	\\&= [c,d]^4 [a,d]^{-2} [b,d]^{-1}
	\\&= [c,d]^4 [d,a]^2 [d,b]
	. \end{align*}
\end{aid}

\begin{aid} \label{a<>c}
We have $[d,a]^6 [d,b] \in \ZZ_2$. Assuming the same is true when we interchange $a$ and~$c$, we also have $[d,c]^6 [d,b] \in \ZZ_2$. So 
	$$[d,c]^6 \equiv [d,b]^{-1} \equiv [d,a]^6 \pmod{\ZZ_2} .$$
Since $p \neq 3$, this implies $[d,c] \equiv[d,a] \pmod{\ZZ_2}$.
\end{aid}


\begin{aid} \label{S=4asNotZpAbelga}
Suppose $[s,a] \notin \ZZ_2$, so $g = s$. By the definition of~$s$, we also know $[s,a] \notin \ZZ_p$. Therefore $[s,a]$ generates~$G'$.

Assume now that $[s,a] \in \ZZ_2$, so $g = sb^2$. Calculating modulo~$\ZZ_p$, we have
	\begin{align*}
	[sb^2, a] 
	&= [s,a] [b,a]^2
		&& \text{(by \cref{Cents->Homo}, since $G'$ is central in $G/\ZZ_p$)}
	\\&\equiv [s,a] 
	\\&\notin \ZZ_p
		&& \text{(definition of~$s$)}
	. \end{align*}
Calculating modulo~$\ZZ_2$, we have
	\begin{align*}
	[sb^2, a] 
	&\equiv [b^2, a]
		&& \text{($[s,a] \in \ZZ_2$)}
	\\&\not\equiv e
	. \end{align*}
so $[sb^2, a]$ generates~$G'$.
\end{aid}

 \begin{aid} \label{S=4asNotZpNonabelC0}
 Let $L$ be the following hamiltonian path in $\Cay \bigl( \langle \quot w, \quot x \rangle ; w,x \bigr)$:
$$ \raise3cm\hbox{$\displaystyle\frac{ |\langle \quot w, \quot x \rangle |}{| \langle  \quot w \rangle|}$ odd: } \includegraphics[scale=0.9]{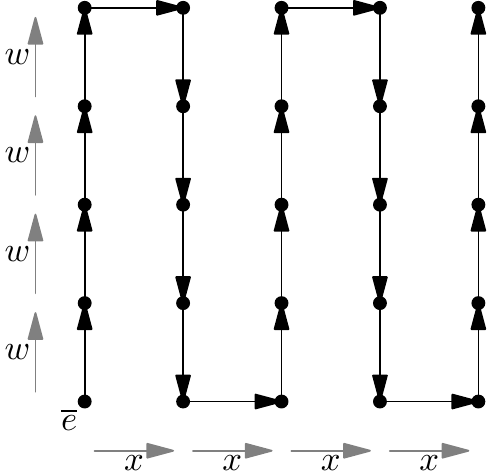} 
%
%
%
%
%
%
%
%
%
%
 \qquad 
 \raise3cm\hbox{$\displaystyle\frac{ |\langle \quot w, \quot x \rangle |}{| \langle  \quot w \rangle|}$ even: } \includegraphics[scale=0.9]{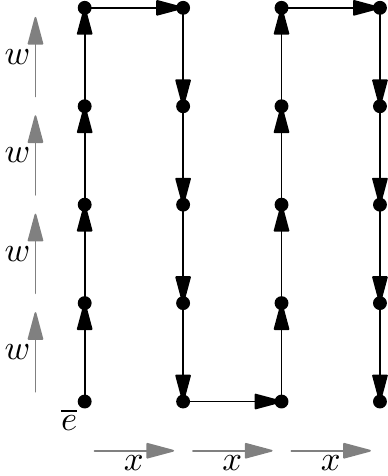} 
$$
%
%
%
%
%
%
%
%
%
%
Note that  $L$ contains the oriented path $[\quot w^{|\quot w | - 2} ](w, x, w^{-1})$. We can also write this path as $[\quot {h w^{-1} y^{-1}}](w, x, w^{-1})$, for $h = w^{|\quot w | - 1} y$.

Then $C' = (L, y, L^{-1}, y^{-1})$ is a cycle through the vertices of $\Cay \bigl( \langle \quot w, \quot x, \quot y \rangle ; w,x,y \bigr)$ that are in $\langle \quot w, \quot x \rangle \cup \quot y \langle \quot w, \quot x \rangle$. Therefore, if $C''$ is any hamiltonian cycle in $\Cay \bigl( \langle \quot w, \quot x \rangle ; w,x \bigr)$ that shares an edge with~$C'$, then an appropriate connected sum 
	$$ C_0 = C' \connsum_{t_1}^{s_1} -g_1C'' \connsum_{t_2}^{s_2} g_2C'' \connsum_{t_3}^{s_3} \cdots \connsum_{t_d}^{s_d} \pm g_dC'' $$
of $C'$ with translates of~$C''$ is a hamiltonian cycle in $\Cay(\quot G; S_0)$. 

Note that $C'$ contains both $[\quot {h  w^{-1}   y^{-1}}](w, x, w^{-1})$ and $[\quot {h x}](x^{-1})$. Therefore, if the edge used to form the first connected sum $C' \connsum_{t_1}^{s_1} -g_1C''$ is not in either of these paths, then $C_0$ also contains both of these paths. (For example, we could use the edge $[\quot y \quot w](w^{-1})$ to form the connected sum.) 
 \end{aid}

 \begin{aid} \label{not16}
 It suffices to show that $C_0$ has an edge labeled $s^{\pm1}$ that is not in either of the given paths, for we may assume that this oriented edge is of the form $[\quot {s^{\epsilon}}](s^{-\epsilon})$ (by replacing $C_0$ with a translate). 

 If $s \neq b$, then \cref{minmodab} implies there is at least one edge labeled $s^{\pm1}$. If $s \notin \{w^{\pm1}, x^{\pm1}\}$, then this edge is obviously not in either of the given paths (since all of the edges in those paths are labeled $w^{\pm1}, x^{\pm1}$). 

Therefore, we may assume $s \in \{ w^{\pm1}, x^{\pm1}, b \}$. To deal with this situation, we assume $C_0$ has been constructed as in \cref{S=4asNotZpNonabelC0} (from the path~$L$ and cycle~$C'$ that are specified there).

Suppose $s = w^{\pm1}$. The edges $[\quot{wy}](w^{-1})$ and $[\quot{xy}](w)$ are in~$C'$, but are not in either of the given paths. One of these edges may have been removed in forming the connected sum that defines~$C_0$, but the other will remain. So $C_0$ has at least one edge labeled~$w^{\pm1}$.

Suppose $s = x^{\pm1}$. 
	\begin{itemize}
	\item If $|\langle \quot w, \quot x \rangle/ \langle \quot w \rangle| > 2$, then $C'$ contains the edges $[\quot x](x)$ and $[\quot{x^2y}](x^{-1})$. At least one of these must be in~$C_0$ (and neither of these edges is in the given paths).
	\item If $|\quot G : \langle \quot w, \quot x \rangle| > 2$, then $C_0 \neq C'$ (in other words, $d > 1$ in the definition of~$C_0$). We may assume $C'$ has at least one edge labeled~$x^{\pm1}$, and that only edges labeled~$w^{\pm1}$ are used in forming the connected sums. Therefore, this edge labeled~$x^{\pm1}$ is in~$C_0$ (and it is not in the given paths).
	\item We may now assume $|\quot G : \langle \quot w \rangle| = 4$. Since \pref{xnots16} tells us that $|\quot G| > 16$ (and we know $|\quot a| = 2$), we must have $|\quot w| > 2$. Therefore, we have the following hamiltonian path~$\widetilde L$ in $\Cay(\langle \quot w, \quot x \rangle; w,x )$:
		$$ \includegraphics{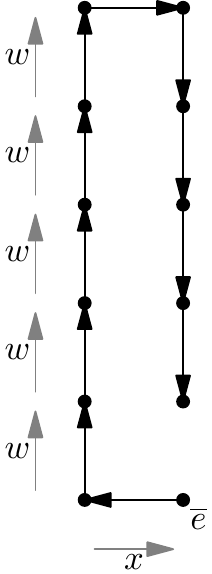} $$
%
%
%
%
%
%
%
%
%
%
	Use this path in the place of~$L$ to construct cycles~$\widetilde C'$ and~$\widetilde C_0$ analogous to $C'$ and~$C_0$ (and let $h = w^{|\quot w | - 1} x^{-1} y$). Note that $\widetilde L$ has two edges labelled $x^{\pm1}$, so $\widetilde C'$ has four edges labelled $x^{\pm1}$. Two of these are in the given paths, and one may be deleted in the construction of the connected sum, but at least one of these edges labelled $x^{\pm1}$ remains in~$\widetilde C_0$ and is not in either of the given paths.
	\end{itemize}

We may now assume $s = b \notin \{w,x\}$. 
From \cref{minmodab}, we know that $w \notin \langle \quot b \rangle$. Therefore, it is easy to construct a hamiltonian path $P = (t_i)_{i=1}^r$ in $\Cay \bigl( \langle \quot w, \quot b \rangle : b,w \bigr)$ that has at least one edge labeled $b^{\pm1}$, and such that $t_r = w$. Now, in place of the hamiltonian path
	$$ L = (w^r, x, w^{-r},  x, w^r, x, w^{-r},  x, \ldots) $$
that was used in \cref{S=4asNotZpNonabelC0}, use the hamiltonian path
	$$ \widetilde L = ( P, x, P^{-1},  x, P, x, P^{-1},  x,  \ldots) $$
  to construct cycles~$\widetilde C'$ and~$\widetilde C_0$ analogous to $C'$ and~$C_0$ (and let $h = y \prod_{i=1}^r t_i $).
  Then $C_0$ contains at least two edges labeled $b^{\pm1}$ (one from the first occurrence of~$P$ in~$L$, and another from the first occurrence of~$P^{-1}$). Neither of these is in the given paths (because $b \notin \{w,x\}$) and at least one of them remains in~$\widetilde C_0$.
 \end{aid}

\begin{aid} \label{anotS0G'}
Suppose $a \in \langle S_0 , G'\rangle$. Then
	$$ G 
	= \langle S \rangle
	= \langle S_0, a \rangle 
	\subseteq \langle S_0, G' \rangle 
	= \langle S_0, \ZZ_p, \ZZ_2 \rangle 
	\subseteq \langle S_0, \langle S_0 \rangle', \ZZ_2 \rangle 
	= \langle S_0, \ZZ_2 \rangle 
	,$$
so \cref{Z2inFrattini} implies $G = \langle S_0 \rangle$. This contradicts the irredundance of~$S$.
\end{aid}

 \begin{aid} \label{Ascendtnotb}
 If $b \notin \{x^{\pm1}, t^{\pm1}\}$, then \cref{minmodab} immediately implies that \pref{ascending} is satisfied. Then, since $b \notin \{ t^{\pm1} \} = \{ y^{\pm1} \}$, we may assume $b \in \{x^{\pm1} \}$. We wish to show $\langle \quot w \rangle \nsubseteq \langle \quot w, \quot b \rangle$.  In other words, we wish to show $\quot b \notin \langle \quot w \rangle$. 
 
 Suppose $\quot b \in \langle \quot w \rangle$.  Then, since $b$ does not centralize~$G'$, we know that $w$ does not centralize~$G'$, so \cref{S>3Assump} tells us that $|\quot w| = 2$. Since $\quot b \in \langle \quot w \rangle$, this implies that $\quot b = \quot w$, so \cref{s=t} applies.
 \end{aid}

\begin{aid} \label{S>3somenontrivL}
Since $\ZZ_p \subseteq \langle [a,b] \rangle$ and $b$ centralizes~$G'$, we know $|\quot b|$ is divisible by~$p$, so $|\quot b| > 2$. 

If $|\quot b|$ is even, we may let $L$ be a hamiltonian path of the following shape:
	$$ \includegraphics{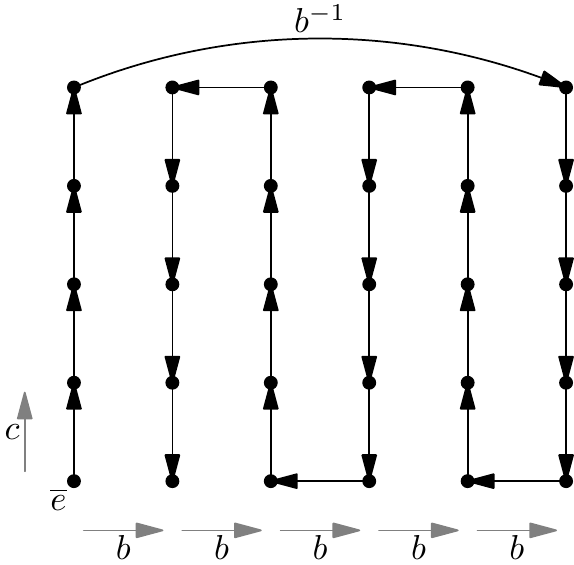} $$
%
%
%
%
%
%
%
%
%
%
%

Henceforth, we assume that $|\quot b|$ is odd. From \cref{Cent->Divides}, we see that $|\quot G : \langle \quot a, \quot b, \quot c \rangle|$ is even (so, in particular, $\langle \quot a, \quot b, \quot c \rangle \neq \quot G$). And \cref{Cent->Divides} also implies that $|\langle \quot a, \quot b, \quot c \rangle : \langle \quot a, \quot b \rangle|$ is even. Let $(t_i)_{i=1}^{\ell-1}$ be a hamiltonian path in $\Cay \bigl( \quot G / \langle \quot a, \quot b \rangle; S \sm \{a,b\} \bigr)$, such that $t_1 = c$ and $t_k = c^{\pm1}$ for some $k \notin \{1,2\}$. For example, $(t_i)_{i=1}^{\ell-1}$ could be of the following form:
	$$ \includegraphics{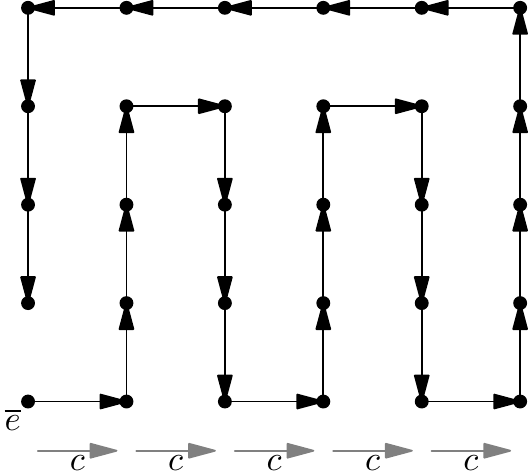} $$
%
%
%
%
%
%
%
%
%
%

If $|\quot b| > 3$, we may let $L$ be a hamiltonian path of the following shape:
	$$ \includegraphics{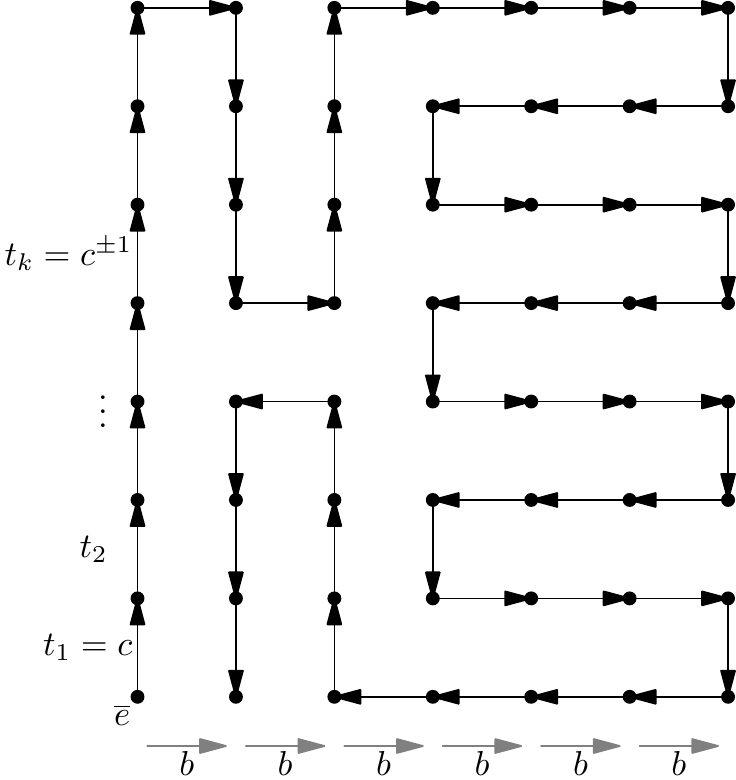} $$
%
%
%
%
%
%
%
%
%
%

Now assume $|\quot b| = 3$. The hamiltonian path~$(t_i)_{i=1}^{\ell-1}$ that is pictured above can be extended to a hamiltonian cycle~$(t_i)_{i=1}^{\ell}$, such that $\prod_{i=1}^{\ell} t_i \in G'$.
Now, we may let $L$ be the following hamiltonian path:
	$$ \includegraphics{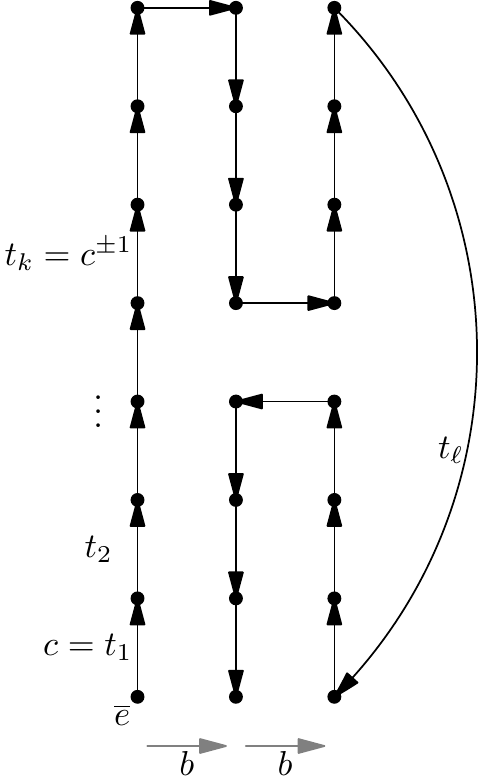} $$
%
%
%
%
%
%
%
%
%
%
%
\end{aid}

\begin{aid} \label{S>3somenontriva}
\Cref{StandardAlteration} tells us that
	$$ \bigl( (\voltage C)^{-1} (\voltage C') \bigr)^{ab}
	= [t^{-1}, u] \, [s, t^{-1}]^u
	= [a, b] \, [c^{-1}, a]^b .$$
Since $a$ inverts $G'$ and $b$~centralizes~$G'$, this implies
	$$ (\voltage C)^{-1} (\voltage C') = [b,a] \, [a,c^{-1}] = \gamma_a .$$
\end{aid}

\begin{aid} \label{S>3somenontrivb}
\Cref{StandardAlteration} tells us that the $b$-transform multiplies the voltage by a conjugate of
	\begin{align*}
	 [t^{-1}, u] \, [s, t^{-1}]^u
	&= [b^{-\delta}, a] \, [c^{-\epsilon}, b^{-\delta}]^a .
	\\&= [b^{-\delta}, a] \, [b^{-\delta}. c^{-\epsilon} ]
	\\&= \bigl( [b, a] \, [b. c^{-\epsilon} ] \bigr)^{-\delta}
	\\&= \bigl( [a,b,] \, [c^{-\epsilon}, b ] \bigr)^{\delta}
	\\&= (\gamma_b)^\delta
	. \end{align*}
Since $a$ inverts $G'$, and $\delta \in \{\pm1\}$, this is conjugate to~$\gamma_b$.
\end{aid}

%
%
%
%



\begin{aid} \label{4-bCentsG'C0}
Let $\widehat G = \quot G/ \langle \quot a \rangle$. We may let $C_0 = (a, L, a, L^{-1} )$, where $L = (t_i)_{i=1}^r$ is a hamiltonian path from~$\widehat e$ to~$\widehat b$ in $\Cay \bigl( \langle \widehat{S_0} \rangle; b,d \bigr)$, such that $t_1 = d$ (and $L$ contains at least one edge labeled $b^{\pm1}$).

Here is one way to construct such a hamiltonian path. 
We know that $\ZZ_2 \nsubseteq \langle a,b,c \rangle'$ (since $[s,t] \in \ZZ_p$ for all $s \in \{a,b\}$ and $t \in S$). On the other hand, $\ZZ_2 \subseteq \langle [c,d] \rangle$ (from the choice of $c$ and~$d$). So \cref{Cent->Divides} tells us that $|\langle \widehat b, \widehat d \rangle: \langle \widehat b \rangle|$ is even. (In fact, $|\langle \widehat b, \widehat c, \widehat d \rangle: \langle \widehat b, \widehat c \rangle|$ is even.) Therefore, we may let $L$ be a hamiltonian path of the following shape:
	$$ \includegraphics{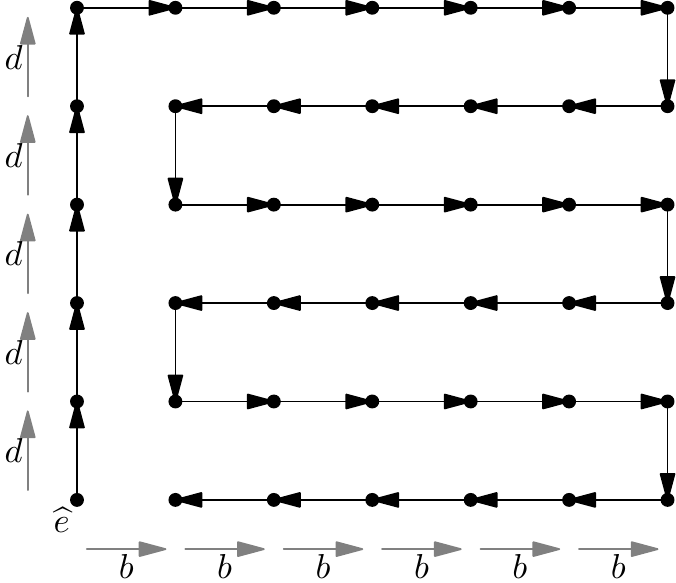} $$
%
%
%
%
%
%
%
%
%
%
\end{aid}

\begin{aid} \label{4-bCentsG'gamma}
\Cref{StandardAlteration} tells us that the voltage is multiplied by a conjugate of
	$$ [t^{-1}, u] \, [s,t^{-1}]^u
	= [a, b] \, [d^{-1}, a]^b
	= [a,b] \, [d^{-1},a] 
	= \gamma.$$
\end{aid}

%
%
%
%

 \begin{aid} \label{abinZpnotZ2C0}
 Let $L = (t_i)_{i=1}^r$ be a hamiltonian path in $\Cay(\quot H / \langle \quot a,b \rangle; S \sm \{a,b,c\})$, such that $t_1 = d$. Then the desired hamiltonian cycle is $( b, L, a, L^{-1}, b , L, a, L^{-1})$.
 \end{aid}

 \begin{aid} \label{abinZpnotZ2voltage}
 \Cref{StandardAlteration} tells us that the voltage is multiplied by a conjugate of
 	$$ [t^{-1}, u] \, [s,t^{-1}]^u
	= [b,a] \, [d^{-1}, b]^a
	= [b,a] \, [b,d^{-1}] ,$$
since $a$~inverts~$G'$.
 \end{aid}

\vfill\vfill
\end{appendix}
}


\begin{thebibliography}{99}
\addcontentsline{toc}{section}{\protect\numberline{}\refname} 

\normalsize 
\raggedright 


\bibitem{Bannai-HamCycGenPet}
K.\,Bannai:
Hamiltonian cycles in generalized Petersen graphs,
\emph{J.~Comb.\ Th.~B} 24 (1978) 181--188.
\MR{0505806}%
\AllowBreakBeforeURL{http://dx.doi.org/10.1016/0095-8956(78)90019-9}

\bibitem{ChenQuimpo}
C.\,C.\,Chen and N.\,F.\,Quimpo:
On strongly Hamiltonian abelian group graphs,
in K.\,L.\,McAvaney, ed., \emph{Combinatorial Mathematics, VIII \textup(Geelong, 1980\textup)}, pp.~23--34.
Lecture Notes in Math., 884, Springer, Berlin-New York, 1981. 
\MR{0641233}%
\AllowBreakBeforeURL{http://dx.doi.org/10.1007/BFb0091805}


\bibitem{CurranMorrisMorris-16p}
S.\,J.\,Curran, D.\,W.\,Morris, and J.\,Morris:
Cayley graphs of order $16p$ are hamiltonian,
\emph{Ars Math. Contemp.} 5 (2012) 185--211. 
\MR{2912833}%
\AllowBreakBeforeURL{http://amc-journal.eu/index.php/amc/article/view/207}

\bibitem{Durnberger-semiprod}
E.\,Durnberger:
Connected Cayley graphs of semi-direct products of cyclic groups of prime order by abelian groups are Hamiltonian,
\emph{Discrete Math.} 46 (1983), no. 1, 55--68. 
\MR{0708162}%
\AllowBreakBeforeURL{http://dx.doi.org/10.1016/0012-365X(83)90270-4}

\bibitem{Durnberger-prime}
E.\,Durnberger:
Every connected Cayley graph of a group with prime order commutator group has a Hamilton cycle,
in:
B.\,Alspach and C.\,Godsil, eds.,
\emph{Cycles in Graphs \textup(Burnaby, B.C., 1982\/\textup),}
North-Holland, Amsterdam, 1985, 
pp.~75--80.
\MR{0821506}%
\AllowBreakBeforeURL{http://dx.doi.org/10.1016/S0304-0208(08)72997-9}

\bibitem{GhaderpourMorris-27p}
E.\,Ghaderpour and D.\,W.\,Morris:
Cayley graphs of order $27p$ are hamiltonian, 
\emph{Internat. J. Comb.} 2011 (2011), Article ID 206930, 16 pages. 
\MR{2822405}%
\AllowBreakBeforeURL{http://dx.doi.org/10.1155/2011/206930}

\bibitem{GhaderpourMorris-30p}
E.\,Ghaderpour and D.\,W.\,Morris:
Cayley graphs of order $30p$ are hamiltonian,
\emph{Discrete Math.} 312 (2012) 3614--3625. 
\MR{297949}%
\AllowBreakBeforeURL{http://dx.doi.org/10.1016/j.disc.2012.08.017}

\bibitem{GhaderpourMorris-Nilpotent}
E.\,Ghaderpour and D.\,W.\,Morris:
Cayley graphs on nilpotent groups with cyclic
commutator subgroup are hamiltonian,
 \emph{Ars Math. Contemp.} 7 (2014), no.~1, 55--72. 
 \MR{3029452}%
\AllowBreakBeforeURL{http://amc-journal.eu/index.php/amc/article/view/280}

\bibitem{GodsilRoyle}
C.\,Godsil and G.\,Royle:
\emph{Algebraic Graph Theory}.
Springer-Verlag, New York, 2001. 
ISBN: 0-387-95241-1,
\MR{1829620}%
\AllowBreakBeforeURL{http://dx.doi.org/10.1007/978-1-4613-0163-9}


\bibitem{GrossTucker}
J.\,L.\,Gross and T.\,W.\,Tucke:
\emph{Topological Graph Theory.}
John Wiley \& Sons, Inc., New York, 1987. ISBN: 0-471-04926-3,
 \MR{0898434}



\bibitem{KeatingWitte} 
K.\,Keating and D.\,Witte: 
On Hamilton cycles in Cayley graphs with cyclic commutator subgroup,
in: 
B.\,R.\,Alspach and C.\,D.\,Godsil, eds.,
\emph{Cycles in Graphs \textup(Burnaby, B.C., 1982\textup)},
North-Holland, Amsterdam, 1985,
pp.~89--102.
ISBN 0-444-87803-3,
\MR{0821508}%
\AllowBreakBeforeURL{http://dx.doi.org/10.1016/S0304-0208(08)72999-2}

\bibitem{M2Slovenian-LowOrder}
K.\,Kutnar, D.\,Maru\v si\v c, J.\,Morris, D.\,W.\,Morris, and P.\,\v Sparl: 
Hamiltonian cycles in Cayley graphs whose order has few prime factors, 
\emph{Ars Math. Contemp.} 5 (2012), no.~1, 27--71.
\MR{2853700}%
\AllowBreakBeforeURL{http://amc-journal.eu/index.php/amc/article/view/177}

\bibitem{Marusic-HamCircCay}
D.\,Maru\v si\v c:
Hamiltonian circuits in Cayley graphs, \emph{Discrete Math.} 46 (1983) 49--54.
\MR{0708161}%
\AllowBreakBeforeURL{http://dx.doi.org/10.1016/0012-365X(83)90269-8}

\bibitem{Morris-OddPQCommHam}
D.\,W.\,Morris:
Odd-order Cayley graphs with commutator subgroup of order $pq$ are hamiltonian, \emph{Ars Math. Contemp.} 8 (2015) 1--28.%
\AllowBreakBeforeURL{http://amc-journal.eu/index.php/amc/article/view/330}


\bibitem{Schenkman}
 E.~Schenkman:
 \emph{Group Theory}.
 Robert E. Krieger Publishing Co., Huntington, N.Y., 1975. 
 ISBN: 0-88275-070-4,
 \MR{0460422}


\bibitem{WitteGallian-survey}
D.\,Witte and J.\,A.\,Gallian:
 A survey: Hamiltonian cycles in Cayley graphs, 
 \emph{Discrete Math.} 51 (1984) 293--304.
\MR{0762322}%
\AllowBreakBeforeURL{http://dx.doi.org/10.1016/0012-365X(84)90010-4}

\end{thebibliography}

\begin{thebibliography}{A}
\normalsize

\bibitem[A]{AlspachGPHam}
B.~Alspach,
The classification of Hamiltonian generalized Petersen graphs. J. Combin. Theory Ser. B 34 (1983), no. 3, 293–312. \MR{0714452}%
\AllowBreakBeforeURL{http://dx.doi.org/10.1016/0095-8956(83)90042-4}
\end{thebibliography}
\end{document}